%% file: main_revised.tex
\documentclass[12pt]{amsart}

\input{Preambolo}

\numberwithin{equation}{section}

\title[Metric extrapolation in the Wasserstein space]{Metric extrapolation in the Wasserstein space}

\author[T.\ O. Gallou\"et]{Thomas O.\ Gallou\"et}
\address{Thomas O. Gallou\"et, Université Paris-Saclay, Inria, CNRS, UMR 8628 - Laboratoire de mathématiques d'Orsay, ParMA, 91405 Orsay, France} 
\email{thomas.gallouet@inria.fr}

\author[A. Natale]{Andrea Natale}
\address{Andrea Natale, Université de Lille, Inria,  CNRS, UMR 8524 - Laboratoire Paul Painlevé, 59000 Lille, France} 

\email{andrea.natale@inria.fr}

\author[G. Todeschi]{Gabriele Todeschi}
\address{Gabriele Todeschi, LIGM, Univ. Gustave Eiffel, CNRS, F-77454 Marne-la-Vallée, France
} 
\email{gabriele.todeschi@univ-eiffel.fr}

\date{}

\begin{document}

\maketitle
\begin{abstract} 
In this article we study a variational problem providing a way to extend for all times minimizing geodesics connecting two given probability measures, in the Wasserstein space. This is simply obtained by allowing for negative coefficients in the classical variational characterization of Wasserstein barycenters. 
We show that this problem admits two equivalent convex formulations: the first can be seen as a particular instance of Toland duality and the second is a barycentric optimal transport problem.
We propose an efficient numerical scheme to solve the latter formulation based on entropic regularization and a variant of Sinkhorn algorithm.
\end{abstract}

\section{Introduction}

Given a metric space $(X,d)$, a globally minimizing geodesic on $X$ defined on $[t_0,t_1]\subset \mathbb{R}$, with $t_1>t_0$, is a curve $x:s\in[t_0,t_1] \rightarrow x(s)\in X$ verifying
\begin{equation}\label{eq:geodesicd}
d(x(s_0),x(s_1)) = \frac{|s_1-s_0|}{|t_1-t_0|} d(x(t_0),x(t_1))\,, 
 \end{equation}  
 for all $s_0,s_1\in [t_0,t_1]$.
If such a curve exists, its points can be characterized as the minimizers of the following variational problem:
for any $s\in [t_0,t_1]$
 \begin{equation}\label{eq:diffd0}
 x(s)\in \argmin_{x\in X} \left\{ (s-t_0) \frac{d^2(x,x(t_1))}{2} +(t_1-s) \frac{d^2(x,x(t_0))}{2}\right\}.
 \end{equation}

Let us consider problem \eqref{eq:diffd0} for $s \notin [t_0,t_1]$. Specifically, after a rescaling, we consider the following problem: 
given two points $x_0,x_1 \in X$ and $t>1$,  find $x\in X$ that solves 
 \begin{equation}\label{eq:diffd}
 \inf_{x\in X} \left\{ \frac{d^2(x,x_1)}{2(t-1)} - \frac{d^2(x,x_0)}{2t}\right\}.
 \end{equation}
 By triangular and Young's inequality, we always have that such infimum is larger than $-d^2(x_0,x_1)/2$. Moreover, if there exists a globally minimizing geodesic  $x:s\in[0,t] \rightarrow x(s)\in X$ such that $x(0)=x_0$ and $x(1)=x_1$ then this lower bound is attained by $x(t)$, which is therefore a minimizer. In particular, in the case where $X =\mathbb{R}^d$, equipped with the Euclidean distance, one always has a unique minimizer
 \[
 x(t) = x_0 + t (x_1-x_0)\,.
 \]
 In general, problem \eqref{eq:diffd} gives a variational definition of geodesic extrapolation even if no geodesic, connecting $x_0$ to $x_1$ on the interval $[0,1]$, may be extended up to time $t>1$ while staying globally length minimizing.

In this work, we consider a specific instance of problem \eqref{eq:diffd} where $(X,d)$ is $\mathcal{P}_2(\mathbb{R}^d)$, the set of probability measures with finite second moments, equipped with the $L^2$-Wasserstein distance $W_2$. The latter is defined as follows: for any $\mu,\nu\in \mathcal{P}_2(\mathbb{R}^d)$
\begin{equation}\label{eq:w2}
W^2_2(\mu,\nu) \coloneqq {\min}\left\{ \int |x-y|^2 \mathrm{d} \gamma(x,y) ~;~ \gamma \in \Gamma(\mu,\nu) \right\}\,,
\end{equation}
where $\Gamma(\mu,\nu) \subset \mathcal{P}_2(\mathbb{R}^d\times \mathbb{R}^d)$ is the set of couplings having $\mu$ and $\nu$ as first and second marginal, respectively. The problem we consider is therefore given by:
\begin{equation}\label{eq:diffW2}\tag{$\mathcal{P}$}
{\min_{\mu\in \mathcal{P}_2(\mathbb{R}^d)}}  \left\{
 \frac{W^2_2(\mu,\nu_1)}{2(t-1)} - \frac{W^2_2(\mu,\nu_0)}{2t} \right\} \,.
\end{equation}
By the derivation above, problem \eqref{eq:diffW2} provides a natural notion for geodesic extrapolation in the Wasserstein space. This was named \emph{metric extrapolation} in \cite{gallouet2024geodesic} and used for the construction of higher order time discretization of Wasserstein gradient flows. 
In contrast to the classical Wasserstein version of problem \eqref{eq:diffd0} introduced in \cite{agueh2011barycenters}, it was shown in \cite{gallouet2024geodesic} (see also Section \ref{sec:existence}) that problem \eqref{eq:diffW2} always admits a unique minimizer even when there are multiple geodesics connecting $\nu_0$ to $\nu_1$. This may be surprising given the lack of convexity of the problem with respect to the linear interpolation of probability measures, but is due to the fact that \eqref{eq:diffW2} is strongly convex along particular curves known as generalized geodesics \cite{ambrosio2008gradient}. The goal of this paper is to provide a full characterization of the hidden convexity in problem \eqref{eq:diffd0} which is responsible for its well-posedness.
More precisely, we show that it admits two different convex formulations, which can be seen as an instance of Toland duality and of weak optimal transport, respectively, and which we describe in general terms in this section. We will use such formulations to give a precise characterization of the minimizers of problem \eqref{eq:diffd0} and to construct an efficient numerical scheme to compute them.

\subsection{Geodesics in the Wasserstein space}
Consider the optimal transport problem \eqref{eq:w2} from $\nu_0$ to $\nu_1$. If, for example, $\nu_0$ is absolutely continuous, then Brenier's theorem states that there exists a unique solution $\gamma^* \in \Gamma(\nu_0,\nu_1$) to this problem which is furthermore induced by a transport map $\nabla u:\mathbb{R}^d\rightarrow \mathbb{R}^d$, where $u:\mathbb{R}^d\rightarrow \mathbb{R}$ is a convex function usually called the Brenier potential for the transport from $\nu_0$ to $\nu_1$, i.e.
 \[
\gamma^* = (\mathrm{Id},\nabla u)_\# \nu_0\,.
\]
Furthermore, in this case, there exists a uniquely defined geodesic on the interval $[0,1]$ connecting $\nu_0$ to $\nu_1$, which is given by McCann's interpolant \cite{mccann1997convexity}
 \begin{equation}\label{eq:geod}
 \nu(s)  = ((1-s)\mathrm{Id} +s\nabla u )_\# \nu_0\,, \quad \forall \,s\in[0,1]\,.
 \end{equation}
This curve can be extended up to $s=t>1$ while staying a length-minimizing geodesic if and only if $((1-t)\mathrm{Id} +t\nabla u )$ is an optimal transport map, or equivalently 
\[
 x\mapsto u(x) - \frac{t-1}{t} \frac{|x|^2}{2} \quad \text{is convex}.
\]
However, this condition is not satisfied in general, since we may only expect $u$ to be convex (and not strongly convex), which corresponds to the fact that the trajectories of the particles, induced by \eqref{eq:geod}, may cross precisely at $t=1$.

Note that in the following, for brevity, we will occasionally omit to specify ``length-minimizing" when referring to Wasserstein geodesics, but we will always refer to curves that are globally length-minimizing.

\subsection{Toland's duality}
Toland's duality \cite{toland1978duality,toland1979duality} concerns the minimization of the difference of two convex, proper and lower semi-continuous functions $F, G:V \rightarrow (-\infty, \infty]$, where $V$ is a normed vector space. Specifically, we have 
the equivalence
\[
\inf_{x \in V} \left\{ F(x)-G(x) \right\} = \inf_{p\in V^*} \left\{G^*(p)-F^*(p)\right\}\,,
\]
where $V^*$ is the topological dual of $V$ and $F^*$ and $G^*$ are the Legendre transforms of $F$ and $G$ respectively. 

The idea of using Toland duality to deal with differences of Wasserstein distances stems from the work of Carlier \cite{carlier2008Toland}, who considered a variant of problem \eqref{eq:diffW2} where the coefficients multiplying the two distances are equal and $\mathbb{R}^d$ is replaced by a compact convex set. In our case, $F$ and $G$ are replaced by the maps
\[
\mu \in \mathcal{P}_2(\mathbb{R}^d) \mapsto \frac{W^2_2(\mu,\nu_1)}{2(t-1)} \quad \text{and} \quad
\mu \in \mathcal{P}_2(\mathbb{R}^d) \mapsto \frac{W^2_2(\mu,\nu_0)}{2t} \,.
\]
Then, at least formally, one can check that the resulting dual problem is given by 
\begin{equation}\label{eq:dualBrenier} \tag{$\mathcal{P}^*$}
\inf \left \{ \int u^* \mathrm{d} \nu_1  + \int u \mathrm{d} \nu_0 ~:~ u - \frac{t-1}{t} \frac{|\cdot|^2}{2} \text{ is convex {and l.s.c.}} \right\}.
\end{equation}
{A detailed proof of this result is provided in Section \ref{sec:toland}.
Remarkably, and differently from \eqref{eq:diffW2}, this is a convex optimization problem in the usual sense. 
Also, we observe that requiring $u$ to be only convex rather than strongly convex,  one recovers one of the possible dual formulations of the quadratic optimal transport problem from $\nu_0$ to $\nu_1$. In particular, in this case, any Brenier potential $u$ is a solution and, if $\nu_1$ is a.c., $\nabla u^*_\#\nu_1 = \nu_0$.  
On the other hand, if $u$ is a solution to \eqref{eq:dualBrenier} we show  that $\overline{\nu}_0 \coloneqq \nabla u^*_\# \nu_1$ may be different from $\nu_0$ but is dominated by it in convex order (see Lemma \ref{lem:optimalityu}), i.e.\
\begin{equation}\label{eq:convexorder}
\int \varphi\,  \mathrm{d} \overline{\nu}_0 \leq \int \varphi \,\mathrm{d} {\nu}_0
\end{equation}
for all convex functions $\varphi:\mathbb{R}^d\rightarrow \mathbb{R}$. Moreover, we show that the (unique) solution to problem \eqref{eq:diffW2} can be written as
\begin{equation}\label{eq:nut}
\nu_t \coloneqq ((1-t) \mathrm{Id} +t \nabla u)_\# \overline{\nu}_0\,,
\end{equation}
at least in the case where $\overline{\nu}_0$ is absolutely continuous (see Theorem \ref{th:dualityW2} for the precise statement in the general case). 
Since $u$ is strongly convex this means that the measure $\overline{\nu}_0$, defined via the solution of problem \eqref{eq:dualBrenier} is such that the geodesic from $\overline{\nu}_0$ to $\nu_1$ on the time interval $[0,1]$ can be extended up to time $t$ and the resulting extension is precisely the solution to \eqref{eq:diffW2}.

\subsection{Weak optimal transport formulation} The convex order relation between $\overline{\nu}_0 = \nabla u^*_\# \nu_1$ and $\nu_0$, with $u$ solving $\eqref{eq:dualBrenier}$, can be exploited to derive an equivalent formulation of problem \eqref{eq:diffW2} which fits in the framework of weak (and in particular barycentric) optimal transport, a generalization of optimal transport introduced in \cite{gozlan2017kantorovich}. This reads as follows: 
\begin{equation} \label{eq:barycentric}\tag{$\mathcal{B}$}
\inf_{\gamma\in \Gamma(\nu_0,\nu_1)} \int |t x_1 - (t-1) \mathrm{bary}(\gamma_{x_1})|^2 \mathrm{d} \nu_1(x_1)\,,\quad  \mathrm{bary}(\gamma_{x_1}) = \int x_0 \mathrm{d} \gamma_{x_1}(x_0)\,,
\end{equation}
where we denote by $\ed \gamma (x_0,x_1) = \ed\gamma_{x_1}(x_0) \ed \nu_1(x_1)$ the disintegration of $\gamma$ with respect to its second marginal $\nu_1$. 
On the other hand, by Strassen's theorem (see Lemma \ref{lem:strassen} and \cite{strassen1965existence}), {the convex order} condition \eqref{eq:convexorder} implies the existence of a coupling $\theta \in \Gamma(\overline {\nu}_0,\nu_0)$ which is the law of a martingale, i.e.\
$ \mathrm{d}\theta(x,y) =\mathrm{d} \theta_{x}(y) \mathrm{d} \overline{\nu}_0(x)$ and
\[
\int y \,\mathrm{d} \theta_{x}(y) = x\,,  \quad \text{ for $\overline{\nu}_0$-a.e. } x\in\mathbb{R}^d,
\]
where now $\theta_x$ is the probability kernel obtained by disintegrating  $\theta$ with respect to its first marginal $\overline{\nu}_0$.
Thus, formally, the link between problem \eqref{eq:dualBrenier} and \eqref{eq:barycentric} can be stated as follows: $u$ solves \eqref{eq:dualBrenier} if and only if the coupling $\pi \in \Gamma(\nu_0,\nu_1)$ defined by
\begin{equation}\label{eq:barymin}
\mathrm{d} \pi(x_0,x_1) =  \mathrm{d} \theta_{\nabla u^*(x_1)}(x_0) \mathrm{d} {\nu}_1(x_1)
\end{equation}
solves \eqref{eq:barycentric} (see Theorem \ref{th:barycentric}). Note that this gives a characterization of the minimizers of problem \eqref{eq:barycentric} as the composition of a martingale and a sufficiently smooth transport plan. Such a characterization can also be derived as a slight modification of a result {of Gozlan and Juillet} \cite{gozlan2020mixture}. Our proof shows that this can be alternatively obtained as a consequence of Strassen's theorem and Toland duality.

\subsection{Related works} 
The problem of defining Wasserstein geodesic extensions has appeared in the literature in different contexts.
Several approaches, different from the one discussed in this article, were proposed
in \cite{benamou2019second} as a byproduct of different definitions for cubic splines in the Wasserstein space, with the aim of interpolating (and extrapolating) multiple measures by a single curve.
A variant of problem \eqref{eq:diffW2}
was introduced in \cite{Matthes2019bdf2}
in order to construct higher order versions of the JKO scheme for the numerical computation of Wasserstein gradient flows. 
The idea of using geodesic extensions is also proposed in \cite{Legendre2017VIM} with the same purpose. 

Problem \eqref{eq:diffW2} has also a strong link with fluid dynamics. In one dimension,  the curve $t\mapsto \nu_t$ defined by equation \eqref{eq:nut} yields a ``sticky solution" of the pressureless Euler equation \cite{brenier1998one,natile2009wasserstein}; see also Section \ref{sec:1d}. In general, the idea of modifying the initial conditions to retrieve the solution of a fluid dynamic model  at the final time via convex optimization (as it occurs in our model; see equation \eqref{eq:nut}) was proposed by Brenier in \cite{brenier2018initial} for the incompressible Euler equations, but it applies to a large class of models \cite{brenier2020examples,vorotnikov2022partial}.

From a computational perspective, the use of strongly convex Brenier potentials has been advocated in \cite{paty2020regularity} to regularize standard optimal transport, but via a variational formulation that differs from ours and in particular is non-convex. On the other hand, problem \eqref{eq:dualBrenier}, with additional constraint of $L$-smoothness on the potential, has been recently considered in \cite{vacher2023semi} to provide an estimator of the true Brenier potential between sampled measures. In our case, however, the equivalence with the barycentric problem \eqref{eq:barycentric} enables us to propose an algorithm with similar performance as the well-known Sinkhorn method. Furthermore, the numerical scheme that we propose leverages Sinkhorn and entropic regularization in a different way than related methods introduced in \cite{paty2022algorithms} for general weak optimal transport problems, or more recently in \cite{kim2024statistical}. 

The variational formulation of the Wasserstein  barycenters between multiple measures with negative coefficients appears naturally also in the context of nonlinear reduced basis models in the Wasserstein space, as in \cite{dalery2023nonlinear} for example, and for statistical regression \cite{fan2021conditional}.
However, in these works the non-convexity of the problem is not addressed directly. In particular, in \cite{dalery2023nonlinear} the authors use a simpler proxy for the $W_2$ distance which allows for an explicit characterization of the barycenters. In \cite{fan2021conditional} instead, the problem is tackled numerically using a convex optimization approach, without convergence guarantees.
{We refer to Remark \ref{rem:mm} for some considerations on the extension of our work to the setting with multiple measures (see also the recent work \cite{tornabene2024generalized} for a detailed study of the one-dimensional case).} 
We also remark that a weak optimal transport version of Wasserstein barycenters (rather than extrapolation as in our case) has been recently introduced in \cite{cazelles2021novel}.

Finally, after this work first appeared, problem \eqref{eq:diffW2} was also studied in \cite{bertucci2024approximation} as a means of approximating $W^2_2(\nu_0,\nu_1)$, since, up to the multiplicative factor $-1/2$,  this is precisely the value of \eqref{eq:diffW2} when a geodesic extension exists. This approximation was then used for the study of Hamilton-Jacobi equations on the space of probability measures. The duality theory developed here, along with its geometric interpretation and its connection to weak optimal transport, offers a complementary perspective to that developed in \cite{bertucci2024approximation}, which is rather focused on probabilistic considerations.

\subsection{Structure of the paper} The rest of the article is structured as follows. In Section \ref{sec:main} we collect some well-posedness results for the metric extrapolation problem, and prove strong duality with problem \eqref{eq:dualBrenier}.  In Section \ref{sec:weak}, we prove the equivalence with the weak optimal transport formulation \eqref{eq:barycentric}. In Section \ref{sec:relation} we give a precise link of our problem with the $H^1$ projection on convex functions, via a $\Gamma$-convergence result. In Section \ref{sec:particular} we provide some exact solutions to the problem providing some intuition on the behavior of the metric extrapolation. Finally, in Section \ref{sec:numerical}, we propose and analyze a variant of Sinkhorn algorithm to solve the problem in the case of atomic measures, and provide some numerical results.

\section{Main properties and Toland duality}\label{sec:main}

In this section we collect some well-posedness results for problem \eqref{eq:diffW2}, i.e., existence and uniqueness of solutions and convexity of the problem along generalized geodesics, that were already proven in \cite{gallouet2024geodesic}, and also stability of the minimizers with respect to the data of the problem.
We further show that the problem admits a convex dual formulation given by problem \eqref{eq:dualBrenier}. This last result is inspired by the duality principle studied in \cite{carlier2008Toland} for the case where the transport costs are multiplied by the same coefficient with opposite signs, but it does not contain the latter since our proof  requires the coefficients to be different in absolute value. In this section, as in the rest of the paper, we always suppose $t>1$.

\subsection{Existence and uniqueness of solutions}\label{sec:existence} Existence and uniqueness for solutions to problem \eqref{eq:diffW2} follows from the fact that the functional 
\begin{equation}\label{eq:diffW2fun}
\mc{F}_t(\nu_0,\nu_1;\mu) \coloneqq \frac{W^2_2(\mu,\nu_1)}{2(t-1)} -\frac{W^2_2(\mu,\nu_0)}{2t}
\end{equation}
is bounded from below and is strongly convex along specific curves, known as generalized geodesics \cite[Definition 9.1.1]{ambrosio2008gradient}. As for the lower bound, this can be found directly using the notion of Wasserstein geodesic. In fact, if $\nu:s\in[0,t]\rightarrow \nu(s) \in \mc{P}_2(\mathbb{R}^d)$ is a geodesic from $\nu_0$ to $\mu$, we have
\begin{equation}\label{eq:boundbeta}
\begin{aligned}
\frac{W_2^2(\mu,\nu_0)}{2t} &= \frac{W^2_2(\mu,\nu(1))}{2(t-1)} + \frac{W^2_2(\nu(1),\nu_0)}{2}\\
 &\leq \frac{W^2_2(\mu,\nu_1)}{2(t-1)} + \frac{W^2_2(\nu_1,\nu_0)}{2}\,,
\end{aligned}
\end{equation}
where we have used the variational characterization of globally minimizing geodesics in equation \eqref{eq:diffd0}, and specifically the fact that
\[
\nu(1) \in \underset{\nu \in \mc{P}_2(\mathbb{R}^d)}{\mathrm{argmin}} \left\{ \frac{W^2_2(\mu,\nu)}{2} + (t-1) \frac{W^2_2(\nu,\nu_0)}{2}\right\}\,.
\]
Then, equation \eqref{eq:boundbeta} yields
\begin{equation}\label{eq:lowerbound}
\mc{F}_t(\nu_0,\nu_1;\mu) \geq - \frac{W_2^2(\nu_0,\nu_1)}{2}\,.
\end{equation}

A precise statement for the strong convexity of $\mc{F}_t(\nu_0,\nu_1; \cdot)$ along generalized geodesics is contained in the following lemma:
\begin{lemma}[Theorem 3.4 in \cite{Matthes2019bdf2}]\label{lem:abconvex}
	Let $\nu_0,\nu_1 \in \mc{P}_2(\mathbb{R}^d)$, and consider the functional $\mc{F}_t(\nu_0,\nu_1; \cdot):\mathcal{P}_2(\mathbb{R}^d) \rightarrow\mathbb{R}$ defined in \eqref{eq:diffW2fun}. For any $\mu_0, \mu_1 \in \mc{P}_2(\mathbb{R}^d)$, let  $\mu:[0,1]\rightarrow \Pc_2(\mathbb{R}^d), \mu(0)=\mu_0, \mu(1)=\mu_1$, be a generalized geodesic with base $\nu_1$. Then, for all $0\leq s\leq 1$ it holds
	\begin{equation}\label{eq:abconvex}
		\mc{F}_t(\nu_0,\nu_1; \mu(s)) \le (1-s) \mc{F}_t(\nu_0,\nu_1; \mu_0) + s \mc{F}_t(\nu_0,\nu_1; \mu_1)- \frac{s(1-s)}{2}\frac{W^2_2(\mu_0,\mu_1)}{t(t-1)}.
	\end{equation}
\end{lemma}
Note that this Lemma can be viewed as a simple consequence of the fact that $\mc{F}_t(\nu_0,\nu_1; \cdot)$ is the sum of $W^2_2(\cdot,\nu_1)/(2t-2)$, which is $1/(t-1)$-convex along generalized geodesics with base $\nu_1$ \cite[Lemma 9.2.1]{ambrosio2008gradient} (in the sense of equation \eqref{eq:abconvex}), and $- W^2_2(\cdot,\nu_1)/(2t)$, which is $-1/t$-convex along \emph{any} generalized geodesic \cite[Proposition 9.3.12]{ambrosio2008gradient}. Therefore, the sum must be $(1/(t-1)-1/t)$-convex along generalized geodesics with base $\nu_1$.

Using the fact that the space $\mc{P}_2(\mathbb{R}^d)$ equipped with the Wasserstein distance is complete, one can show that Lemma \ref{lem:abconvex} together with the lower bound \eqref{eq:lowerbound} implies that problem \eqref{eq:diffW2} admits a unique solution. We refer to \cite[Proposition 4.10]{gallouet2024geodesic} for the proof.

\begin{remark}[Link with extrapolation] \label{rem:extrap} Suppose that there exists a length-minimizing geodesic $s\in [0,t]  \rightarrow \nu(s) \in \mc{P}_2(\mathbb{R}^d)$ such that $\nu(0) = \nu_0$ and $\nu(1)= \nu_1$, then $\nu(t)$ is the unique minimizer to problem \eqref{eq:diffW2}, since
\[
\frac{W^2_2(\nu(t),\nu_1)}{2(t-1)} - \frac{W^2_2(\nu(t),\nu_0)}{2t} = - \frac{ W_2^2(\nu_0,\nu_1)}{2}\,.
\]
The converse is also true: if the unique minimizer $\mu$ realizes the lower bound, i.e.\ $\mc{F}_t(\nu_0,\nu_1;\mu) = -  W_2^2(\nu_0,\nu_1)/2$, then the geodesic $\nu$ is well-defined and $\mu = \nu(t)$. In fact, equation \eqref{eq:boundbeta} holds in this case with an equality instead of an inequality, which implies that $\nu_1$ belongs to a globally minimizing geodesic connecting $\nu_0$ and $\mu$.
\end{remark} 

As a direct consequence of Lemma \ref{lem:abconvex}, if $\nu_t$ minimizes the functional $\mc{F}_t(\nu_0, \nu_1; \cdot)$, then we have that
\begin{equation}\label{eq:strongc}
\frac{ W^2_2(\mu,\nu_t)}{2 t(t-1)} + \mc{F}_t(\nu_0, \nu_1;\nu_t) \leq  \mc{F}_t(\nu_0, \nu_1;\mu) 
\end{equation}
for any $\mu \in \mc{P}_2(\mathbb{R}^d)$, and moreover, 
\[
W_2(\nu_0,\nu_t) \leq t W_2(\nu_0,\nu_1)\,, \quad W_2(\nu_1,\nu_t) \leq (t-1) W_2(\nu_0,\nu_1)\,,
\]
in analogy to Wasserstein geodesics (for which equality holds);
see \cite[Proposition 4.10]{gallouet2024geodesic} for details. In turn, these bounds imply that the map $(t,\nu_0,\nu_1)\mapsto \nu_t$ is at least locally Lipschitz in $t$ and locally 1/2-H\"older continuous in $\nu_0$ and $\nu_1$ (with respect to the $W_2$ metric). We sketch below a simple argument to establish this, although a more refined approach is necessary to obtain a sharper result, which we leave for future work.
\begin{lemma}[Stability of the extrapolation]\label{lem:stability}
Let $\nu_t$ and $\tilde{\nu}_s$ be the unique minimizers of  $\mc{F}_t(\nu_0,\nu_1; \cdot)$ and $\mc{F}_s(\tilde{\nu}_0,\tilde{\nu}_1; \cdot)$, respectively, where $\nu_i, \tilde{\nu}_i \in \mc{P}_2(\mathbb{R}^d)$ for $i=0,1$, and for all $t,s>1$. Then,
\begin{equation}\label{eq:dissipation}
 W^2_2(\nu_t,\tilde{\nu}_s) \leq C  \left( |t-s|^2 + W_2(\nu_0,\tilde{\nu}_0)+ W_2(\nu_1,\tilde{\nu}_1) \right)\,,
\end{equation}
where $C>0$ depends polynomially only on $\max(t,s)$,  $W_2(\nu_0,\tilde{\nu}_0)$, $W_2(\nu_1,\tilde{\nu}_1)$, and $W_2(\nu_0,\nu_1)$.
\end{lemma}\begin{proof} For simplicity, let us use $\mc{G}_t$ and $\tilde{\mc{G}}_s$ to denote $t(t-1)\mc{F}_t(\nu_0,\nu_1;\cdot)$ and $s(s-1)\mc{F}_s(\tilde{\nu}_0,\tilde{\nu}_1;\cdot)$, respectively.
    Setting $\mu = \tilde{\nu}_s$ in \eqref{eq:strongc} and $\mu = \nu_t$ in the same inequality written for $\mc{F}_s(\tilde{\nu}_0,\tilde{\nu}_1;\cdot)$, we obtain
   \begin{equation}\label{eq:stab_est}
    \begin{aligned}
    W^2_2(\tilde{\nu}_s,\nu_t)&\leq  \mc{G}_t(\tilde{\nu}_s) - \tilde{\mc{G}}_s(\tilde{\nu}_s)+\tilde{\mc{G}}_s(\nu_t)-\mc{G}_t(\nu_t)\,.
    \end{aligned}
    \end{equation}
    If $\tilde{\nu}_i = \nu_i$ for $i=0,1$, then $\tilde{\nu}_s = \nu_s$ and the right-hand side is equal to 
    \[
    \frac{1}{2}\left(t- s\right)(W^2_2({\nu}_s,\nu_1) - W^2_2({\nu}_t,\nu_1)) + \frac{1}{2}\left(t-s\right)( W^2_2({\nu}_t,\nu_0)-W^2_2({\nu}_s,\nu_0))\,,
    \]
    and we have, for example,
    \begin{equation}\label{eq:diffbound2}
    W^2_2({\nu}_s,\nu_1) - W^2_2({\nu}_t,\nu_1)\leq W_2({\nu}_s,\nu_t) (s+t-2)W_2(\nu_1,\nu_0)\,,
    \end{equation}
    where we used the triangular inequality and equation \eqref{eq:dissipation}. This gives
    \[
    W_2({\nu}_s,\nu_t) \leq (s+t-1)W_2(\nu_1,\nu_0) |t-s|\,.
    \]
    Similarly, if $s=t$, the right-hand side of \eqref{eq:stab_est} is equal to
    \begin{multline*}
    \frac{t}{2}(W_2^2(\tilde{\nu}_t,\nu_1) - W^2_2(\tilde{\nu_t},\tilde{\nu_1}))+\frac{t}{2}(W^2_2(\nu_t,\tilde{\nu_1}) -W^2_2(\nu_t,\nu_1))\\
    - \frac{t-1}{2}(W_2^2(\tilde{\nu}_t,\nu_0) - W^2_2(\tilde{\nu_t},\tilde{\nu_0})) -\frac{t-1}{2}(W^2_2(\nu_t,\tilde{\nu_0}) -W^2_2(\nu_t,\nu_0))\,.
    \end{multline*}
    and we can control each of the four terms as in equation \eqref{eq:diffbound2}, leading to
    \[
    W_2^2(\nu_t,\tilde{\nu}_t) \leq C (W_2(\nu_1,\tilde{\nu}_1) + W_2(\nu_0,\tilde{\nu}_0))
    \]
    where $C>0$ is as in the statement of the lemma. We conclude by a further triangular inequality $W_2(\nu_t,\tilde{\nu}_s) \leq W_2(\nu_t,{\nu}_s)+W_2(\nu_s,\tilde{\nu}_s)$.
    \end{proof}

\subsection{Dual formulation}\label{sec:toland} Let us recall the classical dual Kantorovich formulation of the $L^2$-optimal transport problem: for any $\mu,\nu \in \mc{P}_2(\mathbb{R}^d)$ and $s>0$
\begin{equation}\label{eq:dualKantorovich}
\frac{W^2_2(\mu,\nu)}{2s}= \sup\left\{ \int \phi\, \ed \mu +\int \phi^{c}_s \, \ed \nu  \right\},
\end{equation}
where the supremum is taken over any continuous bounded function $\phi \in C_b(\mathbb{R}^d)$ and $\phi^{c}_s\in C_b(\mathbb{R}^d)$ denotes the $c$-transform of $\phi$ associated with the  $l^2$ cost rescaled by the factor $2s$, i.e.
\[
\phi^{c}_s(y) \coloneqq \inf_x \frac{|x-y|^2}{2s} - \phi(x)\,.
\]
We will denote the $c$-transform with respect to the cost $c(x,y)=|x-y|^2/2$ by $\phi^c \coloneqq \phi^c_1$. In the following we will also use an alternative form of the dual formulation \eqref{eq:dualKantorovich} written in terms of the Brenier potential. More precisely
\begin{equation}\label{eq:dualKantorovichb}
\frac{W^2_2(\mu,\nu)}{2}= -\inf \left\{ \int u\, \ed \mu +\int u^* \, \ed \nu  \right\} + \int \frac{|x|^2}{2} \ed \mu(x) + \int \frac{|x|^2}{2} \ed \nu(x) \,,
\end{equation}
where the infimum is computed over all  convex functions $u:\mathbb{R}^d\rightarrow \mathbb{R}$ (or equivalently over all lower semi-continuous convex function $u:\mathbb{R}^d\rightarrow (-\infty,\infty]$). Note that this formulation can be obtained 
from \eqref{eq:dualKantorovich} via the change of variable $u = |\cdot|^2/2 - \phi$.
{The problem on the right-hand side of \eqref{eq:dualKantorovichb} always admits a lower semi-continuous solution (which is  not necessarily continuous). We refer to any such solution as an optimal Brenier potential for the transport from $\mu$ to $\nu$.} 

By the dual formulation in equation \eqref{eq:dualKantorovich}, we have 
\begin{equation}\label{eq:boundtoland}
\begin{aligned}
\eqref{eq:diffW2} &= \inf_{\mu, \varphi} \sup_{\psi} 
\left\{
\int (\psi - \varphi)\, \ed \mu +\int \psi^{c}_{t-1} \, \ed \nu_1 - \int \varphi^c_{t} \ed \nu_0 \right\}\\
& \geq \inf_{\varphi} \left\{\int \varphi^{c}_{t-1} \, \ed \nu_1 - \int \varphi^c_{t} \ed \nu_0 \right\},
\end{aligned}
\end{equation}
where the first equality is obtained by using the dual formulation for both Wasserstein distances, whereas the inequality follows by  selecting $\psi = \varphi$ in the $\sup$ over $\psi$. 
Now, let us  {observe that }
\begin{equation}\label{eq:phicc}
\begin{aligned}
-[\varphi^c_{t}]^c(x) & = -\inf_y \left\{\frac{|x-y|^2}{2} - \inf_{z} \left\{ \frac{|y-z|^2}{2t} - \varphi(z) \right\} \right\}\\
& \leq \inf_{z} \sup_y \left\{-\frac{|x-y|^2}{2} +\frac{|y-z|^2}{2t} - \varphi(z) \right\} = \varphi^c_{t-1}(x)
\end{aligned}
\end{equation}
and therefore we have 
\[
\eqref{eq:diffW2} \geq 
- \sup \left\{\int \phi^{c} \, \ed \nu_1 + \int \phi\, \ed \nu_0 \right\},
\]
where now $\phi \in C_b(\mathbb{R}^d)$,
and with the constraint that there must exist a function $\varphi \in C_b(\mathbb{R}^d)$ such that
\begin{equation}\label{eq:constraintc}
\phi(x) = \varphi_t^c(x)= \inf_y\left\{ \frac{|x-y|^2}{2t} - \varphi(y)\right\} .
\end{equation}
Note that without this last constraint, we have precisely the dual formulation of the $L^2$-optimal transport problem from $\nu_0$ to $\nu_1$ in equation \eqref{eq:dualKantorovich}, with $s=1$. 
Making the change of variable $u = |\cdot|^2/2 - \phi$, which would correspond to the Brenier potential from $\nu_0$ to $\nu_1$, the constraint \eqref{eq:constraintc} implies the existence of a function $\xi\in C(\mathbb{R}^d)$ given by $\xi = |\cdot|^2/2 - t\varphi$ such that
\[
u(x) = \frac{t-1}{t} \frac{|x|^2}{2} +\frac{\xi^*(x)}{t} \,,
\]
where $\xi^*$  is the convex function given by the Legendre transform of $\xi$. Hence, if we consider the following problem
\begin{equation}\label{eq:dualBrenier} \tag{$\mc{P}^*$}
\inf \left \{ \int u^* \ed \nu_1  + \int u \ed \nu_0 ~:~ u - \frac{t-1}{t} \frac{|\cdot|^2}{2} \text{ is convex {and l.s.c.}}\right\},
\end{equation}
where again $u^*$ is the Legendre transform of $u$, we obtain
\begin{equation}\label{eq:ineq}
\eqref{eq:diffW2} \geq  \eqref{eq:dualBrenier} - \int \frac{|x|^2}{2} \ed \nu_0(x) - \int \frac{|x|^2}{2} \ed \nu_1(x)\,.
\end{equation}

{
Problem \eqref{eq:dualBrenier} is a convex optimization problem. It appeared already in \cite{gozlan2020mixture}, up to slight variations, as the dual of a barycentric optimal transport problem.  For completeness, we provide a proof of existence of minimizers, which follows essentially the same lines as the one used in \cite{paty2020regularity} to show existence of minimizers for the $W_2$ projection on measures obtained via push-forward by smooth and strongly convex Brenier potentials. The key property we will use is the fact that a lower semi-continuous proper convex function $u:\mathbb{R}^d \rightarrow (-\infty,\infty]$ is $1/L$-strongly convex if and only if $\nabla u^*$ is $L$-Lipschitz; see, e.g., Theorem 18.15 in \cite{bauschke2017convex}.
}

\begin{lemma}\label{lem:existence} There exists a {lower semi-continuous} strongly convex function $u$ solving \eqref{eq:dualBrenier}.
\end{lemma}
\begin{proof} {Since $u^{**} = u$}, setting $v=u^*$ we obtain the equivalent problem
\begin{equation}\label{eq:dualBrenier2}
\inf \left \{ \int v \ed \nu_1  + \int v^* \ed \nu_0 ~:~ v\in C^1(\mathbb{R}^d) \text{ is convex} \,, ~ \nabla v \text{ is $L$-Lipschitz}  \right\}\,,
\end{equation}
where $L = t/(t-1)$. {Note that for any admissible $v$, the functional minimized in \eqref{eq:dualBrenier2} cannot take the value $-\infty$. In fact, on one hand $v^*$ is bounded from below being strongly convex, and on the other $v(y) \geq \langle y,x_0\rangle - v^*(x_0)$ for some $x_0\in \mathbb{R}^d$ for which $v^*(x_0)$ is finite and $\nu_0$ has finite first moment.} We prove the existence of a minimizer for problem \eqref{eq:dualBrenier2}. Let $(v_n)_n$ be a minimizing sequence. We start by showing that up to a subsequence extraction this converges pointwise to a function $\bar{v}$ which is admissible for problem \eqref{eq:dualBrenier2}. Since the problem is invariant under the addition of a constant to $v$, without loss of generality we can assume $v_n(0)=0$. Let $y_n = \nabla v_n(0)$, for any $p_n \in \partial v^*(y_n)$, the sub-differential of $v^*$ at $y_n$, and for all $y\in \mathbb{R}^d$, we have 
\begin{equation}\label{eq:quadbound}
v^*_n(y) \geq  v^*_n(y_n) + \langle p_n , y-y_n\rangle +\frac{1}{2L}|y-y_n|^2
\end{equation}
and we can pick $p_n=0$ since $y_n = \nabla v_n(0)$. Therefore we obtain, by Jensen's inequality,
\[
\begin{aligned}
\int v_n \ed \nu_1  + \int v^*_n \ed \nu_0 &\geq 
v_n(\mathrm{bary}(\nu_1)) + v_n^*(y_n) + \frac{1}{2L} \int |y-y_n|^2 \ed \nu_0\\
& \geq \langle \mathrm{bary}(\nu_1),y_n\rangle + \frac{1}{2L}\int |y-y_n|^2 \ed \nu_0 \,,
\end{aligned}
\]
where the second inequality follows from Fenchel–Young inequality, and where
\[
 \mathrm{bary}(\nu_1) \coloneqq \int x \ed \nu_1(x)\,.
\]
This implies that $y_n=\nabla v_n(0)$ is uniformly bounded, moreover for all $x\in\mathbb{R}^d$
\[
|\nabla v_n(x) | \leq |\nabla v_n(x) - \nabla v_n(0)| + |\nabla v_n(0)| \leq L|x| + |\nabla v_n(0)|\,.
\]
Therefore, by the Arzelà-Ascoli theorem $(\nabla v_n)_n$ converges uniformly on compact subsets (up to a subsequence extraction) to a function $g$ which is $L$-Lipschitz. Since $v_n(0)=0$, one can show that $v_n$ converges pointwise (again, up to a subsequence extraction) to a convex function $\bar{v}$ such that $\bar{v}(0) =0$  and that $g= \nabla \bar{v}$. Indeed, we have that for all $n,m\geq0$,
\[
|v_n(x) - v_m(x)| = \left|\int_0^1 \langle \nabla v_n(sx)-\nabla v_m(sx),x\rangle \ed s \right|  \leq  |x|\int_0^1 |\nabla v_n(sx)-\nabla v_m(sx)|\ed s
\]
and so the uniform convergence of $\nabla v_n$ on compact sets implies that $(v_n(x))_n$ is Cauchy for all $x\in\mathbb{R}^d$. Furthermore, since $\bar{v}(0)=0$, for all $x,y\in\mathbb{R}^d$
\begin{multline}
v_n(x) -v_n(y)= \int_0^1 \langle \nabla v_n(sy+(1-s)x),x-y\rangle \ed s \\ \implies \quad  \bar{v}(y) -\bar{v}(x)= \int_0^1 \langle  g(sy+(1-s)x),y-x\rangle \ed s\,,
\end{multline}
which in turn implies that $\bar{v}$ is differentiable and $g= \nabla \bar{v}$.

Now, observe that by the definition of Legendre transform and the fact that $v_n(0)=0$, we have $v^*_n\geq 0$. Therefore by Fatou's lemma, 
\[
\liminf_{n\rightarrow \infty } \int v^*_n \ed \nu_0 \geq \int \liminf_{n\rightarrow \infty} v^*_n \ed \nu_0 \geq \int  {\bar{v}}^* \ed \nu_0 \,.
\] 
On the other hand $v_n(x) - \langle y_n, x\rangle \geq 0$ for all $x\in\mathbb{R}^d$. Applying again Fatou's lemma with respect to the sequence of functions $x\mapsto v_n(x) - \langle y_n, x\rangle$, and using the fact that $y_n$ converges to $\nabla \bar{v}(0)$, we obtain
\[
\liminf_{n\rightarrow \infty } \int v_n \ed \nu_1 \geq  \int  \bar{v} \ed \nu_1.
\]
Therefore $\bar{v}$ solves \eqref{eq:dualBrenier2} and $\bar{v}^*$ solves \eqref{eq:dualBrenier}.
\end{proof}

Additionally, we have the following  characterization of the minimizers, which will be useuful to establish the strong duality between \eqref{eq:diffW2} and \eqref{eq:dualBrenier}.
{This is also contained in \cite{gozlan2020mixture}, but we include a sketch of the proof here for completeness}. 
Before stating the result, we remark that for any strongly convex function $u$, $\nabla u^*$ is Lipschitz continuous so $(\nabla u^*)_\# \nu_1$ is well-defined also when $\nu_1$ is not absolutely continuous. 

\begin{lemma}\label{lem:optimalityu} Define $\eta\coloneqq (\nabla u^*)_\# \nu_1 - \nu_0$. Then,  a $(t-1)/t$-convex {lower semi-continuous} function $u$ solves \eqref{eq:dualBrenier} if and only if
\begin{equation}\label{eq:optimalityu}
\int \left(u(x) - \frac{t-1}{t}\frac{|x|^2}{2}\right) \ed \eta(x) =0 \quad \text{and} \quad \int  \varphi(x)\, \ed \eta(x) \leq 0
\end{equation}
for all {lower semi-continuous convex functions $\varphi:\mathbb{R}^d\rightarrow (-\infty,\infty]$.}
\end{lemma}
\begin{proof}
Equation \eqref{eq:optimalityu} is equivalent to the optimality conditions of problem \eqref{eq:dualBrenier}. In fact, {if $u$ is a lower semi-continuous function solving \eqref{eq:dualBrenier}, given any lower semi-continuous function $\varphi$ such that $u+\varphi$ is $(t-1)/t$-convex, then 
$u + \varepsilon \varphi$ is $(t-1)/t$-convex and lower semi-continuous for all $0<\varepsilon \leq 1$ and it is therefore an admissible competitor for problem \eqref{eq:dualBrenier}. Hence we find that
\[
\int (u +\varepsilon \varphi )^* \ed \nu_1 + \int (u + \varepsilon \varphi ) \ed \nu_0 \geq \int u^* \ed \nu_1 + \int u \ed \nu_0 \,,
\]
or equivalently
\[
\int \frac{(u + \varepsilon \varphi )^*  - u^*}{\varepsilon} \ed \nu_1  \geq  -\int \varphi  \ed \nu_0 \,.
\]
By Lemma 6.1 in \cite{gozlan2020mixture}, generalizing a result of Gangbo \cite{gangbo1994elementary}, if there exist $a,b\geq 0$ such that $\varphi(x) \geq -a|x| -b$ for all $x\in\mathbb{R}^d$ and $\varphi$ is continuous, then the integrand on the left-hand side tends pointwise to $-\varphi \circ \nabla u^*$ as  $\varepsilon \rightarrow 0^+$. Therefore for such functions $\varphi$, by Fatou's lemma,
\[
\int \varphi  \ed (\nabla u^*)_\# \nu_1 \leq     \int \varphi  \ed \nu_0\,.
\]
In particular, the inequality holds for any lower semi-continuous convex function $\varphi$, since these may be approximated by Lipschitz functions, e.g., by setting $\varphi^n(x) =  \inf_y \varphi(y) + n|x-y|$. Moreover, formally, taking $\varphi = \pm (u -|\cdot|^2(t-1)/(2t) )$ 
we get the first condition in \eqref{eq:optimalityu}. Note however that even if $\varphi = -(u -|\cdot|^2(t-1)/(2t))$ is such that $u+ \varphi$ is convex, it may itself be neither convex nor continuous so one needs an additional regularization argument to conclude the proof; see point (b) in the proof of Theorem 1.2 in \cite{gozlan2020mixture}.
}

On the other hand, given a {lower semi-continuous} $(t-1)/t$-convex function $u$ that satisfies \eqref{eq:optimalityu}, for any lower semi-continuous $(t-1)/t$-convex function $\tilde{u}$,
\[
    \int (\tilde{u} - u) \ed \eta = \int \left(\tilde{u}(x) - \frac{t-1}{t}\frac{|x|^2}{2} \right) \ed \eta(x) - \int \left( u(x)-\frac{t-1}{t}\frac{|x|^2}{2}\right) \ed  \eta(x)\leq 0\,.
\]
This implies that
\[
\begin{aligned}
\int u \ed \nu_0 + \int u^* \ed  \nu_1&  =\int u \ed   (\nabla u^*)_\# \nu_1  +\int u^* \ed \nu_1 - \int u \ed \eta \\
&\leq \int \tilde{u} \ed   (\nabla u^*)_\# \nu_1  +\int \tilde{u}^* \ed \nu_1 - \int \tilde{u} \ed \eta \\
& = \int \tilde{u}^* \ed \nu_1 + \int \tilde{u} \ed \nu_0\,,\\
\end{aligned}
\]
where we also used the fact that $u$ is an optimal Brenier potential for the transport from $\nabla u^*_\# \nu_1$ to $\nu_1$.
\end{proof}

The following theorem establishes that strong duality holds between problems \eqref{eq:diffW2} and \eqref{eq:dualBrenier}.

\begin{theorem}\label{th:dualityW2} We have
\begin{equation}\label{eq:equalitypp}
\eqref{eq:diffW2} =  \eqref{eq:dualBrenier} - \int \frac{|x|^2}{2} \ed \nu_0(x) - \int \frac{|x|^2}{2} \ed \nu_1(x)\,.
\end{equation}
Moreover, 
\begin{enumerate}
    \item $\mu$ solves \eqref{eq:diffW2} if and only if $\mu = T(t)_\# \nu_1$, where 
 \[
 T(t)  \coloneqq  \nabla u^* + t(\mathrm{Id} - \nabla u^*)
 \]   
and $u$ solves \eqref{eq:dualBrenier};
\item conversely, $u$ solves \eqref{eq:dualBrenier} if and only if
$-(t-1)u^* + t|\cdot|^2/2$ and $tu -(t-1)|\cdot|^2/2$ are optimal Brenier potentials from $\nu_1$ to $\mu$ and from $\nu_0$ to $\mu$, respectively,
and $\mu$ solves  \eqref{eq:diffW2}.
\end{enumerate}
\end{theorem}

\begin{proof}
    
Let $u$ solve \eqref{eq:dualBrenier}. By the optimality conditions in Lemma \ref{lem:optimalityu}, and specifically the equality in \eqref{eq:optimalityu}, we have
\begin{equation}\label{eq:pstaroptimal}
\begin{aligned}
\eqref{eq:dualBrenier}  &= \int \left(u (\nabla u^*(x)) + u^*(x)\right) \ed \nu_1(x) - \frac{t-1}{2t} \int |x|^2 \ed \eta(x)\\
&=\int \left(\langle \nabla u^*(x),x\rangle - \frac{t-1}{2t} |\nabla u^*(x)|^2\right) \ed \nu_1(x) + \frac{t-1}{2t} \int |x|^2 \ed \nu_0(x)\,,\\
\end{aligned}
\end{equation}
where $\eta$ is defined in Lemma \ref{lem:optimalityu}.
Define 
\[ \mu \coloneqq  (\nabla u^* + t(\mathrm{Id} - \nabla u^*))_\# \nu_1 \quad \text{and } \quad \overline{\nu}_0 \coloneqq (\nabla u^*)_\# \nu_1.
\]

By the optimality conditions for $u$, and denoting $w\coloneqq tu - (t-1)|\cdot|^2/2$ (which is convex due to the strong convexity of $u$),
we deduce that
\begin{equation}\label{eq:w20bound}
\begin{aligned}
-\frac{W^2_2(\mu,{\nu}_0)}{2t} &= \frac{1}{t} \inf_{v \text{ convex}} \left\{ \int v\ed \nu_0 + \int v^* \ed \mu \right\}- \int \frac{|x|^2}{2t} \ed \nu_0(x) - \int \frac{|x|^2}{2t} \ed \mu(x) \\
& \leq \frac{1}{t} \left[ \int w\ed \overline{\nu}_0 + \int w^* \ed \mu\right] - \int \frac{|x|^2}{2t} \ed \nu_0(x) - \int \frac{|x|^2}{2t} \ed \mu(x)  \\
& = \frac{1}{t} \int \left[  w(\nabla u^*(x))  + w^*( \nabla u^*(x) + t(x - \nabla u^*(x))) \right] \ed \nu_1(x) \\
&\quad -\int \frac{|x|^2}{2t} \ed \nu_0(x) - \int \frac{|x|^2}{2t} \ed \mu(x)\,,
\end{aligned}
\end{equation}
where for the first equality we used the dual formulation in \eqref{eq:dualKantorovichb}. 
Since $x \in \partial u({\nabla u^*(x)})$, the subdifferential of $u$ at $\nabla u^*(x)$, and given that $w$ is convex, we have that $\nabla u^*(x) + t(x - \nabla u^*(x)) \in \partial w({\nabla u^*(x)})$, the subdifferential of $w$ at $\nabla u^*(x)$. Therefore, we obtain
\begin{multline}\label{eq:bound0}
-\frac{W^2_2(\mu,{\nu}_0)}{2t} \leq \frac{1}{t} \int \langle \nabla u^*(x),  \nabla u^*(x) + t(x - \nabla u^*(x))\rangle \ed \nu_1(x) \\ -\int \frac{|x|^2}{2t} \ed \nu_0(x) - \int \frac{|\nabla u^*(x) + t(x - \nabla u^*(x))|^2}{2t} \ed \nu_1(x)\,.
\end{multline}

Consider now the coupling:
\[
\gamma \coloneqq (\mathrm{Id}, \nabla u^* + t(\mathrm{Id} - \nabla u^*))_\# \nu_1  \in \Gamma(\nu_1,\mu)\,.
\] 
Since $\gamma$ is an admissible coupling for the transport from $\nu_1$ to $\mu$,
\begin{equation}\label{eq:boundtP}\begin{aligned}
\eqref{eq:diffW2} & \leq  
\int  \frac{|x-x_1|^2}{2(t-1)}\ed \gamma(x_1,x)  -   \frac{W^2_2(\mu,\nu_0)}{2t} \\
& \leq  -\int \frac{|tx -(t-1) \nabla u^*(x)|^2}{2t(t-1)} \ed \nu_1(x)  +\int \frac{|x|^2}{2(t-1)}\ed\nu_1(x) - \int \frac{|x|^2}{2t} \ed \nu_0(x)\,,
\end{aligned}
\end{equation}
where the second inequality can be obtained by using \eqref{eq:bound0} and the definition of $\gamma$, and rearranging terms. 
Developing the square and comparing the result with equation \eqref{eq:pstaroptimal} we get 
\[
\eqref{eq:diffW2} \leq  \eqref{eq:dualBrenier} - \int \frac{|x|^2}
{2}\ed \nu_0(x) - \int \frac{|x|^2}{2} \ed \nu_1(x)\,.
\]
Recalling \eqref{eq:ineq}, this implies the equality. 

The  reasoning above implies a posteriori that given $u$ optimal for \eqref{eq:dualBrenier}, $\mu = (\nabla u^* + t(\mathrm{Id} - \nabla u^*))_\# \nu_1$ is necessarily the unique minimizer of problem \eqref{eq:diffW2}. This implies point (1) in the theorem.

For the only if part of point (2), in view of \eqref{eq:equalitypp} and the inequalities \eqref{eq:bound0} and \eqref{eq:boundtP}, we just need to check that
$(1-t)u^* +t |\cdot|^2/2$ is convex if $u$ is $(t-1)/t$-strongly convex. 
Observe that if $\nabla u^*$ is $t/(t-1)$-Lipschitz, for all $x,y\in\mathbb{R}^d$,
\begin{equation}\label{eq:lsmooth}
u^*(y) \leq u^*(x) + \langle \nabla u^*(x), y-x\rangle + \frac{t}{t-1}\frac{|y-x|^2}{2}\,.
\end{equation}
Then, since $t>1$, this is equivalent to
\[
\begin{aligned}
(1-t) u^*(y) + t\frac{|y|^2}{2}&\geq (1-t)u^*(x) +(1-t) \langle \nabla u^*(x), y-x\rangle - t \frac{|y-x|^2}{2}+ t\frac{|y|^2}{2}\\
& = (1-t)u^*(x) + t\frac{|x|^2}{2} +\langle (1-t)  \nabla u^*(x) +t x, y-x\rangle\,,
\end{aligned}
\]
which in turn is equivalent to the convexity of $(1-t)u^* +t |\cdot|^2/2$.

Finally, to show the if part in point (2) of the theorem we follow the same steps to show \eqref{eq:ineq}. Specifically, let $\mu$ solve \eqref{eq:diffW2} and $tu +(1-t)|\cdot|^2/2$ be an optimal Brenier potential from  from $\nu_0$ to $\mu$. Then $u$ is admissible for problem \eqref{eq:dualBrenier}. Moreover, proceeding as in equation \eqref{eq:boundtoland},
\begin{equation}\label{eq:boundtoland2}
\eqref{eq:diffW2} \geq  \int \varphi^{c}_{t-1} \, \ed \nu_1 - \int \varphi^c_{t} \ed \nu_0 \,,
\end{equation}
where $\varphi$ is an optimal Kantorovich potential for the transport from $\mu$ to $\nu_0$ (i.e., an optimal potential for the dual formulation in equation \eqref{eq:dualKantorovich} with $s=t$). By direct calculation, $\varphi^c_t =- u + |\cdot|^2/2$ and $\varphi^c_{t-1} \geq - [\varphi^c_t]^c = u^* -|\cdot|^2/2$, by equation \eqref{eq:phicc}. Reinserting this in \eqref{eq:boundtoland2} and using the equality \eqref{eq:equalitypp}, we deduce that $u$ solves \eqref{eq:dualBrenier}.

\end{proof}

From the proof of Theorem \ref{th:dualityW2}, one can also deduce that given $u$ solving \eqref{eq:dualBrenier} then $t u + (1-t)|\cdot|^2/2$  is
an optimal Brenier potential for the transport not only from $\nu_0$ to $\mu$ but also from $\overline{\nu}_0$ to $\mu$. In particular, if $\nu_0$ and $\overline{\nu}_0$ are absolutely continuous, one has that the solution $\mu$ to \eqref{eq:diffW2} satisfies
\begin{equation}\label{eq:muexpression}
\mu = (t\nabla u +(1-t) \mathrm{Id})_\# \overline{\nu}_0=(t\nabla u +(1-t) \mathrm{Id})_\# {\nu}_0\,.
\end{equation}

Theorem \ref{th:dualityW2} confirms the observation in Remark \ref{rem:extrap}. In particular, the optimal Brenier potential $u$ from $\nu_0$ to $\nu_1$ is $(t-1)/t$-convex if and only if there exists a length-minimizing geodesic $s\in [0,t]  \mapsto \nu(s) \in \mc{P}_2(\mathbb{R}^d)$ such that $\nu(0) = \nu_0$ and $\nu(1)= \nu_1$. In this case $u$ solves \eqref{eq:dualBrenier}, and the solution to \eqref{eq:diffW2} is $\nu(t)$.
{If on the contrary the optimal Brenier potential from $\nu_0$ to $\nu_1$ is not sufficiently convex, problem \eqref{eq:dualBrenier} still selects a potential $u$ which is sufficiently convex to trace a geodesic connecting $\nu_0$ at time $0$ to the extrapolation $\mu$ at time $t$. However, by the optimality conditions \eqref{eq:optimalityu}, $\nabla u^*$ pushes $\nu_1$ to $\overline{\nu}_0\preceq_C \nu_0$: the mass at the initial time is reorganized to guarantee the existence of a length-minimising geodesic up to time $t$ passing through $\nu_1$ at time 1.}

\begin{remark}[Toland's duality] 
The strong duality established in Theorem \ref{th:dualityW2} can be partly seen as an instance of Toland's duality \cite{toland1978duality,toland1979duality}, just as for the case studied in \cite{carlier2008Toland}. Given two convex, proper and lower semi-continuous functions $F, G:V \rightarrow (-\infty, \infty]$, where $V$ is a normed vector space, Toland's duality yields the equivalence
\[
\inf_{x \in V} \left\{ F(x)-G(x) \right\} = \inf_{p\in V^*} \left\{G^*(p)-F^*(p)\right\}\,,
\]
where $V^*$ is the topological dual of $V$ and $F^*$ and $G^*$ are the Legendre transforms of $F$ and $G$ respectively. In our case, $F$ and $G$ are replaced by the maps
\[
\mu \in \mc{P}_2(\mathbb{R}^d) \mapsto \frac{W^2_2(\mu,\nu_1)}{2(t-1)} \quad \text{and} \quad
\mu \in \mc{P}_2(\mathbb{R}^d) \mapsto \frac{W^2_2(\mu,\nu_0)}{2t} \,,
\]
respectively, whereas $F^*$ and $G^*$ are replaced by
\[
\varphi \in C_b(\mathbb{R}^d) \mapsto -\int \varphi^c_{t-1}\ed \nu_1\quad \text{and} \quad \varphi \in C_b(\mathbb{R}^d) \mapsto -\int \varphi^c_{t} \ed \nu_0\,,
\]
respectively. Then, at least formally, the Toland dual is the problem appearing on the right-hand side in \eqref{eq:boundtoland}. Note, however, that this is not a convex problem and one needs a further change of variable to recover \eqref{eq:dualBrenier}.
\end{remark}

\section{Barycentric Optimal Transport formulation}\label{sec:weak}

In this section we provide yet another equivalent formulation of problem \eqref{eq:diffW2}, which turns out to coincide with a specific type of barycentric optimal transport given by 
\begin{equation} \label{eq:barycentric2}\tag{$\mc{B}$}
\inf_{\pi\in \Gamma(\nu_0,\nu_1)} \int |t x_1 - (t-1) \mathrm{bary}(\pi_{x_1})|^2 \ed \nu_1(x_1)\,,
\end{equation}
where
\[
\mathrm{bary}(\pi_{x_1})\coloneqq \int x_0 \,\ed \pi_{x_1}(x_0)\,.
\]
This is a specific instance of a weak optimal transport problem as studied in \cite{gozlan2017kantorovich,gozlan2020mixture}. This equivalence is not surprising given that the dual formulation \eqref{eq:dualBrenier} coincides up to minor modifications to that derived for the barycentric optimal transport problem in \cite{gozlan2017kantorovich}. We give here a direct proof of the equivalence between \eqref{eq:diffW2} and \eqref{eq:barycentric}: this can also be seen as an alternative way to prove the dual formulation \eqref{eq:dualBrenier} for problem \eqref{eq:barycentric} via Toland's duality. As a byproduct, it also  yields the same type of characterization for the minimizers of problem \eqref{eq:barycentric} as the one proved in \cite{gozlan2020mixture} for a slight variant of our problem.

Our main tool will be Strassen's theorem \cite{strassen1965existence}. 
Before stating it, let us recall that given $\mu,\nu \in \mc{P}(\mathbb{R}^d)$  and a coupling $\theta \in \Gamma(\mu,\nu)$, we can always represent $\theta$ in the following disintegrated form (with respect to the projection on the first marginal):
\[
\ed \theta(x, y) = \ed \theta_x(y) \ed \mu(x)\,,
\]
where we can interpret $\theta_x$ as a probability measure on $\mathbb{R}^d$ for all $x$ (which is uniquely defined $\mu$-a.e.). We will use a similar notation for couplings with multiple marginals.
Furthermore, we say that $\mu,\nu \in \mc{P}_1(\mathbb{R}^d)$ are in convex order $\mu \preceq_C \nu$ if and only if
\[
\int \varphi \ed \mu \leq \int \varphi \ed \nu 
\]
for all (continuous) convex  functions $\varphi:\mathbb{R}^d\rightarrow \mathbb{R}$.

\begin{lemma}[Strassen's theorem] \label{lem:strassen} Let $\mu, \nu \in \mc{P}_1(\mathbb{R}^d)$ then $\mu \preceq_C \nu$ if and only if there exists a coupling $\theta \in \Gamma(\mu,\nu)$ such that $\ed \theta(x,y) = \ed \theta_x(y) \ed \mu(x)$ and
\begin{equation}\label{eq:martingale}
\int y \ed \theta_x(y) = x \quad \text{ for $\mu$-a.e. } x\in\mathbb{R}^d\,.
\end{equation}
\end{lemma}

\begin{remark} In probabilistic terms, the condition on $\theta$ in equation \eqref{eq:martingale} is equivalent to the existence of a martingale $(X,Y)$, where  $X$ and $Y$ are random variables with laws $\mu$ and $\nu$, respectively. Then, $\theta$ is the law of the martingale $(X,Y)$. 
\end{remark}

Let us go back to showing the equivalence of problems \eqref{eq:diffW2} and \eqref{eq:barycentric}. We start by rewriting problem \eqref{eq:diffW2} as follows, using the definition of the Wasserstein distance in \eqref{eq:w2},
\begin{equation}\label{eq:tricoupling}
\begin{aligned}
\eqref{eq:diffW2} & = \inf_\mu \left[\inf_{\gamma_1 \in \Gamma(\nu_1,\mu)} \int \frac{|x_1-x|^2}{2(t-1)}\ed \gamma_{1}(x_1,x)  - \inf_{\gamma_0 \in \Gamma(\nu_0,\mu)} \int \frac{|x_0-x|^2}{2t}\ed \gamma_{0}(x_0,x)\right]\\
& = \inf_\mu \left[\inf_{\gamma_1 \in \Gamma(\nu_1,\mu)} \int \frac{|x_1-x|^2}{2(t-1)}\ed \gamma_{1}(x_1,x)  - \inf_{\eta \in \Gamma(\nu_0,\nu_1,\mu)} \int \frac{|x_0-x|^2}{2t}\ed \eta(x_0,x_1, x)\right].\\
\end{aligned}
\end{equation}
Given a $\pi\in \Gamma(\nu_0,\nu_1)$, which can be disintegrated as follows \[
\ed \pi(x_0,x_1) = \ed \pi_{x_1}(x_0)\ed \nu_1(x_1)\,,
\]
one can construct a coupling $\eta$ by setting 
\[
\ed \eta (x_0,x_1,x) = \ed \pi_{x_1}(x_0) \ed \gamma_1(x_1,x)\,.
\]
It is easy to check that $\eta \in \Gamma(\nu_0,\nu_1,\mu)$.
If we restrict $\eta$ in \eqref{eq:tricoupling} to belong to this specific class of couplings we get the following lower bound:
\[
\begin{aligned}
\eqref{eq:diffW2} & \geq  \inf_\mu \inf_{\gamma_1 \in \Gamma(\nu_1,\mu)} \sup_{\pi \in \Gamma(\nu_0,\nu_1)} \left[ \int \frac{|x_1-x|^2}{2(t-1)}\ed \gamma_{1}(x_1,x)  - \int \int \frac{|x_0-x|^2}{2t}\ed \pi_{x_1}(x_0) \ed \gamma_{1}(x_1,x)\right]\\
& \geq  \sup_{\pi \in \Gamma(\nu_0,\nu_1)}  \inf_{\substack{\gamma_1 \in \Gamma(\nu_1,\mu)\\ \mu} } \left[ \int \frac{|x_1-x|^2}{2(t-1)}\ed \gamma_{1}(x_1,x)  - \int \int \frac{|x_0-x|^2}{2t}\ed \pi_{x_1}(x_0) \ed \gamma_{1}(x_1,x)\right]\,,
\end{aligned}
\]
where the second inequality is due to the exchange of $\inf$ and $\sup$.
We can obtain a further lower bound by minimising the whole integrand with respect to $x$, i.e.
\[
\begin{aligned}
\eqref{eq:diffW2} &\geq   \sup_{\pi \in \Gamma(\nu_0,\nu_1)}  \inf_{\substack{\gamma_1 \in \Gamma(\nu_1,\mu)\\ \mu} } \int \inf_{x}\left[ \frac{|x_1-x|^2}{2(t-1)}  - \int \frac{|x_0-x|^2}{2t}\ed \pi_{x_1}(x_0) \right]  \ed \gamma_1(x_1,x)\\
&=  \sup_{\pi \in \Gamma(\nu_0,\nu_1)}  \int \inf_{x}\left[ \frac{|x_1-x|^2}{2(t-1)}  - \int \frac{|x_0-x|^2}{2t}\ed \pi_{x_1}(x_0) \right]  \ed \nu_1(x_1)\,. 
\end{aligned}
\]
Expanding the squares and using the fact that $\pi \in\Gamma(\nu_0,\nu_1)$ we obtain
\begin{multline*}
\eqref{eq:diffW2} \geq - \frac{1}{2t(t-1)} \inf_{\pi\in \Gamma(\nu_0,\nu_1)} \int |t x_1 - (t-1) \mathrm{bary}(\pi_{x_1})|^2 \ed \nu_1(x_1)\\
+\frac{1}{2(t-1)}\int|x|^2\ed\nu_1(x) -\frac{1}{2t} \int |x|^2 \ed \nu_0(x)\,.
\end{multline*}
Hence we have
\[
\eqref{eq:diffW2} \geq - \frac{1}{2t(t-1)} \eqref{eq:barycentric}
+\frac{1}{2(t-1)}\int|x|^2\ed\nu_1(x) -\frac{1}{2t} \int |x|^2 \ed \nu_0(x)\,.
\]

The fact that the inequality we obtained between \eqref{eq:diffW2} and \eqref{eq:barycentric} is in fact an equality is a direct consequence of Strassen's theorem.  Specifically, if $u$ solves \eqref{eq:dualBrenier}, then by Lemma \ref{lem:optimalityu}, we have that  $\overline{\nu}_0 \coloneqq \nabla u^*_\# \nu_1\preceq_C \nu_0$, and therefore by Strassen's theorem, there exists a coupling $\theta \in \Gamma(\overline{\nu}_0,\nu_0)$ with $ \ed\theta(x,y) =\ed \theta_{x}(y) \ed \overline{\nu}_0(x)$ and
\begin{equation}\label{eq:martingalebar}
\int y \ed \theta_{x}(y) = x\,,  \quad \text{ for $\overline{\nu}_0$-a.e. } x\in\mathbb{R}^d.
\end{equation}
This can be used to show the following result:

\begin{theorem}\label{th:barycentric}
    The following equality holds:
    \begin{equation}\label{eq:eqBP}
    \eqref{eq:diffW2} = - \frac{1}{2t(t-1)} \eqref{eq:barycentric}
+\frac{1}{2(t-1)}\int|x|^2\ed\nu_1(x) -\frac{1}{2t} \int |x|^2 \ed \nu_0(x)\,.
    \end{equation}
    Moreover, 
\begin{enumerate}
    \item $\pi$ solves \eqref{eq:barycentric} if and only if
\begin{equation}\label{eq:baryustar}
\nabla u^*(x_1) = \mathrm{bary}(\pi_{x_1}) \,,\quad \text{ for } \nu_1\text{-a.e. } x_1\in \mathbb{R}^d\,,
\end{equation}
and $u$ solves \eqref{eq:dualBrenier};
    \item conversely, $u$ solves \eqref{eq:dualBrenier} if and only if \eqref{eq:baryustar} holds
and $\pi$ solves \eqref{eq:barycentric};
\end{enumerate}
\end{theorem}
\begin{proof}

Let $u$ solve \eqref{eq:dualBrenier}. Reasoning as in the the proof of Theorem \ref{th:dualityW2} we get the bound in equation \eqref{eq:boundtP}, i.e.
\[
\eqref{eq:diffW2} \leq  -\int \frac{|tx -(t-1) \nabla u^*(x)|^2}{2t(t-1)} \ed \nu_1(x) +\frac{1}{2(t-1)}\int|x|^2\ed\nu_1(x) -\frac{1}{2t} \int |x|^2 \ed \nu_0(x)\,.
\]
By Strassen's theorem, there exists a coupling $\theta \in \Gamma(\overline{\nu}_0,\nu_0)$ verifying \eqref{eq:martingalebar}. Therefore, we can define a coupling $\pi \in \Gamma(\nu_0,\nu_1)$ as follows
\begin{equation}\label{eq:pidef}
\ed \pi(x_0,x_1) =  \ed \theta_{\nabla u^*(x_1)}(x_0) \ed {\nu}_1(x_1)\,,
\end{equation}
which is equivalent to \eqref{eq:baryustar}. Then the inequality \eqref{eq:boundtP} is equivalent to
\begin{multline*}
\eqref{eq:diffW2} \leq  -\frac{1}{2t(t-1)} \int \left|t x_1 - (t-1)\int x_0 \ed \pi_{x_1}(x_0)\right|^2 \ed \nu_1(x_1) \\+\frac{1}{2(t-1)}\int|x|^2\ed\nu_1(x) -\frac{1}{2t} \int |x|^2 \ed \nu_0(x)\,.
\end{multline*}
This shows that $\pi$ is optimal for \eqref{eq:barycentric} and directly implies equation \eqref{eq:eqBP}.

In order to conclude the proof, we just need to observe that the function minimized in problem \eqref{eq:barycentric} is strongly convex when viewed as a function of  $\mathrm{bary}(\pi_{x_1})$. This means that any minimizer of problem \eqref{eq:barycentric} must satisfy \eqref{eq:baryustar} with $u$ solving \eqref{eq:dualBrenier}.
\end{proof}

\begin{remark}[Generalizations] \label{rem:mm}
It is natural to ask to what extent the study presented in this and the previous sections applies to other variants of problem \eqref{eq:diffW2}. The most natural scenarios are the case where the $l^2$ cost is replaced by a general cost function, and the case with multiple measures. For the first one, when using a general cost it seems reasonable to at least require convexity along generalized geodesics, which is guaranteed upon some conditions on the cost, and namely the so-called non-negative cross curvature condition (a proof of this fact can be found in \cite{leger2024nonnegative}, which extends previous studies on this condition 
\cite{villani2009optimal,kim2010continuity,figalli2011multidimensional}). 
The case with multiple measures would consist in considering the problem:
\[
\inf_\mu \left\{ \sum_{i=1}^{K^+} \lambda_i^+ W^2_2(\mu,\nu_i^+) -\sum_{j=1}^{K^-}\lambda^-_j W^2_2(\mu,\nu_j^-) \right\},
\]
where $K^-,K^+ \geq 1$, $\nu_i^+,\nu_j^-\in \mc{P}_2(\mathbb{R}^d)$ and $\lambda_i^+,\lambda_j^->0$ for all $i,j$, $\sum_i \lambda_i^+ > \sum_j \lambda_j^+$. 
In the particular case where $K^+=1$, the functional is still strongly convex along generalized geodesics with base $\nu_1^+$, and we expect that one can derive an equivalent version of problem \eqref{eq:dualBrenier} and \eqref{eq:barycentric}. In the general case, however, the picture is much less clear, and although one can still apply Toland duality, it is not trivial to check that the resulting problem is convex in appropriate variables and to derive an associated weak optimal transport problem. We reserve to study these issues in future works.
\end{remark}

\section{Relation with the $H^1$ projection on convex functions}\label{sec:relation}
Given an optimal Brenier potential \( u \) for the transport from \( \nu_0 \) to \( \nu_1 \), if \( \nu_0 \) is absolutely continuous, an alternative definition of extrapolation could simply consist in computing the push-forward of \( \nu_0 \) by the map \( (1-t) \mathrm{Id} + t\nabla u \). As discussed in Section \ref{sec:existence}, this provides a solution to \eqref{eq:diffW2} as long as the potential \( u - (t-1)|\cdot|^2/(2t) \) is convex. In this section, we establish a precise link between the $\nu_0$-weighted \( \dot{H}^1 \) projection of this potential onto the set of convex functions and problem \eqref{eq:diffW2}. Specifically, we show that, while the two problems are equivalent in one dimension, in general the \( \dot{H}^1 \) projection problem can still be recovered from \eqref{eq:diffW2} through a specific limit and a \( \Gamma \)-convergence argument.

\subsection{One-dimensional setting}\label{sec:1d}
Let us consider the specific case where $d=1$. Given $\mu \in \mc{P}_2(\mathbb{R})$, let us denote by $F_\mu:\mathbb{R}\rightarrow [0,1]$ the right-continuous non-decreasing function given by the cumulative distribution of $\mu$, i.e.\ $F_\mu(x)\coloneqq \mu((-\infty,x])$, and denote by $F^{[-1]}_\mu:[0,1]\rightarrow \overline{\mathbb{R}}$ its (left-continuous) pseudo-inverse defined by
\[
F^{[-1]}_\mu (a) = \inf \{ x \in\mathbb{R}
~:~ F_\mu(x)\geq a\} \]
with the usual convention $\inf \emptyset =+\infty$, which is also known as the quantile function associated with $\mu$. 

The optimal transport plan between two measures in $\mc{P}_2(\mathbb{R})$ for the $l^2$ cost is given by the monotone rearrangement plan \cite[Theorem 2.9]{santambrogio2015optimal}. As a consequence we can reformulate the minimization problem \eqref{eq:diffW2} as follows 
\[
\inf \left\{ \frac{1}{2(t-1)} \int_0^1  |F^{[-1]}_\mu(x)-F^{[-1]}_{\nu_1}(x)|^2 \ed x- \frac{1}{2t} \int_0^1 |F^{[-1]}_\mu(x)-F^{[-1]}_{\nu_0}(x)|^2 \ed x\right\},
\]
where the infimum is taken over any quantile function $F_\mu^{[-1]}$. 
Rearranging the squares, the minimizers of this problem are the solutions to
\begin{equation}\label{eq:diff1dquantile}
\inf \int_0^1 \left| F^{[-1]}_\mu(x)- (t F^{[-1]}_{\nu_1}(x) - (t-1) F^{[-1]}_{\nu_0}(x)) \right|^2 \ed x\,.
\end{equation}
Suppose that $\nu_0$ is absolutely continuous, then the optimal transport map from $\nu_0$ to $\nu_1$ is $T_{0}\coloneqq F_{\nu_1}^{[-1]} \circ F_{\nu_0}$ and  the problem above is equivalent to
\begin{equation}\label{eq:diff1dT}
\inf_{T} \int_0^1 \left|T - (t T_{0}(x) - (t-1) x) \right|^2 \ed \nu_0(x)\,,
\end{equation}
where $T \in L^2_{\nu_0}(\mathbb{R})$ is required to be non-decreasing, and given a minimizer $T$ we can retrieve the minimizer $\mu$ of the original problem by
\[
  \mu = T_\# \nu_0\,.
\]
This is just the $L^2$ projection on the space of monotone functions of the map
\[
x \mapsto  x+ t (T_{0}(x) - x)\,.
\]

\begin{remark}[Equivalence with pressureless fluids]\label{rem:pressureless} 
Suppose still that $\nu_0$ is absolutely continuous and let $T_0(x) = x+ \partial_x \phi_0$, where $x\mapsto |x|^2/2 + \phi_0(x)$ is the Brenier  potential for the transport from $\nu_0$ and $\nu_1$. The solutions to the minimization problem \eqref{eq:diff1dT} can be related to the particular solutions of the pressureless Euler system, $\mu:[0,\infty)\rightarrow \mc{P}_2(\mathbb{R})$, $v:[0,\infty)\rightarrow L^2(\mu(t);\mathbb{R})$, solving
	\begin{equation}\label{eq:pressureless}\left\{
		\begin{aligned}
		&\partial_t \mu + \partial_x(\mu v) = 0\,,\\
		&\displaystyle \partial_t (\mu  v) + \partial_x\left( \mu v^2 \right) =  0\,,\\
		\end{aligned}\right.
	\end{equation}
	with initial conditions given by
	\[
	\mu (0) = \nu_0\,, \quad v(0,\cdot) = \nabla \phi_0(\cdot)\,,
	\]
	and satisfying the \emph{sticky collision} condition, which consists in requiring particles to share the same position after collision. 
	Brenier and Grenier \cite{brenier1998one} constructed sticky solutions to problem \eqref{eq:pressureless} for which the evolution of $\mu$ is expressed explicitly by the curve
	\[ 
	\mu(t) = \tilde{X}(t,\cdot)_\# \nu_0,
	\] 
	with 
	\begin{equation}\label{eq:pushenvelope}
		\tilde{X}(t,x) \coloneqq  (\partial_x \mathrm{co}\, \psi(t,\cdot)) \circ F_0(x)\,, \quad \psi(t,s) \coloneqq \int_{0}^s X(t, F_{\nu_0}^{[-1]}(s')) \, \ed s'\,, 
	\end{equation}
	where $\mr{co}\,\psi(t,\cdot)$ denotes the convex envelope of $\psi(t,\cdot)$, and $X(t,x) = x + t \nabla \phi_0(x)$. Note that as long as the geodesic can be extended $\psi(t,\cdot)$ stays convex (as it is the integral of a monotone function) and therefore $\tilde{X} = X$. It is easy to check that $\tilde{X}(t,\cdot)$ and $\mu(t)$ are precisely the solutions to \eqref{eq:diff1dT} and \eqref{eq:diffW2}, respectively.
  \end{remark}

\subsection{A $\Gamma$-convergence result} In this section we elucidate the link between problem \eqref{eq:diffW2} and a generalization of the projection onto monotone maps \eqref{eq:diff1dT} in the multi-dimensional setting. For the sake of clarity, from now on we focus on the case where $\nu_0$ is absolutely continuous. 
Consider a function $\phi_0:\mathbb{R}^d\rightarrow \mathbb{R}$ such that $x\mapsto |x|^2 + 2 \bar{\varepsilon} \phi_0(x)$ is convex for some $\bar{\varepsilon}\geq 1$, and such that $\nabla \phi_0$ is $L$-Lipschitz. For $0<\varepsilon\leq \bar{\varepsilon}$, we consider the functionals defined by
\[
\mc{G}_\varepsilon(\mu) \coloneqq \frac{t(t-\varepsilon)}{\varepsilon} \left[ \frac{W^2_2(\mu,\nu_\varepsilon)}{t-\varepsilon} -  \frac{W^2_2(\mu,\nu_0)}{t}  \right]
\]
where 
\[
\nu_\varepsilon \coloneqq ( \mathrm{Id} + \varepsilon \nabla\phi_0)_\# \nu_0 \,.
\]
We investigate here the limit $\varepsilon\rightarrow 0^+ $ which, loosely speaking and up to multiplicative factors, corresponds to the case where $\nu_1$ tends to $\nu_0$ in \eqref{eq:diffW2} along the geodesic connecting the two.

\begin{proposition}
    The functionals $\mc{G}_\varepsilon$, as $\varepsilon \rightarrow 0^+$, $\Gamma$-converge on $(\mc{P}_2(\mathbb{R}^d),W_2)$ towards the functional $\mc{G}:\mc{P}_2(\mathbb{R}^d) \rightarrow \mathbb{R}$ defined by
    \begin{equation}\label{eq:gammalimit}
    \mc{G}(\mu) \coloneqq  \int  | \nabla u(x) - (x + t\nabla \phi_0(x) )|^2 \ed \nu_0(x) - t^2 W^2_2(\nu_0,\nu_1)
    \end{equation}
    where $u:\mathbb{R}^d\rightarrow \mathbb{R}$ is the unique  convex function such that $(\nabla u)_\# \nu_0 = \mu$.
\end{proposition}
\begin{proof}
    First of all, we rewrite $\mc{G}_\varepsilon$ as follows
    \[
    \mc{G}_\varepsilon(\mu) =  t \mc{H}_\varepsilon(\mu)  + W^2_2(\mu,\nu_0)\,,\quad \mc{H}_\varepsilon(\mu) \coloneqq \frac{W^2_2(\mu,\nu_\varepsilon)-W^2_2(\mu,\nu_0)}{\varepsilon}\,.
    \]
    We prove that $\mc{H}_\varepsilon$ $\Gamma$-converge towards
    \begin{equation}\label{eq:Heps}
    \mc{H}(\mu) = -2\int \langle \nabla u(x)-x, \nabla\phi_0(x)\rangle \ed \nu_0(x)\,,
    \end{equation}
    which implies the result directly. For this, observe that
    \[
    \mc{H}_\varepsilon(\mu) \leq \frac{1}{\varepsilon} \left[ \int |\nabla u - (x+\varepsilon \nabla \phi_0(x))|^2\ed \nu_0 - \int |\nabla u -x|^2\ed \nu_0 \right]= \mc{H}(\mu) + \varepsilon W^2_2(\nu_0,\nu_1)
    \]
    where we used the fact that $(\nabla u, \mathrm{Id}+\varepsilon\nabla \phi_0)_\# \nu_0$ is an admissible coupling for the transport from $\mu$ to $\nu_\varepsilon$. This implies directly the $\Gamma-\limsup$. For the $\Gamma-\liminf$, let $(\mu_\varepsilon)_\varepsilon \subset \mc{P}_2(\mathbb{R}^d)$ be such that $\mu_\varepsilon \rightarrow \mu$ in $W_2$ as $\varepsilon \rightarrow 0^+$. Let $\gamma_\varepsilon \in \Gamma(\mu_\varepsilon,\nu_\varepsilon)$ be an optimal coupling and denote
    \[
    u_{\varepsilon}(x) \coloneqq \frac{|x|^2}{2} + \varepsilon \phi_0(x)\,,
    \]
    which is convex by hypothesis.
    Note that  $(\nabla u_{\varepsilon}^*)_\# \nu_\varepsilon=\nu_0$. We have
    \[
    \begin{aligned}
    \mc{H}_\varepsilon(\mu_\varepsilon) & \geq \frac{1}{\varepsilon} \left[ \int |z-y|^2 \ed \gamma_\varepsilon (z,y)- \int |\nabla u_{\varepsilon}^*(y)-z|^2 \ed \gamma_\varepsilon(z,y)\right] \\
    &= -\frac{1}{\varepsilon} \int |\nabla u^*_\varepsilon (y) - y|^2 \ed \nu_\varepsilon(y) + \frac{2}{\varepsilon} \int \langle \nabla u^*_\varepsilon (y) - y, z -y\rangle \ed \gamma_\varepsilon(z,y)\\
    &= -\varepsilon \int |\nabla \phi_0(x)|^2 \ed \nu_0(x) -2 \int \langle \nabla \phi_0(\nabla u^*_\varepsilon(y)), z-y \rangle \ed \gamma_\varepsilon(z,y)\,. 
    \end{aligned}
    \]
    Recalling that $L$ is the Lipschitz constant of $\nabla \phi_0$, we obtain
    \[
    \begin{aligned}
    \mc{H}_\varepsilon(\mu_\varepsilon) 
      & \geq -\varepsilon W^2_2(\nu_0,\nu_1) -2 \int \langle \nabla \phi_0(\nabla u^*(y))-\nabla \phi_0(y)+ \nabla\phi_0(y), z-y \rangle \ed \gamma_\varepsilon(z,y) \\
      & \geq -\varepsilon W^2_2(\nu_0,\nu_1) - 2 L \varepsilon W_2(\nu_0,\nu_1) W_2(\mu_\varepsilon,\nu_\varepsilon)-2 \int \langle \nabla \phi_0(y), z-y \rangle \ed \gamma_\varepsilon(z,y)\,.
    \end{aligned}
    \]
     By the stability of optimal transport maps, as $\varepsilon \rightarrow 0^+$, $\gamma_\varepsilon$ converges in $W_2$ towards the optimal map $\gamma$ from $\nu_0$ to $\mu$ (since it converges weakly and its 2-moments converge towards those of $\gamma$), therefore we get
    \[
    \liminf_{\varepsilon \rightarrow 0^+} \mc{H}_\varepsilon(\mu_\varepsilon) \geq \mc{H}(\mu)\,.
    \]
    \end{proof}

\begin{remark} The functional $\mc{G}$ defined in \eqref{eq:gammalimit} is 2-convex along generalized geodesics with base $\nu_0$, which is expected since it is the $\Gamma$-limit of the functionals $\mc{G}_\varepsilon$ which are 2-convex along generalized geodesics with base $\nu_\varepsilon$. 
\end{remark}

\section{Particular solutions} \label{sec:particular}  When there is no geodesic from $\nu_0$ to $\nu_1$ that stays length-minimizing when extended up to time $t$, the solution to  problem \eqref{eq:diffW2} is not trivial in general. Nonetheless, in this section, we construct some explicit solutions corresponding to specific choices for the measures $\nu_0$ and $\nu_1$.

\subsection{Extrapolation when $\nu_1$ is a Dirac mass} We suppose $\nu_1 = \delta_{x_1}$ with $x_1\in\mathbb{R}^d$ fixed. As before, for any $\mu\in\mc{P}_2(\mathbb{R}^d)$, we denote by $\mathrm{bary}(\mu)\in\mathbb{R}^d$ its barycenter
\[
\mathrm{bary}(\mu) \coloneqq \int x \ed \mu(x)\,.
\]
By Jensen's inequality
\[
\begin{aligned}
W^2_2(\mu,\nu_0) &= \int |x|^2 \ed \mu(x) + \int |x|^2 \ed \nu_0(x) - 2\inf\left\{ \int u \ed \nu_0 +\int u^* \ed \mu ; \, u\text{ convex}\right\} \\
&\leq  \int |x|^2 \ed \mu(x) + \int |x|^2 \ed \nu_0(x) - 2\inf\left\{ \int u \ed \nu_0 + u^*(\mathrm{bary}(\mu)) ; \, u\text{ convex}\right\} \\
&\leq  \int |x|^2 \ed \mu(x) - |\mathrm{bary}(\mu)|^2 +W^2_2(\delta_{\mathrm{bary}(\mu)},\nu_0)\,.
\end{aligned}
\]
{On the other hand,
\[
\begin{aligned}
    W_2^2(\mu,\nu_1)
    &=\int |x-x_1|^2\ed \mu(x)=\int |x|^2\ed \mu(x)+|x_1|^2-2\langle\mathrm{bary}(\mu), x_1\rangle \\
    &=\int |x|^2 \ed \mu(x) - |\mathrm{bary}(\mu)|^2 + W_2^2(\delta_{\mathrm{bary}(\mu)},\nu_1)\,.
\end{aligned}
\]}
This implies that
\begin{multline}\label{eq:boundbary}
\frac{W_2^2(\mu,\nu_1)}{t-1} - \frac{W_2^2(\mu,\nu_0)}{t} \geq 
 \frac{1}{t(t-1)} \left(\int |x|^2 \ed \mu(x) - |\mathrm{bary}(\mu)|^2\right)\\
 + \frac{W_2^2(\delta_{\mathrm{bary}(\mu)},\nu_1)}{t-1} - \frac{W_2^2(\delta_{\mathrm{bary}(\mu)},\nu_0)}{t} \,.
\end{multline}
Applying again Jensen's inequality, the first term on the right-hand side of \eqref{eq:boundbary} is non-negative. Hence $\mu$ must be a Dirac mass, and in particular
\[
\mu = \delta_{x_t} \, \quad \text{where} \quad x_t = \mathrm{bary}(\nu_0) + t( x_1 -\mathrm{bary}(\nu_0))\,. 
\]

\begin{remark} In this specific setting, the extrapolation is consistent with a sticky particle interaction. In particular after collapsing at a single point all particles move together while the total momentum is preserved.
\end{remark}

\subsection{Difference with locally minimizing geodesics} We detail here an example highlighting the difference between the metric extrapolation and locally minimizing geodesics. 

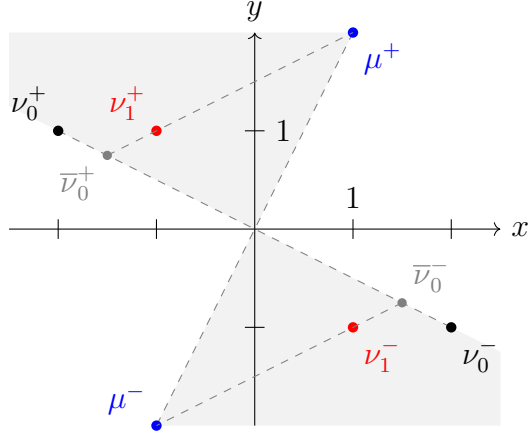
\begin{figure}
\begin{center}
\begin{tikzpicture}[scale =1.3]
         
    \fill[gray!10] plot[smooth, samples=100] (0,0) -- (1,2) -- (-2.5,2) -- (-2.5,1.25) ;
    \fill[gray!10] plot[smooth, samples=100] (0,0) -- (-1,-2) -- (2.5,-2) -- (2.5,-1.25) ;
    \draw[->] (-2.5,0) -- (2.5,0) node[right] {$x$} ;
    \draw[->] (0,-2) -- (0,2) node[above] {$y$};

    \foreach \x in {-2,-1,1,2}
        \draw (\x,-0.1) -- (\x,0.1);
    \foreach \y in {-1,1}
        \draw (-0.1,\y) -- (0.1,\y);

    \draw[gray,dashed] (-2,1) -- (2,-1);
    \draw[gray,dashed] (1,2) -- (-1,-2);

    \node[right] at (0.1,1) {1};
    \node[above] at (1,0.1) {1};

     \fill[black] (-2,1) coordinate(A) circle (1.5pt) node[above left] {$\nu_0^+$};
     \fill[black] (2,-1) coordinate(B) circle (1.5pt) node[below right] {$\nu_0^-$};
     \fill[red] (-1,1) coordinate(E) circle (1.5pt) node[above left] {$\nu_1^+$};
     \fill[red] (1,-1) coordinate(D) circle (1.5pt) node[below right] {$\nu_1^-$};
    \fill[blue] (-1,-2)  coordinate(C) circle (1.5pt) node[above left] {$\mu^-$};
    \fill[blue] (1,2) coordinate(F) circle (1.5pt) node[below right] {$\mu^+$};

    \node[gray,scale=3]  at (1.5,-.75){.};
    \node[gray,scale=3] at (-1.5,.75){.};
    \node[gray,below left] at (-1.5,.75) {$\overline{\nu}_0^+$};
    \node[gray,above right] at (1.5,-.75) {$\overline{\nu}_0^-$};

    \draw[gray,dashed] (-1.5,.75) -- (F);
    \draw[gray,dashed] (1.5,-.75) -- (C);

\end{tikzpicture}
\end{center}
\caption{An example in which a locally length-minimizing geodesic passing by $\nu_0 = \nu_0^++\nu_0^-$ and $\nu_1 = \nu_1^+ +\nu_1^-$ exists for all times but the metric extrapolation does not coincide with it.}
\label{fig:diracs}
\end{figure}

The setting we consider is represented in Figure \ref{fig:diracs}. In this example, $d=2$ and $\nu_0= \nu_0^+ + \nu_0^- $ and $\nu_1 = \nu_1^+ + \nu_1^- $ are each the sum of two Dirac masses:
\[
\nu_0^+ = \frac{1}{2} \delta_{(-2,1)}\,, \quad \nu_0^- = \frac{1}{2} \delta_{(2,-1)}\,, \quad 
\nu_1^+ = \frac{1}{2} \delta_{(-1,1)}\,, \quad \nu_1^- = \frac{1}{2} \delta_{(1,-1)}\,.
\]
Note that the curve \[
s\in\mathbb{R} \mapsto \nu(s) = \frac{1}{2} \delta_{(-2+s,1)} +\frac{1}{2} \delta_{(2-s,-1)}  
\]
is a locally length-minimizing geodesic such that $\nu(0)=\nu_0$ and $\nu(1) = \nu_1$. However we now show that for $t$ sufficiently large the minimizer of \eqref{eq:diffW2} is not $\nu(t)$.

First of all we observe that in this case any minimizer $\mu$ of \eqref{eq:diffW2} must be symmetric with respect to the origin. Furthermore it must be contained in the set $\{ y \geq \max(-x/2,2x)\}\cup \{ y\leq \min(-x/2,2x)\}$ (shaded in the figure). In fact, if this was not the case, we could always reflect the part of $\mu$ which is not contained in that area with respect to the axis $y=2x$ or $y=-x/2$ to get a measure $\overline{\mu}$ satisfying $W_2(\mu,\nu_0) = W_2(\overline{\mu},\nu_0)$ and $W_2(\mu,\nu_1) \geq W_2(\overline{\mu},\nu_1)$, leading to a contradiction.

Hence, we can write $\mu = \mu^+ + \mu^-$ where $\mu^+$ and $\mu^-$ have equal mass and solve
\[
\begin{aligned}
    \mu^+ = \argmin_\mu \left\{\frac{W^2(\mu,\nu_1^+)}{2(t-1)} - \frac{W^2(\mu,\nu_0^+)}{2t}\right\}\,, \\
    \mu^- = \argmin_\mu \left\{\frac{W^2(\mu,\nu_1^-)}{2(t-1)} - \frac{W^2(\mu,\nu_0^-)}{2t}\right\}\,,
\end{aligned}
\]
subject to the constraint that the support of $\mu^+$ is contained in $\{ y \geq \max(-x/2,2x)\}$, whereas the support of $\mu^-$ is contained in $\{ y\leq \min(-x/2,2x)\}$.
Proceeding like in the previous example, we deduce that $\mu^+=\delta_{(x^+,y^+)}/2$ and $\mu^-=\delta_{(x^-,y^-)}/2$ are Dirac masses. Therefore, the minimization becomes explicit and we find 
\[
(x^+,y^+)=
\begin{cases}
    (-2+t,1) & \text{if $t\le \frac{5}{2}$}\,,\\
    (\frac12,1)+(t-\frac{5}{2})(1,2) & \text{if $t> \frac{5}{2}$\,,}
\end{cases}
\]
and $(x^-,y^-) = -(x^+,y^+)$.
}

Note that we can obtain the same result in an alternative way which highlights the link with the barycentric formulation. Specifically, using the symmetry of the solution with respect to the axis $y=-x/2$ we have
\[
W^2_2(\mu, \nu_0) = \min_{s\in[0,1]} W^2_2(\mu^+,(1-s) \nu_0^+ + s \nu_0^-) +  W^2_2(\mu^-, s \nu_0^+ + (1-s) \nu_0^-)\,.
\]
Replacing this in the original problem we obtain
\begin{equation}\label{eq:muplusminus0}
\mu^+ = \argmin_\mu \max_{s\in[0,1]} \left\{\frac{W^2(\mu,\nu_1^+)}{2(t-1)} - \frac{W^2(\mu,s\nu_0^+ + (1-s) \nu_0^-)}{2t}\right\}\,, \end{equation} \begin{equation}\label{eq:muplusminus1} \mu^- = \argmin_\mu \max_{s\in[0,1]}\left\{ \frac{W^2(\mu,\nu_1^-)}{2(t-1)} - \frac{W^2(\mu,(1-s)\nu_0^+ + s \nu_0^-)}{2t}\right\}\,,
\end{equation}
where now the minimizations are unconstrained. 
For a fixed $s$ we are in the setting of the previous section, and in particular $\mu^+$ (respectively $\mu^-$) is the extrapolation up to time $t$ from $\overline{\nu}_0^+$ to ${\nu}_1^+$ (respectively, from $\overline{\nu}_0^-$ to ${\nu}_1^-$), where
\[
\overline{\nu}_0^+ \coloneqq \mathrm{bary}(s\nu_0^+ + (1-s) \nu_0^-)\,, \quad  \overline{\nu}_0^- \coloneqq \mathrm{bary}((1-s)\nu_0^+ + s \nu_0^-)\,,
\]
as shown in Figure \ref{fig:diracs}. The optimal $s$ can be determined by swapping the $\min$ and $\max$ in \eqref{eq:muplusminus0} and \eqref{eq:muplusminus1}, leading  to the same cost as in the barycentric problem \eqref{eq:barycentric}.

\section{A numerical scheme for atomic measures}\label{sec:numerical}

In this section we focus on the case where $\nu_0$ and $\nu_1$ are both atomic measures, i.e.
\[
\nu_0=\sum_{i=1}^M a_i \delta_{x_i} \,, \quad  \nu_1=\sum_{j=1}^N b_j \delta_{y_j}\,,
\]
for $(x_i)_{i=1}^M\eqqcolon X,(y_j)_{j=1}^N\eqqcolon Y \subset \Rd$ and $(a_i)_{i=1}^M,(b_j)_{j=1}^N\subset \R_+$, with $\sum_{i=1}^Ma_i=1$ and $\sum_{j=1}^Nb_j=1$. 
This is not restrictive, since our stability result in Lemma \ref{lem:stability} guarantees that the extrapolation of any measures $\nu_0$ to $\nu_1$ can be approximated in the $W_2$ sense by replacing $\nu_0$ and $\nu_1$ in problem \eqref{eq:diffW2} with an atomic approximation, which can always be done since atomic measures are dense in $\mathcal{P}_2(\Rd)$.
In this setting, any coupling $\pi \in \Gamma(\nu_0,\nu_1)$ is of the form
\[
\pi = \sum_{i=1}^N \sum_{j=1}^M P_{ij} \delta_{(x_i,y_j)}\,
\]
where
\[
P \in G(a,b) \coloneqq \Big\{ P \in \mathbb{R}^{M\times N}_+ \; :\; \sum_j P_{ij}=a_i \,, \; \sum_i P_{ij}=b_j\Big\},
\] 
and the barycentric problem \eqref{eq:barycentric} can be written as follows:
\begin{equation}\label{eq:bar_atomic}
\min_{P}\big\{ g(P) ~:~ P\in G(a,b) \big\} \,,
\end{equation}
where
\[
g(P)=\frac{1}{2t(t-1)}\sum_{j=1}^N \Big|ty_j-(t-1)\mathrm{bary}_j(P) \Big|^2 b_j\,,
\]
with $\mathrm{bary}_j(P)=\sum_i x_i \frac{P_{ij}}{b_j}$, for $1\le j\le N$.
As already pointed out, this is a quadratic but not strictly convex optimization problem, due to the possible kernel of the barycenter operator. Following the discussion in Section \ref{sec:weak}, from a solution $P$ one can recover the extrapolation from $\nu_0$ to $\nu_1$ at time $t$ as 
\[
\nu_t=\sum_j b_j\delta_{z_j}, \quad \text{where}\quad z_j=ty_j-(t-1)\sum_{i} x_i \frac{P_{ij}}{b_j}.
\]
From the marginal constraints on $P$, $z_j\in ty_j-(t-1)\conv(X)$ for every $1\le j\le N$, where $\conv$ denotes the convex hull.

Entropic regularization is a popular approach for solving optimal transport in this fully discrete setting \cite{peyre2020computational}. It consists in adding the entropy of the transport plan as a barrier function, in order to smoothen the problem and enforce at the same time the nonnegativity constraint. The approach is particularly attractive as one can explicitly write the dual problem as an unconstrained minimization problem and solve it via the Sinkhorn algorithm.
However, since problem \eqref{eq:bar_atomic} is quadratic and not linear as the standard optimal transport problem, one cannot rely on the Sinkhorn algorithm directly. Following \cite{carlier2023sista}, we consider a variant of Sinkhorn algorithm, alternating the Sinkhorn iterations with a gradient descent step on an appropriate auxiliary variable, with similar complexity and a linear convergence guarantee. We will provide a detailed proof of this last result following the same strategy as in \cite{carlier2023sista}, but with some adaptations to deal with the specific form of our cost, and to ensure that the gradient descent step-size can scale linearly with $\varepsilon$. This was not a concern in \cite{carlier2023sista} but it is crucial in our case since we are mainly interested in the limit $\varepsilon\rightarrow 0$. 

\subsection{Sinkhorn algorithm with an additional gradient descent step}

For $\varepsilon>0$, the entropic barycentric problem is
\begin{equation}\label{eq:bar_entropic}
\min_{P} \left\{ g(P)+\varepsilon \sum_{ij} P_{ij}\left(\log\Big(\frac{P_{ij}}{a_ib_j}\Big)-1\right)~:~ \sum_j P_{ij}=a_i, \sum_i P_{ij}=b_j \right\},
\end{equation}
where the regularization term is the relative entropy of $P$ with respect to the coupling $(a_i b_j)_{i=1,j=1}^{M,N}\in G(a,b)$. The entropy is by definition equal to $+\infty$ for any negative value, thereby automatically ensuring the non-negativity of any finite candidate coupling.
The following proposition collects useful facts for problem \eqref{eq:bar_entropic}. Importantly, this admits an unconstrained dual formulation, as the standard entropic optimal transport problem.

\begin{proposition}\label{prop:entropic}
    The following holds:
    \begin{enumerate}[i)]
    
        \item there exists a unique solution $P^\varepsilon$ to \eqref{eq:bar_entropic} and $P^\varepsilon\rightarrow P$ for $\varepsilon\rightarrow 0$, where $P$ is the maximal entropy solution to \eqref{eq:bar_atomic};
        
        \item \eqref{eq:bar_entropic} is equal, up to a constant term and a sign change, to 
        \begin{equation}\label{eq:bar_entropic_dual}
        \min_{\phi,\psi,Z} f(\phi,\psi,Z)\,,
        \end{equation}
        where $\phi=(\phi_i)_{i=1}^M,\psi=(\psi_j)_{j=1}^N\subset \R$, $Z=(z_j)_{j=1}^N\subset \Rd$ and
        \begin{multline}
        f(\phi,\psi,Z)=\varepsilon \sum_{ij} a_ib_j\left[\exp\left(\frac{-C^t_{ij}(Z)+\phi_i+\psi_j}\varepsilon\right) - \frac{\phi_i+\psi_j}{\varepsilon} \right]\\+\frac{1}{2(t-1)}\sum_j\left|z_j-y_j\right|^2b_j \,,
        \end{multline}
        for $C_{ij}^t(Z)=\frac{1}{2t} \left|z_j-x_i\right|^2$;
        \item $\nu_t^\varepsilon\coloneqq\sum_{j=1}^N b_j \delta_{z^\varepsilon_j} \rightharpoonup \sum_{j=1}^N b_j \delta_{z_j} =\nu_t$, for $\varepsilon\rightarrow 0$, where $Z^\varepsilon$ solves \eqref{eq:bar_entropic_dual}.
    \end{enumerate}
\end{proposition}
\begin{proof}
    $i)$ Existence and uniqueness is evident since the function to be minimized is strictly convex and coercive and the feasible set is convex and not empty. The convergence towards the unregularized problem can be easily obtained adapting the proof in \cite[Proposition 4.1]{peyre2020computational} to the quadratic case.
    
    $ii)$ Notice that, as done in Section \ref{sec:weak}, the function $g(P)$ can be rewritten as
    \begin{equation}\label{eq:bar_atomic_obj}
        g(P) = \max_{Z} \sum_{ij} C^t_{ij}(Z) P_{ij}-\frac{1}{2(t-1)}\sum_j\left|z_j-y_j\right|^2b_j + C\,,
    \end{equation}
    for any $P\in G(a,b)$, where the constant term is
    \[
    C=-\frac{1}{2t}\sum_{ij} |x_i|^2 P_{ij} +\frac{1}{2(t-1)}\left|y_j \right|^2 b_j = -\frac{1}{2t}\sum_{i} \left|x_i\right|^2 a_i +\frac{1}{2(t-1)} \sum_j \left|y_j \right|^2b_j\,.
    \]
     Enforcing the mass constraints with the Lagrange multiplier $\phi$ and $\psi$, and exchanging inf and sup by a standard duality argument, we obtain the dual problem
    \begin{multline*}
        \max_{\phi,\psi,Z} \min_P \sum_{ij} C^t_{ij}(Z) P_{ij}-\frac{1}{2(t-1)} \sum_j\left|z_j-y_j\right|^2b_j +\varepsilon \sum_{ij} P_{ij}\left(\log\Big(\frac{P_{ij}}{a_ib_j}\Big)-1\right)\\+\sum_i \phi_i a_i+\sum_j\psi_jb_j\,,
    \end{multline*}
    up to the constant term.
    The optimization in $P$ provides
    \[
    P_{ij}=a_ib_j\exp\left( \frac{-C^t_{ij}(Z)+\phi_i+\psi_j}\varepsilon\right)\,, \quad 1\le i\le M, \, 1\le j\le N.
    \]
    Plugging this in the dual problem and switching sign, we obtain \eqref{eq:bar_entropic_dual}.

    $iii)$ By the optimality, $P^\varepsilon\in G(a,b)$ and
    \[
    z^\varepsilon_j=ty_j-(t-1) \mathrm{bary}_j(P^\varepsilon)\,, \quad  1\le j\le N.
    \]
    The convergence of $\nu_t^\varepsilon\rightharpoonup \nu_t$ is therefore an immediate consequence of $i)$.
\end{proof}

A similar entropic optimal transport problem with non-fixed cost matrix has been considered in the recent paper \cite{carlier2023sista}. 
The authors introduced a novel iterative procedure, called SISTA, alternating between the phase of exact minimization over the potentials and a phase of proximal gradient descent over the parameters defining the transport cost.
We follow the same approach for \eqref{eq:bar_entropic_dual}, performing a gradient descent step with respect to $Z$.
Notice that the potentials are defined up to a common constant, as $(\phi+c,\psi-c)$ attain the same value in \eqref{eq:bar_entropic_dual} as $(\phi,\psi)$ for any $c\in \mathbb{R}$. We fix this constant by setting $\phi^n_M=0$ for any $n\ge0$. 
Then, starting from $\psi^0$ and $Z^0$, the $\phi$ and $\psi$ updates are
\begin{align}
    &
    \begin{aligned}
        &\phi^{n+1}_i=-\varepsilon \log \sum_{j=1}^N b_j\exp\left(\frac{-C^t_{ij}(Z^n)+\psi_j^n}\varepsilon\right), \qquad 1\le i\le M-1,\\
        & \phi^{n+1}_M=0,
    \end{aligned}
    \label{eq:phistep}\\
    &\psi^{n+1}_j=-\varepsilon \log \sum_{i=1}^M a_i\exp\left(\frac{-C^t_{ij}(Z^n)+\phi_i^{n+1}}\varepsilon\right), \quad 1\le j\le N, \label{eq:psistep}
\end{align}
whereas for $\tau>0$ the gradient descent step writes for $1\le j\le N$ as
\begin{equation}\label{eq:zstep}
z_j^{n+1}=z_j^n-\tau \left[ \frac{z^n_j-y_j}{t-1}b_j -  \sum_{i=1}^M \frac{z_j^n-x_i}{t} a_i b_j \exp\left(\frac{-C^t_{ij}(Z^n)+\phi_i^{n+1}+\psi_j^{n+1}}\varepsilon\right) \right].
\end{equation}
Note that the first two steps coincide with the classical Sinkhorn algorithm, but we perform a gradient step with respect to $Z$ instead of an exact minimization since this would lead to a nonlinear system to be solved at each iteration.

\subsection{Convergence analysis}
Let us denote  $x^n=(\phi^n,\psi^n,Z^n)\in \R^{M+(1+d)N}$ and
\[
x^{n+\frac{1}{3}}=(\phi^{n+1},\psi^n,Z^n)\,, \quad x^{n+\frac{2}{3}}=(\phi^{n+1},\psi^{n+1},Z^n)\,, \quad x^{n+1}=(\phi^{n+1},\psi^{n+1},Z^{n+1})\,.
\]
To perform the convergence analysis, we will consider the norm $|\cdot|_S$, and its dual, defined as
\[
|x|_S^2\coloneqq |\phi|_\infty^2+|\psi|_\infty^2+|Z|_2^2\,, \quad
|p|_{S^*}^2=|p_\phi|_1^2+|p_\psi|_1^2+|p_Z|_2^2\,,
\]
for $x=(\phi,\psi,Z)$ and $p=(p_\phi,p_\psi,p_Z)$. Note, in particular, that the $l^\infty$ norm for the discrete potentials is the natural counterpart for the supremum norm, and it will simplify some of the necessary estimates in the proof. We define:
\begin{align}
    &L_\phi \coloneqq \sup_{n\ge0}\max_{s\in[0,1]}\max_{\substack{|h_\phi|_\infty=1 \\ (h_\phi)_M=0}}\nabla^2_\phi f\big((1-s)x^{n+\frac13}+sx^{n+1}\big)h_\phi \cdot h_\phi \,,\label{eq:Lphi} \\
    &L_\psi \coloneqq \sup_{n\ge0}\max_{s\in[0,1]}\max_{|h_\psi|_\infty=1}\nabla^2_\psi f\big((1-s)x^{n+\frac23}+sx^{n+1}\big)h_\psi \cdot h_\psi \,, \label{eq:Lpsi} \\
    &L_Z \coloneqq \sup_{n\ge0}\max_{s\in[0,1]}\max_{|h_Z|_2=1}\nabla^2_Z f\big((1-s)x^{n+\frac23}+sx^{n+1}\big)h_Z \cdot h_Z\,. \label{eq:LZ}
\end{align}
These constants quantify the smoothness of $f$, with respect to the $|\cdot|_S$ norm, restricted along linear interpolations between consecutive iterations of the algorithm.

The following lemmas show that the iterates $Z^n$ stay bounded upon a step restriction independent of $\varepsilon$, and as a consequence the constants above scale linearly with $1/\varepsilon$. In the following we will denote
\[
D \coloneqq \max_{ij} |x_i-y_j|\,.
\]

\begin{lemma}\label{lem:bound_iter}
    Assume $z^0_j\in ty_j-(t-1)\conv(X)$, for all $j$, and $\tau\le \frac{t(t-1)}{|b|_\infty}$. Then,  for all $n\geq 0$ and all $j$, $z^n_j\in ty_j-(t-1) \conv(X)$, and in particular,
    \[
    |z_j-y_j|\leq (t-1) D \,,\quad |z_j-x_i|\leq t D.
    \]
\end{lemma}
\begin{proof}
    The gradient descent update on $Z^n$ can be written as
    \[
    z_j^{n+1}=\Big(1-\frac{b_j\tau}{t(t-1)}\Big)z_j^n+\frac{b_j\tau}{t(t-1)}\sum_i \big(ty_j-(t-1)x_i\big) \frac{P^n_{ij}}{b_j}\,, \quad 1\le j\le N\,.
    \]
    Since by assumption $\frac{b_j\tau}{t(t-1)}\le 1$, we deduce recursively that at each step $z_j^{n+1}\in ty_j-(t-1)\conv(X)$. 
\end{proof}
 
\begin{lemma}\label{lem:boundsL}
    Assume $z^0_j\in ty_j-(t-1)\conv(X)$, for all $j$, and $\tau\le \frac{t(t-1)}{|b|_\infty}$. Then,
    \[
    L_\phi\le \frac{C}{\varepsilon}\,, \quad L_\psi\le \frac{C}{\varepsilon}\,, \quad 
    L_Z\le |b|_\infty\Big( \frac{(t-1)C+t}{t(t-1)}+\frac{C D^2}{\varepsilon}\Big),
    \]
    where $C\coloneqq \exp(2\lambda D^2 |b|_\infty)$
     and $\lambda\coloneqq {\tau}/{\varepsilon}$.
\end{lemma}

\begin{proof}
    Notice first that from the updates on $\phi$ and $\psi$, we have
    \begin{alignat}{2}
        &\sum_{j=1}^N b_j \exp \left(\frac{-C_{ij}^t(Z^n)+\phi^{n+1}_i+\psi^n_j}\varepsilon\right)=1\,, \quad &&1\le i\le M-1,\label{eq:proba_j}\\
        &\sum_{i=1}^M a_i \exp\left(\frac{-C_{ij}^t(Z^n)+\phi^{n+1}_i+\psi^{n+1}_j}\varepsilon\right)=1\,, \quad &&1\le j\le N. \label{eq:proba_i}
    \end{alignat}
    Then, the second non-mixed derivatives of the objective function along the iterations of the algorithm are, for all $1\le i\le M-1$ and $1\le j\le N$,
    \[
    \begin{aligned}
        \partial^2_{\phi_i} f(x^{n+\frac13})&=\frac{a_i}{\varepsilon} \sum_{j=1}^N b_j \exp\left(\frac{-C^t_{ij}(Z^n)+\phi_i^{n+1}+\psi_j^n}\varepsilon\right)=\frac{a_i}{\varepsilon}\,, \\
        \partial^2_{\psi_j} f(x^{n+\frac23})&=\frac{b_j}{\varepsilon} \sum_{i=1}^M a_i \exp\left(\frac{-C^t_{ij}(Z^n)+\phi_i^{n+1}+\psi_j^{n+1}}\varepsilon\right)=\frac{b_j}{\varepsilon}\,, \\
        \nabla^2_{z_j} f(x^{n+\frac23})&=
        \begin{multlined}[t]
            \frac{b_j}{t-1}\Id_d-\frac{1}{t} \sum_{i=1}^M a_i b_j \exp\left(\frac{-C_{ij}^t(Z^n)+\phi^{n+1}_i+\psi^{n+1}_j}\varepsilon\right) \Id_d\\
            +\frac{1}{\varepsilon t^2}\sum_{i=1}^M (z_j^n-x_i)\otimes(z_j^n-x_i) a_ib_j\exp\left(\frac{-C_{ij}^t(Z^n)+\phi^{n+1}_i+\psi^{n+1}_j}\varepsilon\right)
        \end{multlined}\\
        &=\frac{b_j}{t(t-1)}\Id_d+\frac{1}{\varepsilon t^2}\sum_{i=1}^M (z_j^n-x_i)\otimes(z_j^n-x_i) a_ib_j\exp\left(\frac{-C_{ij}^t(Z^n)+\phi^{n+1}_i+\psi^{n+1}_j}\varepsilon\right),
    \end{aligned}
    \]
    where $\Id_d$ is the identity matrix on $\Rd$.
    Moreover, $\nabla^2_{\phi} f(x)$ and $\nabla^2_{\psi} f(x)$ are diagonal matrices with nonnegative components, whereas $\nabla_Z f(x)$ is block-diagonal for any $x$. It follows that
    \[
    \begin{gathered}
    \max_{\substack{|h_\phi|_\infty=1\\ (h_\phi)_M=0}}\nabla^2_\phi f(x)h_\phi \cdot h_\phi=\sum_{i=1}^{M-1} \partial^2_{\phi_i} f(x)\,,\\
    \max_{|h_\psi|_\infty=1}\nabla^2_\psi f(x)h_\psi \cdot h_\psi=\sum_{j=1}^N \partial^2_{\psi_j} f(x)\,,\\
    \max_{|h_{Z}|_2=1} \nabla^2_{Z} f(x)h_Z\cdot h_Z=\max_{j}\max_{|h_{z_j}|_2=1} \nabla^2_{z_j} f(x)h_{z_j}\cdot h_{z_j}\,.
    \end{gathered}
    \]

    By Lemma \ref{lem:bound_iter}, for all $n\ge 0$ and all $1\le j\le N$, $z^n_j\in ty_j-(t-1)\conv(X)$ and
    \[
    |\nabla_{z_j} f(x^{n+\frac23}) |\le b_j \Big(\frac{1}{t} \max_i(|z^n_j-x_i|)+\frac{1}{t-1} |z_j^n-y_j|\Big)\le 2 D b_j\,,
    \]
    Therefore, using the gradient descent update \eqref{eq:zstep},
    \[
    \begin{aligned}\label{eq:boundCZ}
    -C_{ij}^t(Z^n+s(Z^{n+1}-Z^n))&\le -C_{ij}^t(Z^{n})+s \tau D |\nabla_{z_j} f(\phi^{n+1},\psi^{n+1},Z^n) |\\
    &\le -C_{ij}^t(Z^{n})+s\tau 2 D^2 |b|_\infty\,.
    \end{aligned}
    \]
    Then, by convexity of the exponential function and using \eqref{eq:proba_j}-\eqref{eq:proba_i}, we can bound $L_\phi$ and $L_\psi$ as
    \[
    \begin{aligned}
        L_\phi&\le \sup_{n}\max_{s\in[0,1]} \exp(s\lambda 2 D^2 |b|_\infty)\left( (1-s) \sum_{i=1}^{M-1} \partial^2_{\phi_i} f(x^{n+\frac13})+s \sum_{i=1}^{M-1} \partial^2_{\phi_i} f(x^{n+\frac23}) \right)\\
        &\le \max_{s\in[0,1]} \frac{\exp(s\lambda 2 D^2 |b|_\infty)}{\varepsilon} \Big( (1-s)\sum_{i=1}^M a_i+s\sum_{j=1}^N b_j\Big) \le \frac{C}{\varepsilon}\,,\\
        L_\psi&\le \sup_{n}\max_{s\in[0,1]} \exp(s\lambda 2 D^2 |b|_\infty)\left( \sum_{j=1}^N \partial^2_{\psi_j} f(x^{n+\frac23}) \right)\\
        &= \max_{s\in[0,1]} \frac{\exp(s\lambda 2 D^2 |b|_\infty)}{\varepsilon} \sum_{j=1}^N b_j \le \frac{C}{\varepsilon}\,.
    \end{aligned}
    \]
 
    Finally, using again \eqref{eq:boundCZ} and \eqref{eq:proba_i},
    \begin{multline*}
        \max_{|h_{z_j}|_2=1} \nabla^2_{z_j} f((1-s)x^{n+\frac23}+sx^{n+1})h_{z_j} \cdot h_{z_j} 
        \le \\
        \begin{aligned}
            &\le\frac{b_j}{t-1}+\frac{Cb_j}{t}+\frac{C b_j}{\varepsilon t^2}\max_i \max_{|h_{z_j}|_2=1} \big[(z_j^n-x_i)\otimes(z_j^n-x_i) \big] h_{z_j} \cdot h_{z_j}\\
            &\le b_j\Big( \frac{(t-1)C+t}{t(t-1)}+\frac{CD^2}{\varepsilon} \Big)
        \end{aligned}
    \end{multline*}
    and the last estimate follows.
\end{proof}

Lemma \ref{lem:boundsL} suggests that if $\tau$ scales linearly with $\varepsilon$, $L_Z$ scales as $1/\varepsilon$ and therefore $L_Z\leq 1/\tau$ for $\tau$ sufficiently small. This is made precise in the following Lemma:

\begin{lemma}\label{lem:Ktau}
    Assuming $\varepsilon\le 1$, there exists $K>0$ only depending on $t$, $|b|_\infty$ and $D$, such that if $\tau\le K \varepsilon$ then  $\tau\le \min\big(\frac{t(t-1)}{|b|_\infty},\frac{1}{L_Z}\big)$. 
\end{lemma}
\begin{proof}
    One needs to find $K$ such that
    \[
    \left\{
    \begin{aligned}
        & K \varepsilon \le \frac{t(t-1)}{|b|_\infty}\,, \\
        & K \exp(K)(A\varepsilon+B)+\frac{K \varepsilon}{t-1} \le \frac{1}{|b|_\infty}\,,
    \end{aligned}
    \right.
    \]
    where $A=\exp(2D^2|b|_\infty)/t$ and $B=\exp(2D^2|b|_\infty)D^2$. In particular, one can take $K= \min\big(\frac{t(t-1)}{|b|_\infty},\overline{K}\big)$ where $\overline{K}$ solves
    \[
    \overline{K}\exp(\overline{K}) +\frac{\overline{K}}{t-1}= \frac{1}{(A+B)|b|_\infty}\,.
    \]
\end{proof}

The following proof of convergence of iterations \eqref{eq:phistep}-\eqref{eq:psistep}-\eqref{eq:zstep} is an adaptation from \cite{carlier2023sista}. 
Before stating it, let us recall that $f$ being $\mu$-convex and $L$-smooth on a set $S$ with respect to the norm $|\cdot|$ means that, $\forall x,y\in S$, 
\begin{gather}
f(x)\ge f(y)+\nabla f(y)\cdot (x-y)+\frac{\mu}{2}|x-y|^2\,,\label{eq:muconvex} \\
f(x)\le f(y)+\nabla f(y)\cdot (x-y)+\frac{L}{2}|x-y|^2\,.\label{eq:Lsmooth} 
\end{gather}
Moreover, the gradient of $f$ is $L$-Lipschitz, that is 
\[
|\nabla f(x)-\nabla f(y)|_*\le L |x-y|
\]
$\forall x,y\in S$, where $|\cdot|_*$ is the dual norm with respect to $|\cdot|$.
See for example \cite{nesterov2018lectures,bubeck2015convex} for more details on strong convexity and smoothness in convex optimization. 

\begin{proposition}\label{prop:convergence}
Suppose $\tau\le K \varepsilon$, with $K$ as in Lemma \ref{lem:Ktau}. Then, iterations \eqref{eq:phistep}-\eqref{eq:psistep}-\eqref{eq:zstep} converge with linear rate, i.e. there exists $0<\eta<1$ such that
\[
f(x^{n+1})-f(x^*)\le \eta \big(f(x^n)-f(x^*)\big)\,, \quad \forall\, n\ge0\,,
\]
where $x^*$ is the unique solution of problem \eqref{eq:bar_entropic_dual}.
\end{proposition}
\begin{proof}
    First of all, the objective function is decreasing monotonically along iterations of the algorithm, i.e.
    \begin{equation}\label{eq:alg_decrease}
    f(x^{n+1})\le f(x^{n+\frac{2}{3}})\le f(x^{n+\frac{1}{3}})\le f(x^n)\,, \quad \forall\, n\ge0.
    \end{equation}
    The first two inequalities come directly by definition of alternate minimization.
    Concerning the third inequality, since $f$ is $L_Z$-smooth along $(1-s)x^{n+\frac23}+sx^{n+1}$, $\forall\, n\ge0$, and $ L_Z \leq 1/{\tau}$ by Lemma \ref{lem:Ktau},
    \begin{equation}\label{eq:grad_decrease}
    \begin{aligned}
    f(x^{n+1})&\le f(x^{n+\frac23})+\nabla f(x^{n+\frac23}) \cdot (x^{n+1}-x^{n+\frac23})+\frac{1}{2\tau} |x^{n+1}-x^{n+\frac23}|_S^2 \\
    &=f(x^{n+\frac23})-\frac{1}{\tau}|x^{n+1}-x^{n+\frac23}|_S^2 +\frac{1}{2\tau} |x^{n+1}-x^{n+\frac23}|_S^2 \\
    &= f(x^{n+\frac23})-\frac{1}{2\tau} |x^{n+1}-x^{n+\frac23}|_S^2\,,
    \end{aligned}
    \end{equation}
    where we used the definition of gradient descent step in the first equality.

    Consider 
    \[
    g(\phi,\psi)\coloneqq \varepsilon \sum_{ij} a_ib_j\left[\exp\left(\frac{\phi_i+\psi_j}\varepsilon-\frac{tD^2}{2\varepsilon} \right) - \left(\frac{\phi_i+\psi_j}{\varepsilon}-\frac{tD^2}{2\varepsilon}\right) \right] - \frac{tD^2}{2},
    \]
    and notice that $f(\phi,\psi,Z)\ge g(\phi,\psi)$ for any $Z$ such that $z_j\in ty_j-(t-1) \conv(X), \forall j$.
    Since $f(x^n)$ is decreasing,
    \[
    g(\phi^n,\psi^n)\le f(x^n)\le f(x^0)\,, \quad \forall\, n\ge 0\,.
    \]
    Using $\exp^x-x\ge |x|$, we obtain
    \[
    |\phi^n_i+\psi^n_j|\le f(x^0)+tD^2\eqqcolon \frac{C}{2}\,, \quad \forall \,(i,j),
    \]
    which in turn provides, using $\phi^n_M=0$ for any $n$, $|\phi^n|_\infty\le C$ and $|\psi^n|_\infty\le C$.
    Therefore, the iterates of the algorithm are uniformly bounded and belong to
    \[
    S\coloneqq \left\{ (\phi,\psi,Z)\in \R^{M+(1+d)N}~:~ |\phi|_\infty\le C\,, |\psi|_\infty\le C\,, z_j \in ty_j -(t-1) \mathrm{conv}(X) \right\}.
    \]
    
    The strictly convex function $f$ is furthermore $\mu$-convex on the compact set $S$, with
    \[
    \mu\coloneqq\min_{x\in S} \min_{|h|_S=1}\nabla^2 f(x)h \cdot h \,.
    \]
    Hence, at each step $n$, by definition of the alternate minimization, which provides $\nabla_\phi f(x^{n+\frac13})=0$ and $\nabla_\psi f(x^{n+\frac23})=0$, and using \eqref{eq:muconvex} with respect to the norm $|\cdot|_S$, it holds
    \[
    \begin{gathered}
    f(x^n)\ge f(x^{n+\frac13})+\frac{\mu}{2} |x^{n+\frac13}-x^n|_S^2 \,,\\
    f(x^{n+\frac13})\ge f(x^{n+\frac23})+\frac{\mu}{2} |x^{n+\frac23}-x^{n+\frac13}|_S^2\,. 
    \end{gathered}
    \]
    Combing these with \eqref{eq:grad_decrease} we obtain
    \begin{equation}\label{eq:fk1}
    \begin{aligned}
    (f(x^n)-f(x^*))-(f(x^{n+1})-f(x^*))&\ge \frac{\mu}{2} (|x^{n+\frac13}-x^n|_S^2+|x^{n+\frac23}-x^{n+\frac13}|_S^2)+\frac{1}{2\tau} |x^{n+1}-x^{n+\frac23}|_S^2 \\
    &\ge\frac{\mu}{2} |x^{n+1}-x^n|_S^2\,,
    \end{aligned}
    \end{equation}
    since $\mu\le L_Z\le\frac{1}{\tau}$.
    Again by \eqref{eq:muconvex},
    \[
    f(x^*)\ge f(x^{n+1})+\nabla f(x^{n+1})\cdot(x^*-x^{n+1})+\frac{\mu}{2} |x^*-x^{n+1}|_S^2
    \]
    and using further $x\cdot y+\frac{\mu}{2}|y|_S^2\ge-\frac{1}{2\mu} |x|_{S^*}^2$, we can write
    \[
    f(x^{n+1})-f(x^*)\le \frac{1}{2\mu} |\nabla f(x^{n+1})|_{S^*}^2\,.
    \]
    Since $f$ is $L_\phi$-smooth along $(1-s)x^{n+\frac13}+sx^{n+1}$, $L_\psi$-smooth along $(1-s)x^{n+\frac23}+sx^{n+1}$, and $L_Z$-smooth along $(1-s)x^{n+\frac23}+sx^{n+1}$, $\forall\, n\ge0$ and for $s\in[0,1]$, using again $\nabla_\phi f(x^{n+\frac13})=0$ and $\nabla_\psi f(x^{n+\frac23})=0$,
    \[
    \begin{aligned}
    |\nabla_\phi f(x^{n+1})|_1^2&=|\nabla_\phi f(x^{n+1})-\nabla_\phi f(x^{n+\frac13})|_1^2 \le L_\phi^2 |x^{n+1}-x^{n+\frac13}|_S^2\,, \\
    |\nabla_\psi f(x^{n+1})|_1^2&=|\nabla_\psi f(x^{n+1})-\nabla_\psi f(x^{n+\frac23})|_1^2 \le L_\psi^2 |x^{n+1}-x^{n+\frac23}|_S^2\,, \\
    |\nabla_Z f(x^{n+1})|_2^2&=
    |\nabla_Z f(x^{n+1})-\nabla_Z f(x^{n+\frac23}) + \frac{1}{\tau}(Z^n-Z^{n+1})|_2^2\\
    &=|\nabla_Z f(x^{n+1})-\nabla_Z f(x^{n+\frac23})|_2^2+  \frac{1}{\tau^2}|(Z^n-Z^{n+1})|_2^2\\
    &\quad +\frac{2}{\tau}(Z^n-Z^{n+1})\cdot (\nabla_Z f(x^{n+1})-\nabla_Z f(x^{n+\frac23})) \\
    &\le (L_Z^2+\frac{1}{\tau^2}) |x^{n+1}-x^{n+\frac23}|_S^2\,,
    \end{aligned}
    \]
    where in the last inequality we also used the gradient descent step and the convexity of $f$ with respect to $Z$. This provides
    \[
    f(x^{n+1})-f(x^*)\le \frac{\hat{L}^2+\tau^{-2}}{2\mu} |x^{n+1}-x^n|_S^2\,,
    \]
    where $\hat{L}^2=L_\phi^2+L_\psi^2+L_Z^2$,
    and combining this with \eqref{eq:fk1} we finally get
    \[
    f(x^{n+1})-f(x^*)\le \frac{\hat{L}^2\tau^2+1}{\hat{L}^2\tau^2+1+\mu^2\tau^2} (f(x^n)-f(x^*))\,.
    \]
\end{proof}

\begin{remark}
    From the proof of Proposition \ref{prop:convergence}, since the iterates are uniformly bounded with respect to $\varepsilon$, one can infer that $\mu\rightarrow 0$ exponentially fast for $\varepsilon\rightarrow 0$.
    Notice that on the compact set $S$ the function $f$ is furthermore $L$-smooth with
    \[
    L\coloneqq\max_{x\in S} \max_{|h|_S=1}\nabla^2 f(x)h \cdot h\,,
    \]
    and similarly, $L\rightarrow \infty$ exponentially fast for $\varepsilon\rightarrow 0$. Since the constant of smoothness is related to the maximal stepsize $\tau$, 
    this motivates our choice of working with the local smoothness constants \eqref{eq:Lphi}-\eqref{eq:Lpsi}-\eqref{eq:LZ}, which scale only linearly with respect to $1/\varepsilon$, contrary to the global constant $L$.
    In the end, for $\tau=K\varepsilon$, we have
    \[
    \eta\le\frac{1}{1+C_1 \exp(-C_2/\varepsilon)}
    \]
    for some positive constants $C_1,C_2$ independent of $\varepsilon$, so that we obtain the same exponential dependency on $\varepsilon$ of the linear convergence rate as in the classical Sinkhorn algorithm.
\end{remark}

\subsection{Numerical tests}

\subsubsection{Extrapolation of shapes} Figure \ref{fig:extrapolation} illustrates some examples of extrapolations obtained using the algorithm discussed in this section. The data $\nu_0$ and $\nu_1$ are given by a collection of Dirac masses with equal weights, whose location is chosen to provide an optimal quantization in the $W_2$ sense of two uniformly distributed densities. For all test cases in this figure $\varepsilon = 10^{-3}$.

\begin{figure}
    \centering
    \setlength{\tabcolsep}{0pt}
    \begin{tabularx}{.8\textwidth}{>{\centering\arraybackslash}X>{\centering\arraybackslash}X>{\centering\arraybackslash}X>{\centering\arraybackslash}X>{\centering\arraybackslash}X>{\centering\arraybackslash}X>{\centering\arraybackslash}X>{\centering\arraybackslash}X}
        {\includegraphics[trim={6cm 4cm 6cm 4cm},clip,width=0.1\textwidth]{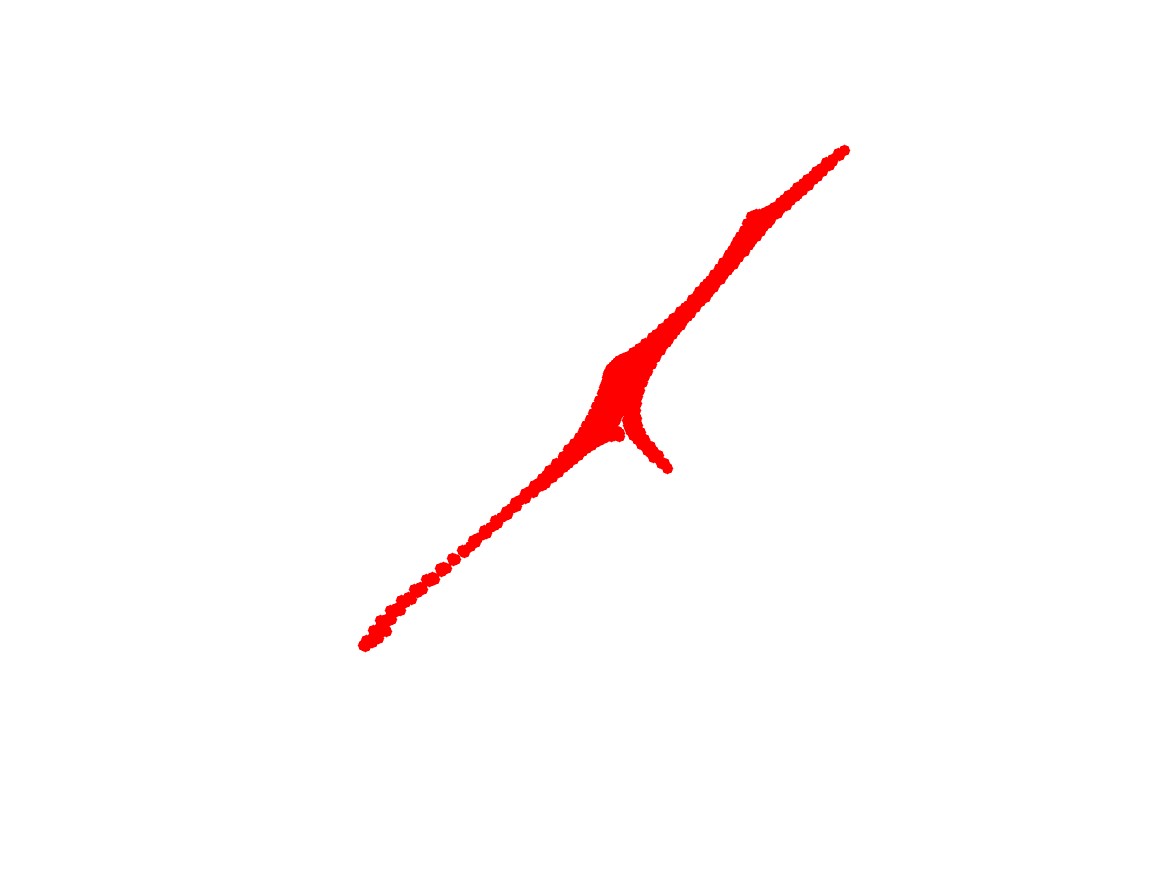}}& & & 
        {\includegraphics[trim={6cm 4cm 6cm 4cm},clip,width=0.1\textwidth]{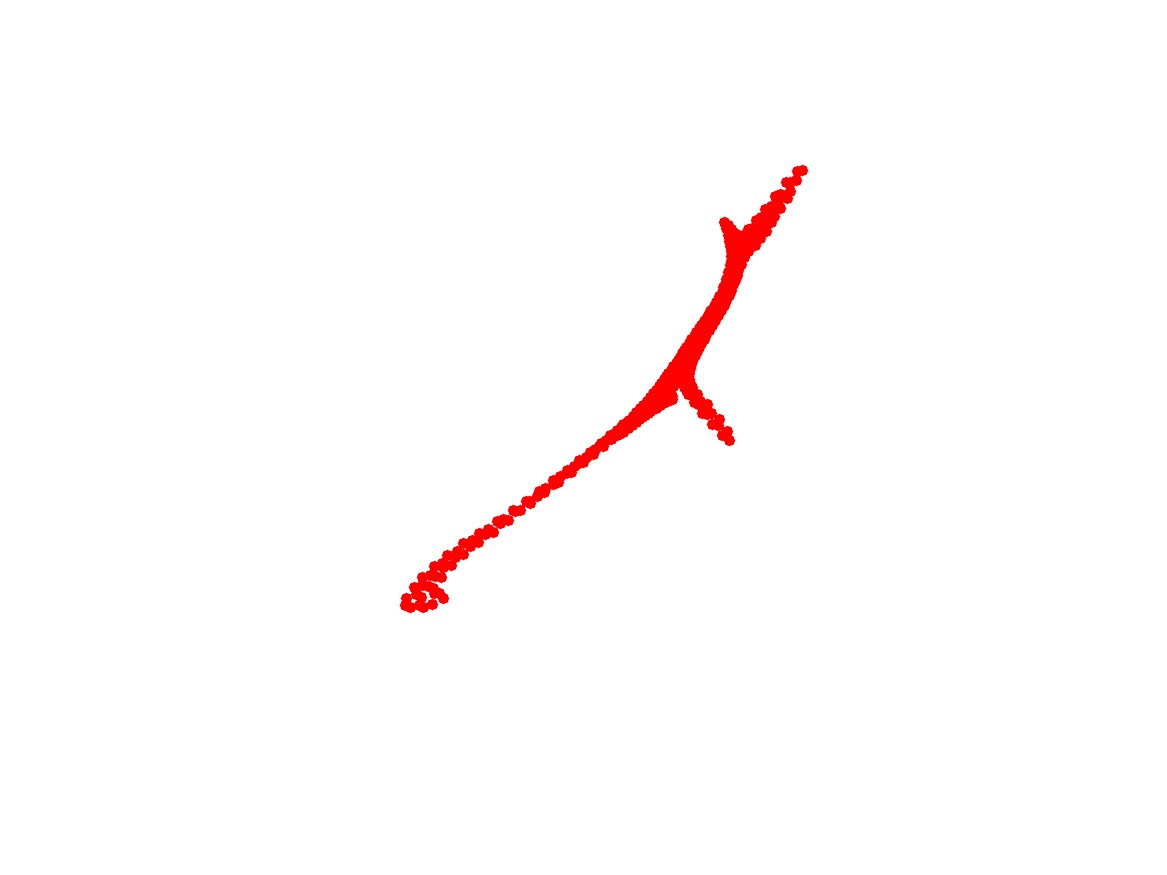}} & 
        {\includegraphics[trim={6cm 4cm 6cm 4cm},clip,width=0.1\textwidth]{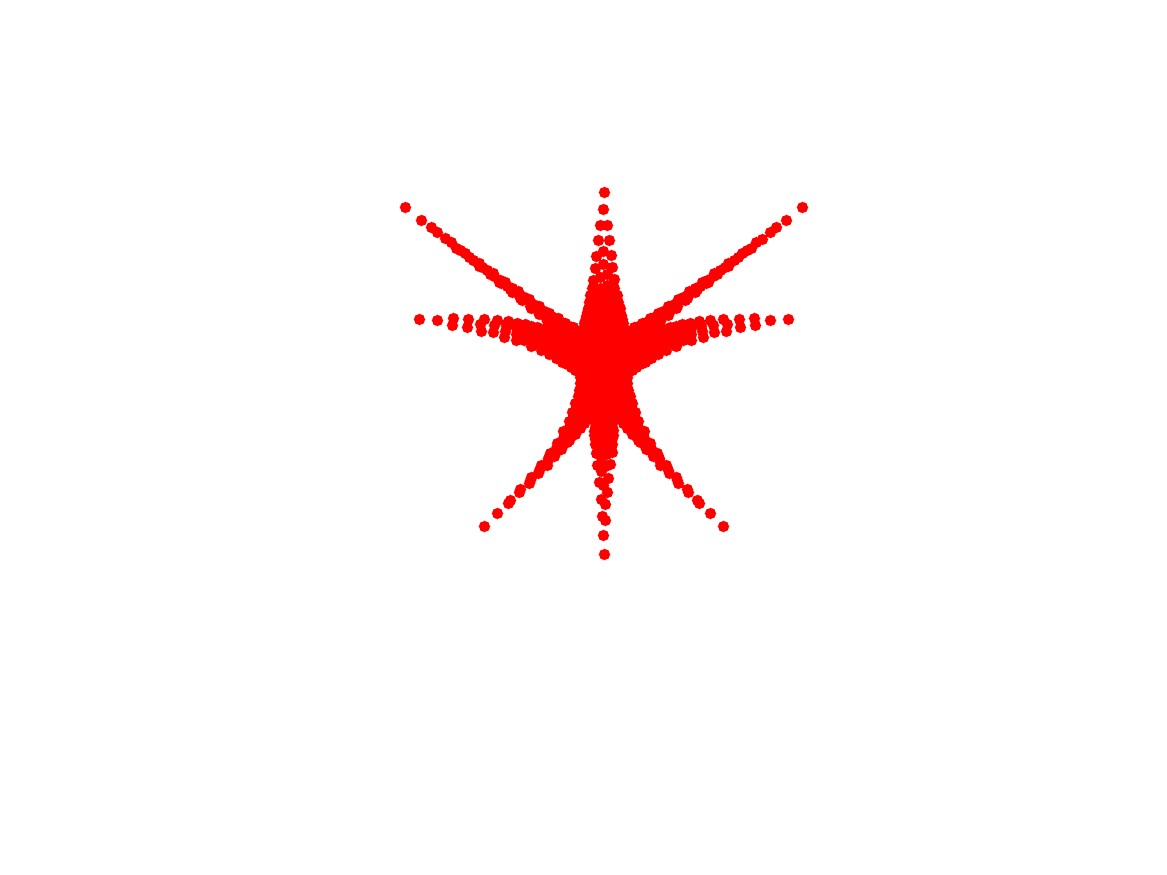}} & & & 
        {\includegraphics[trim={6cm 4cm 6cm 4cm},clip,width=0.1\textwidth]{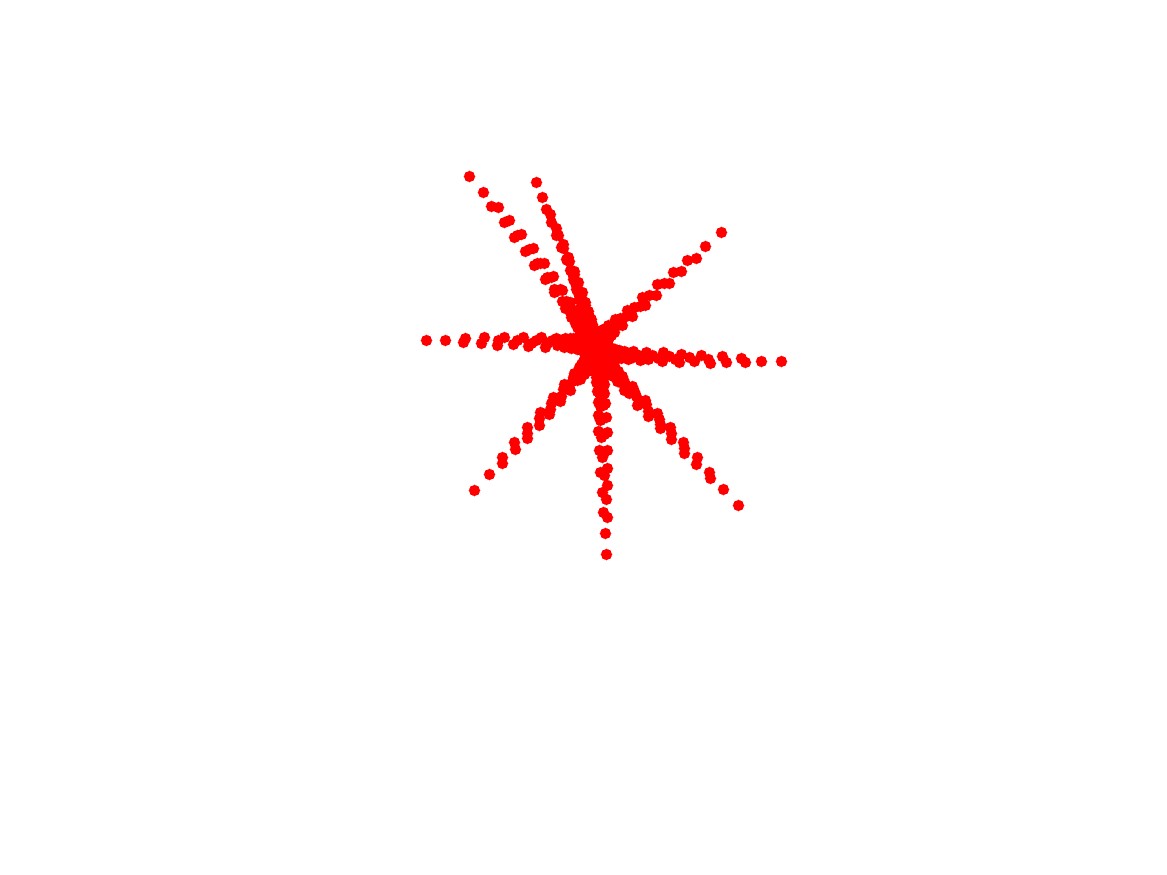}} \\
        & {\includegraphics[trim={6cm 4cm 6cm 4cm},clip,width=0.1\textwidth]{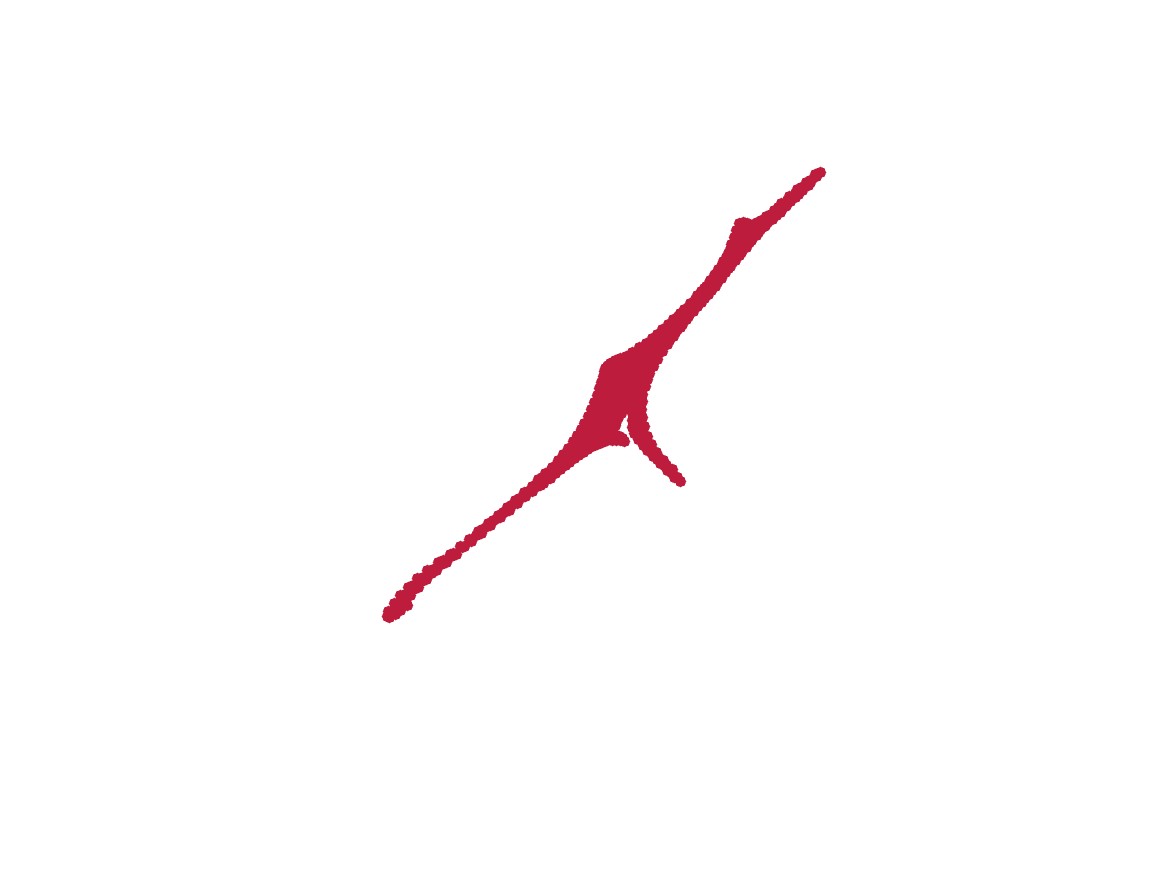}} & & 
        {\includegraphics[trim={6cm 4cm 6cm 4cm},clip,width=0.1\textwidth]{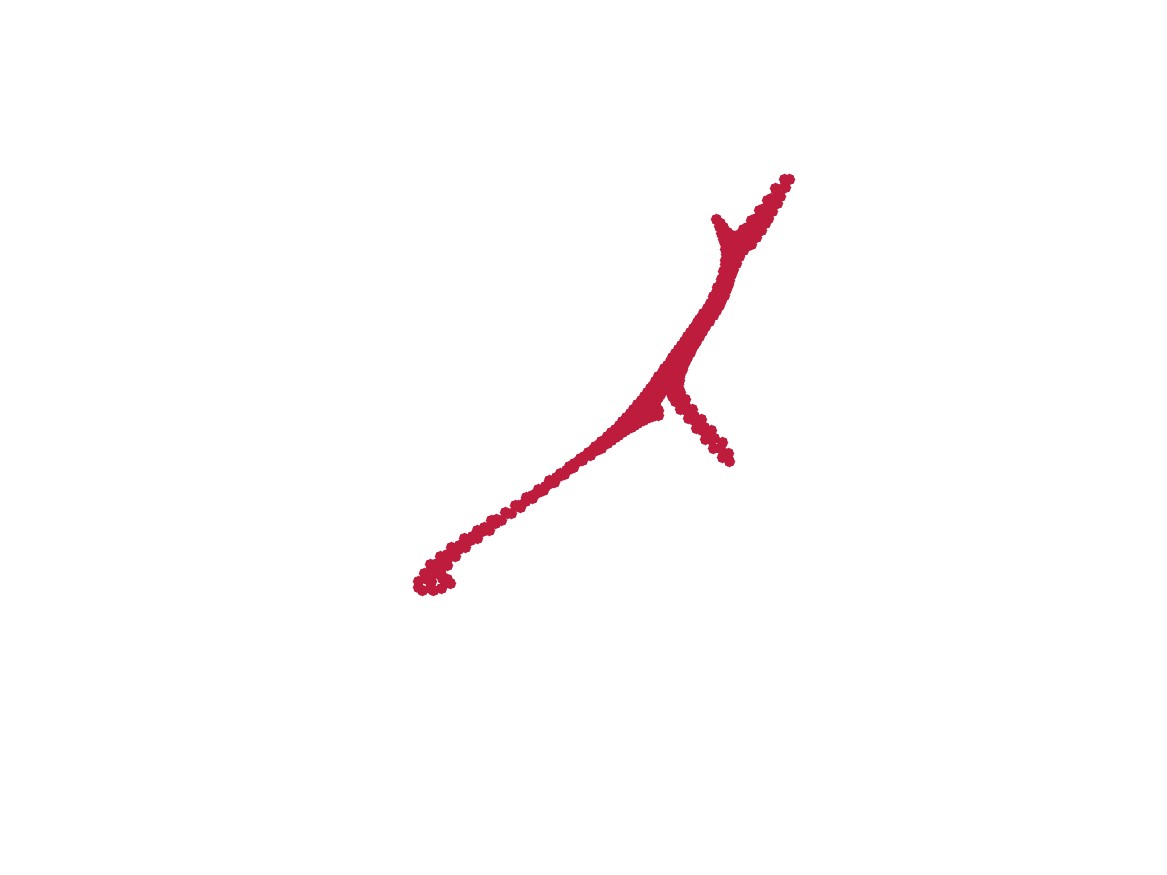}} & 
        {\includegraphics[trim={6cm 4cm 6cm 4cm},clip,width=0.1\textwidth]{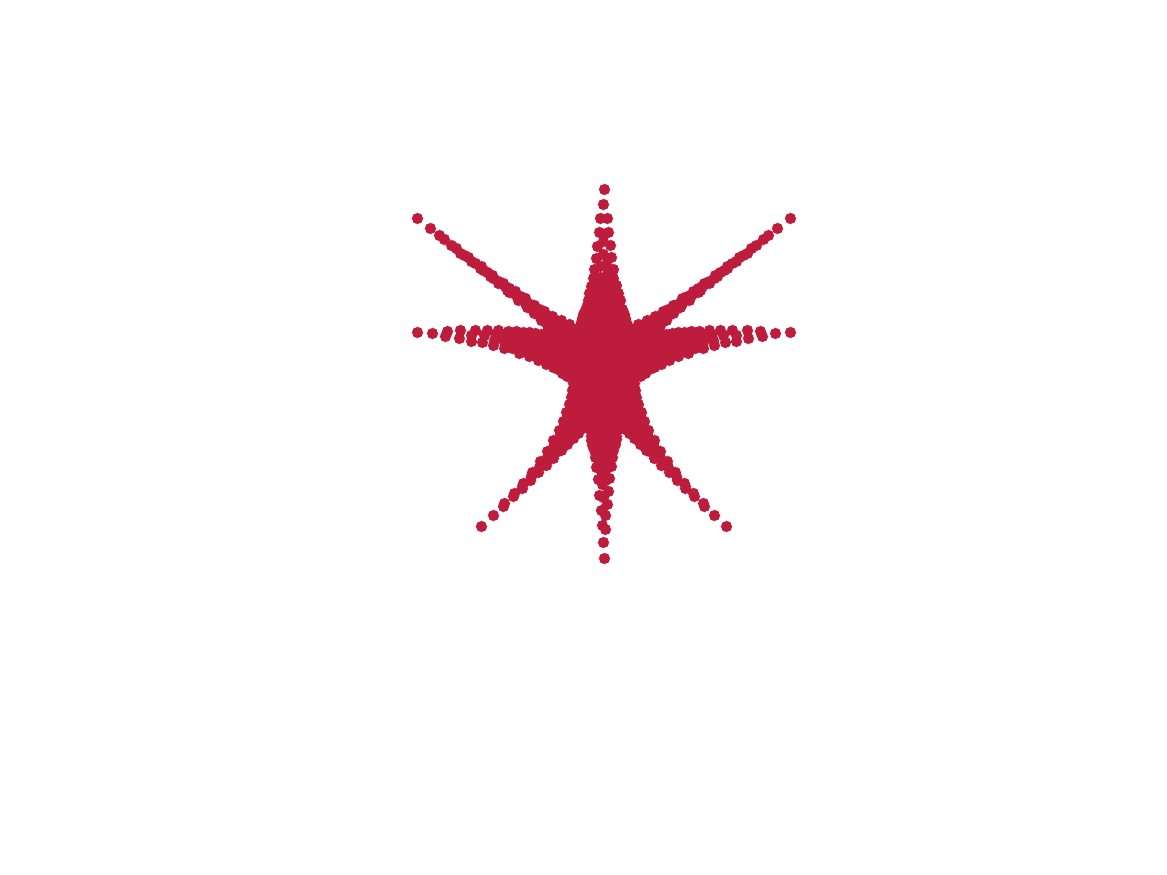}} & & 
        {\includegraphics[trim={6cm 4cm 6cm 4cm},clip,width=0.1\textwidth]{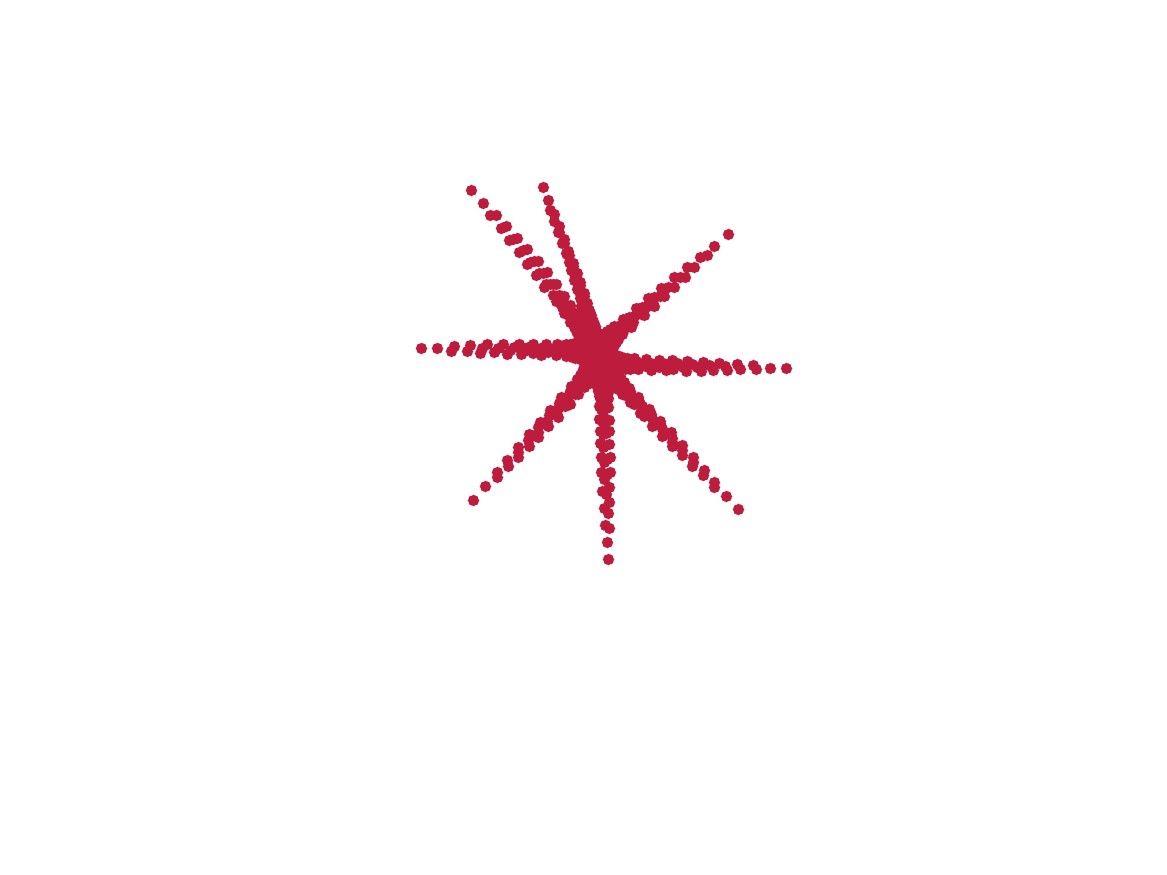}} & \\
        & &  {\includegraphics[trim={6cm 4cm 6cm 4cm},clip,width=0.1\textwidth]{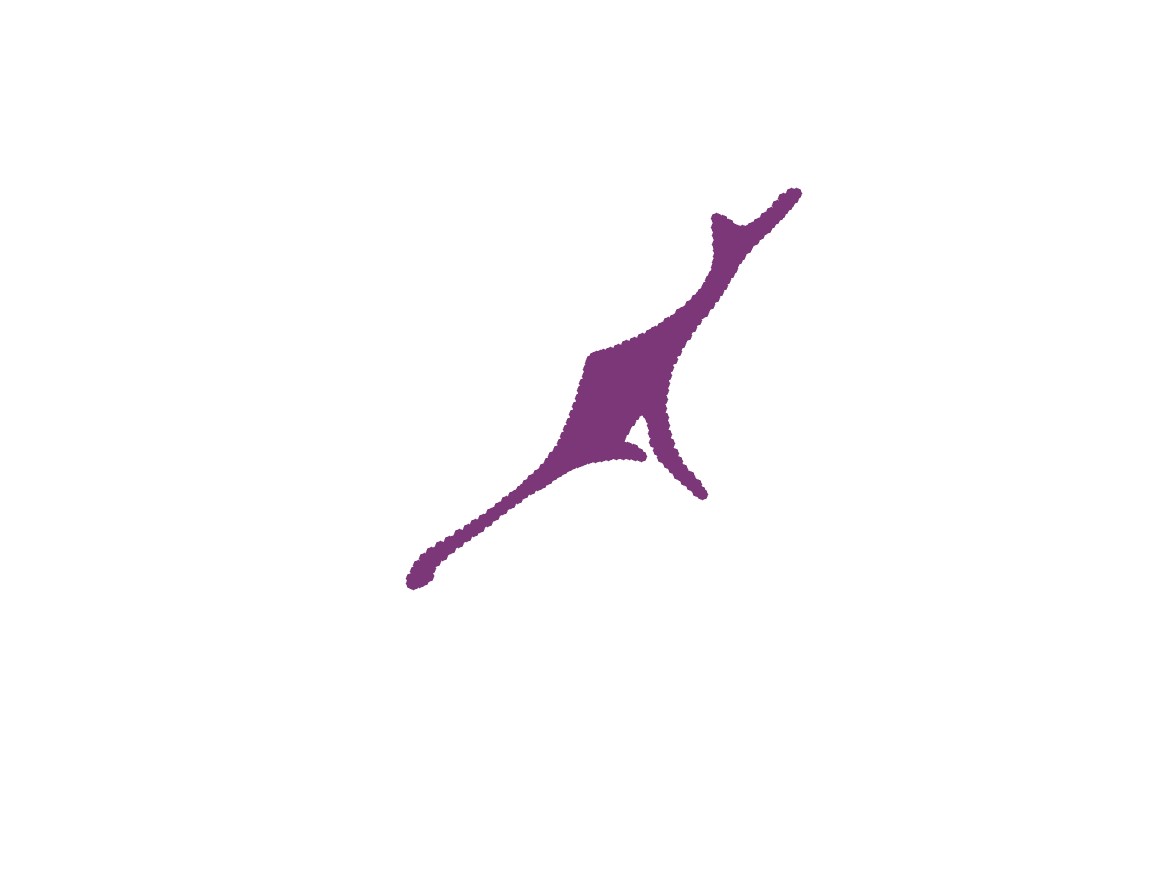}} & 
        {\includegraphics[trim={6cm 4cm 6cm 4cm},clip,width=0.1\textwidth]{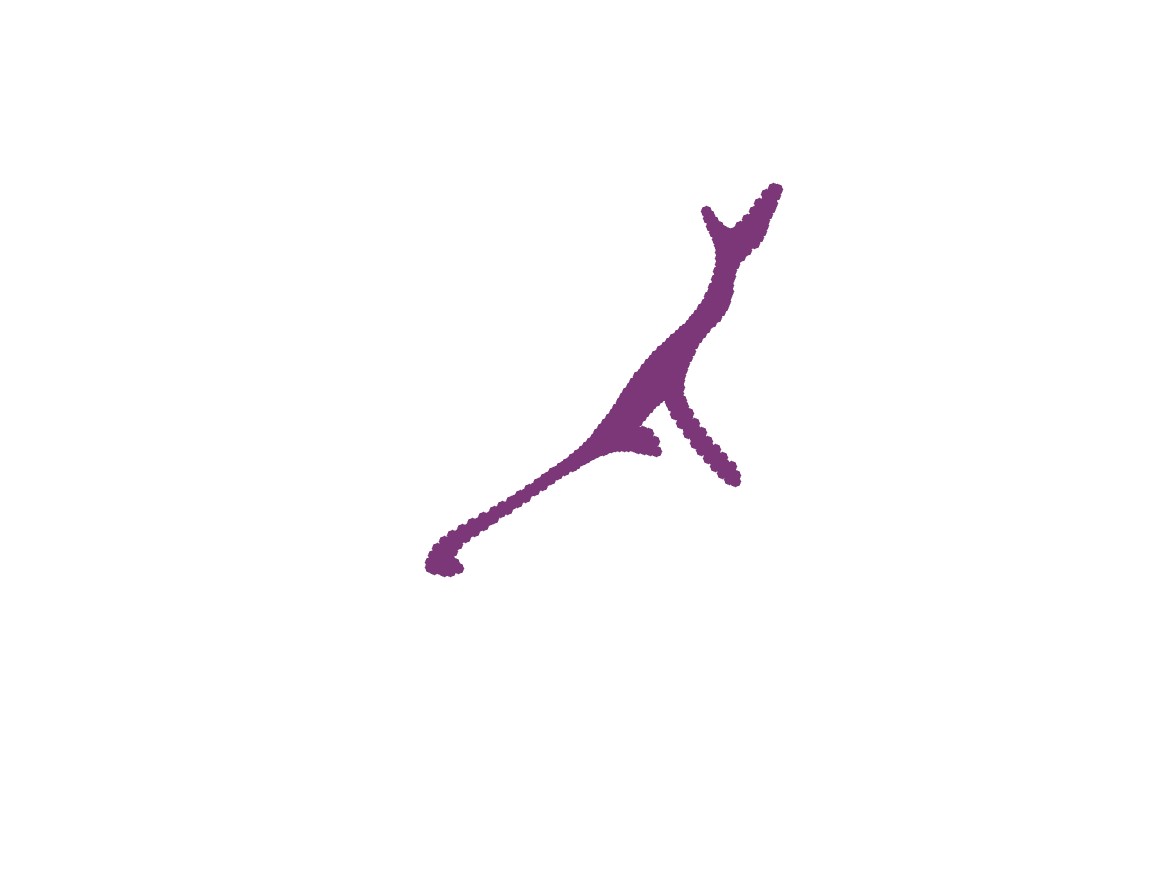}} & 
        {\includegraphics[trim={6cm 4cm 6cm 4cm},clip,width=0.1\textwidth]{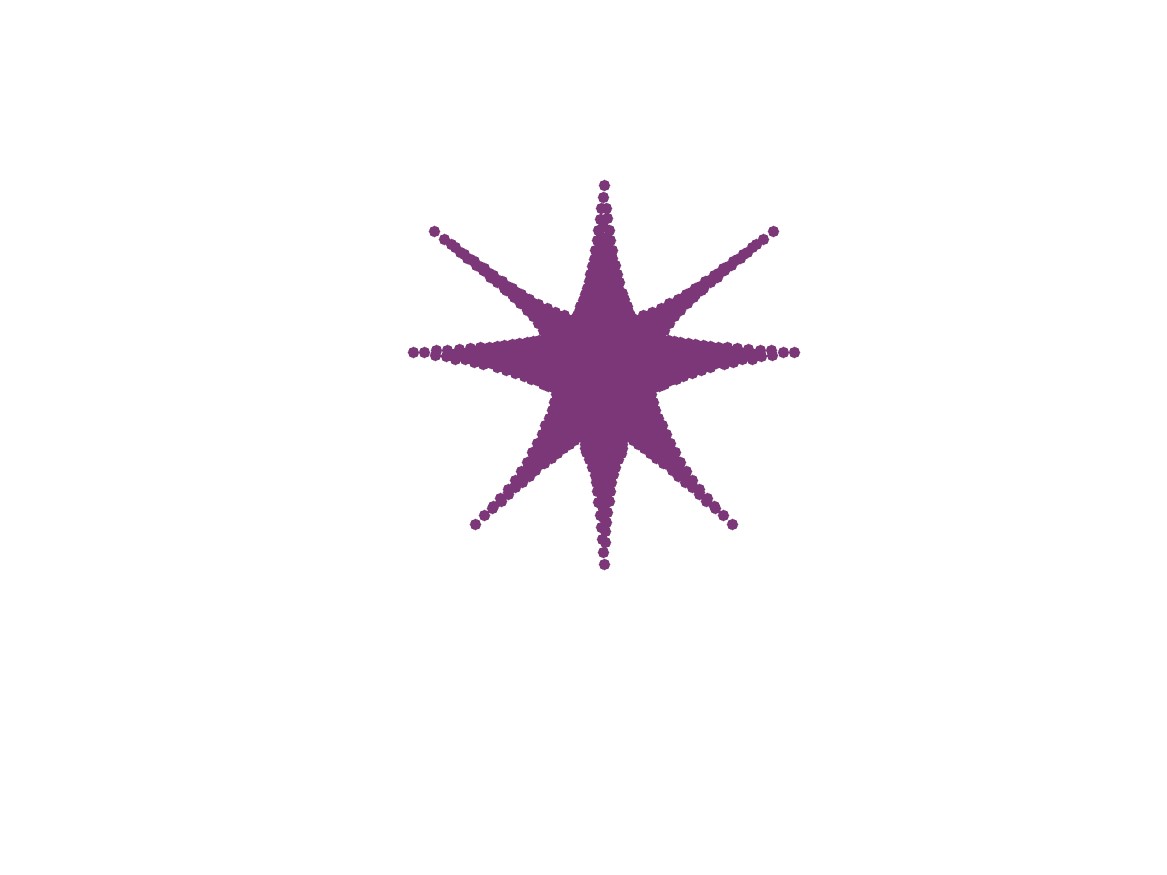}} & 
        {\includegraphics[trim={6cm 4cm 6cm 4cm},clip,width=0.1\textwidth]{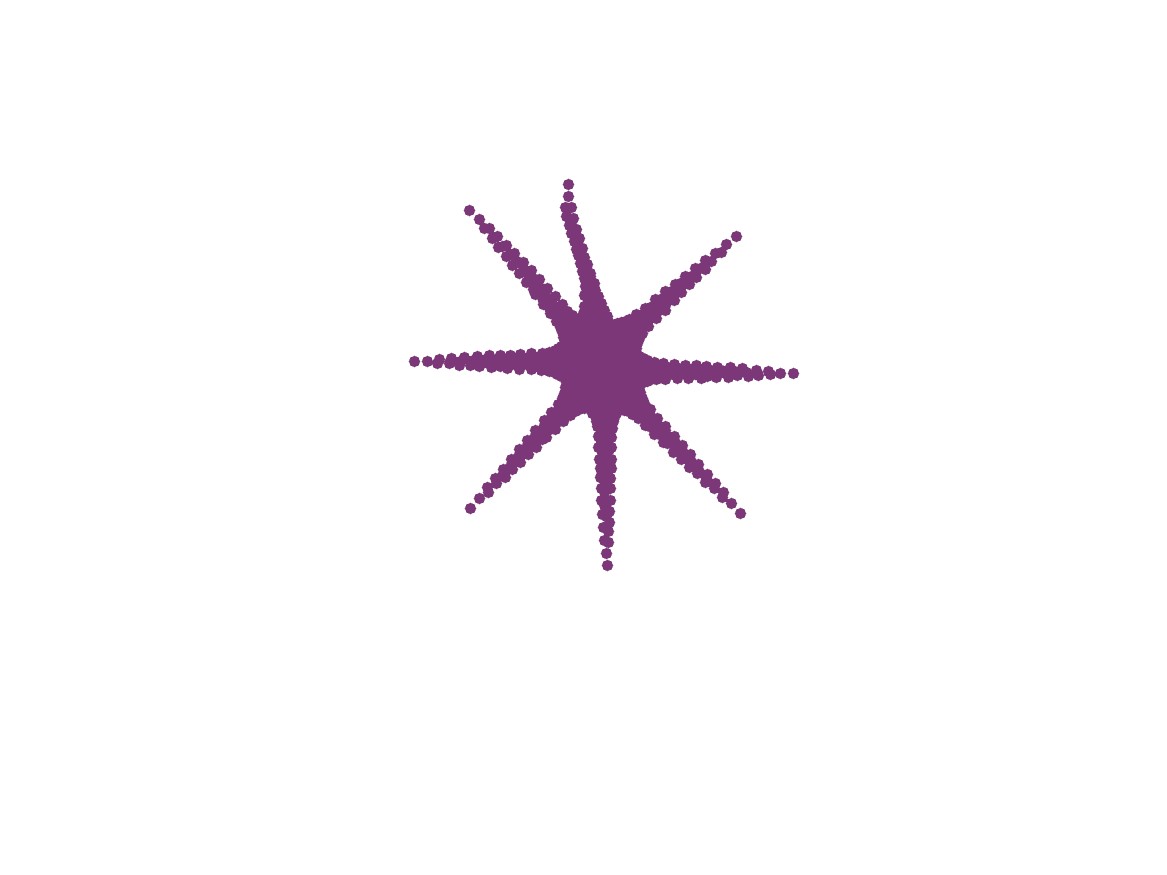}} & & \\
        {\includegraphics[trim={6cm 4cm 6cm 4cm},clip,width=0.1\textwidth]{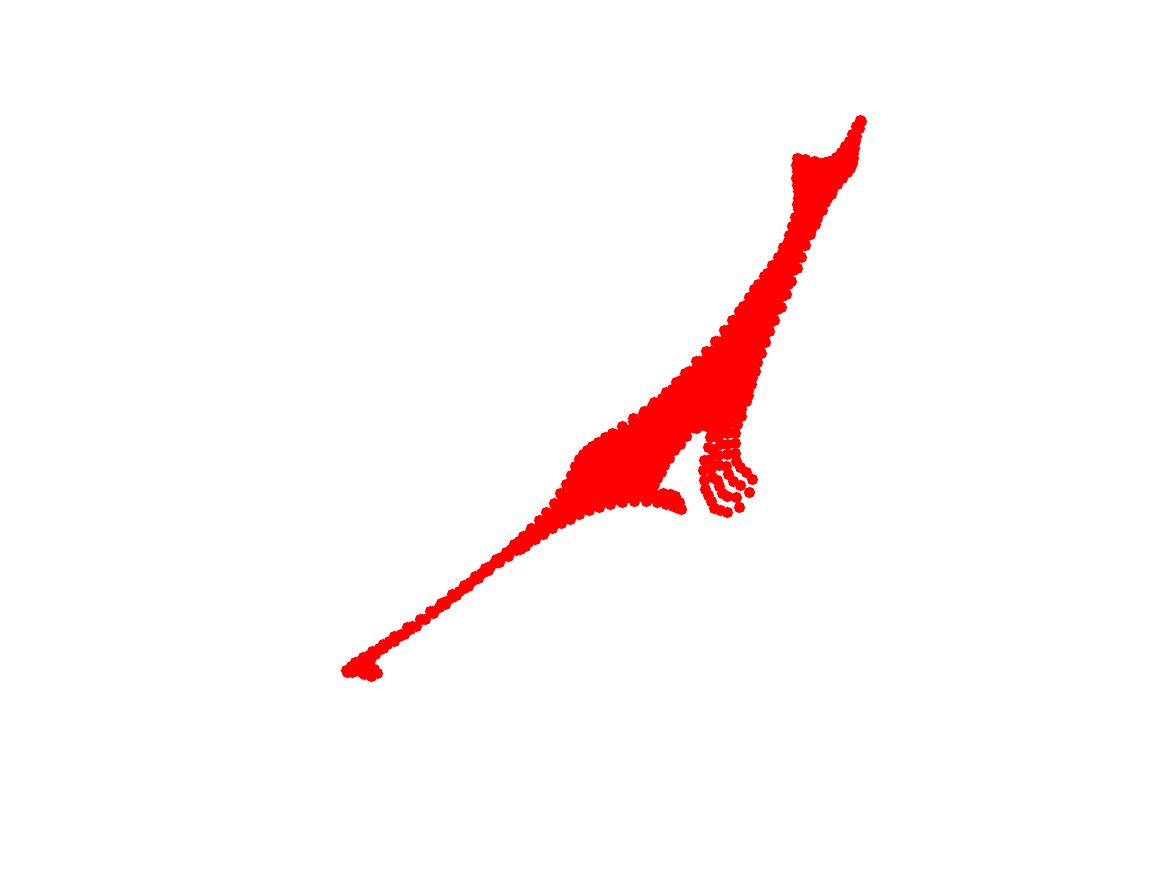}} & {\includegraphics[trim={6cm 4cm 6cm 4cm},clip,width=0.1\textwidth]{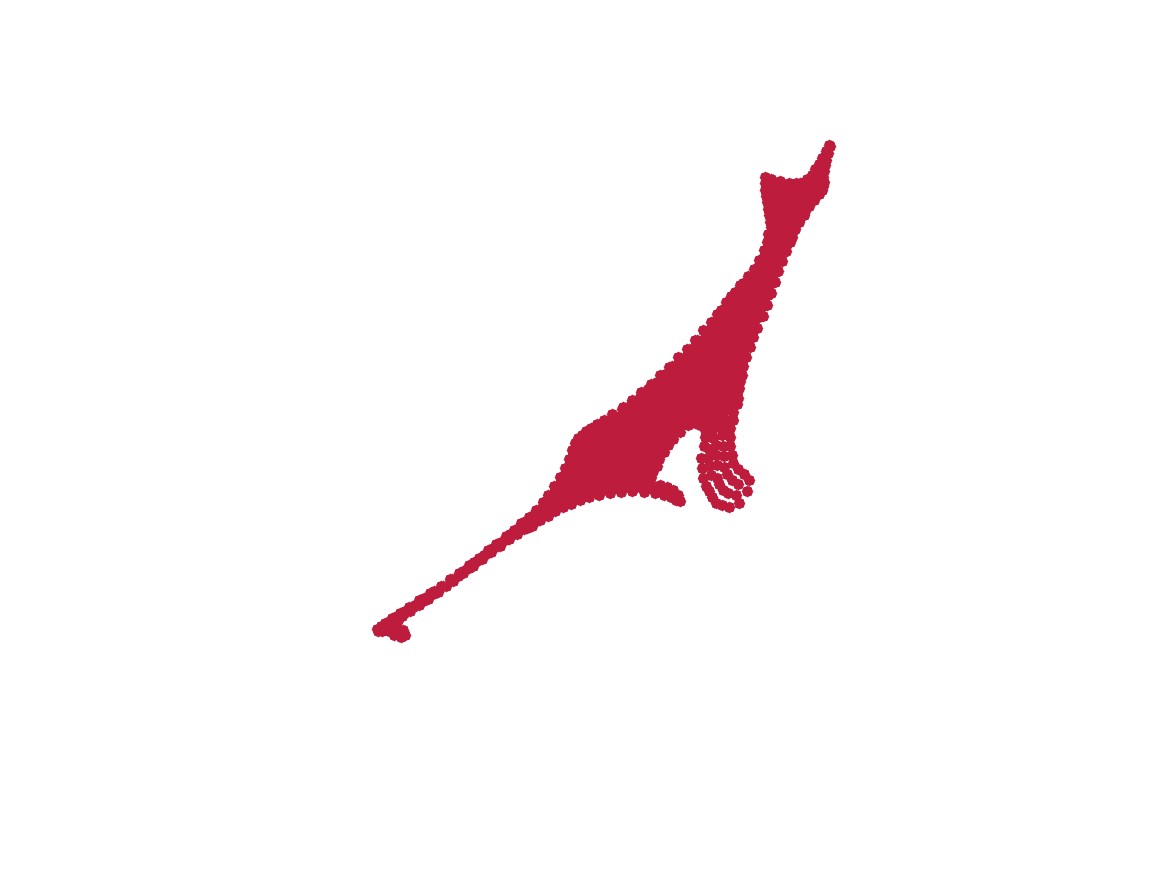}} & 
        {\includegraphics[trim={6cm 4cm 6cm 4cm},clip,width=0.1\textwidth]{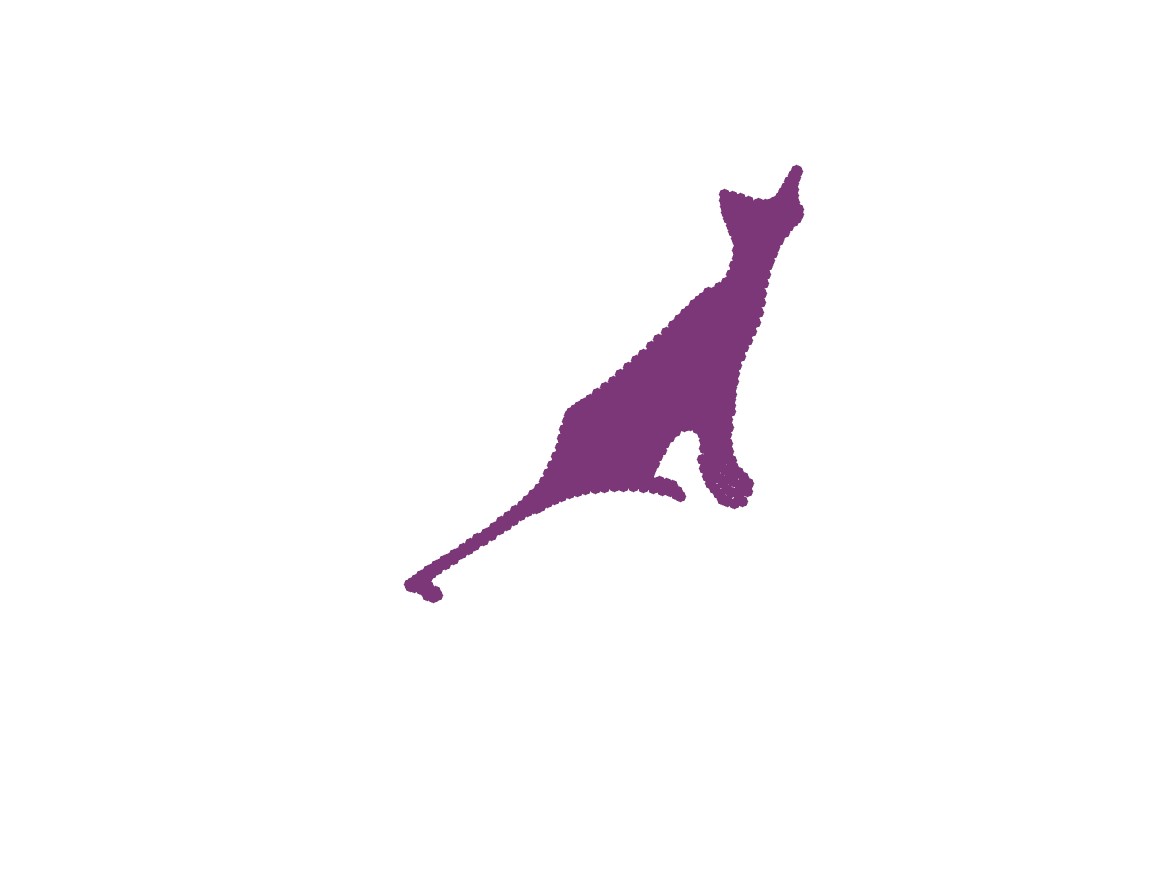}} & 
        {\includegraphics[trim={6cm 4cm 6cm 4cm},clip,width=0.1\textwidth]{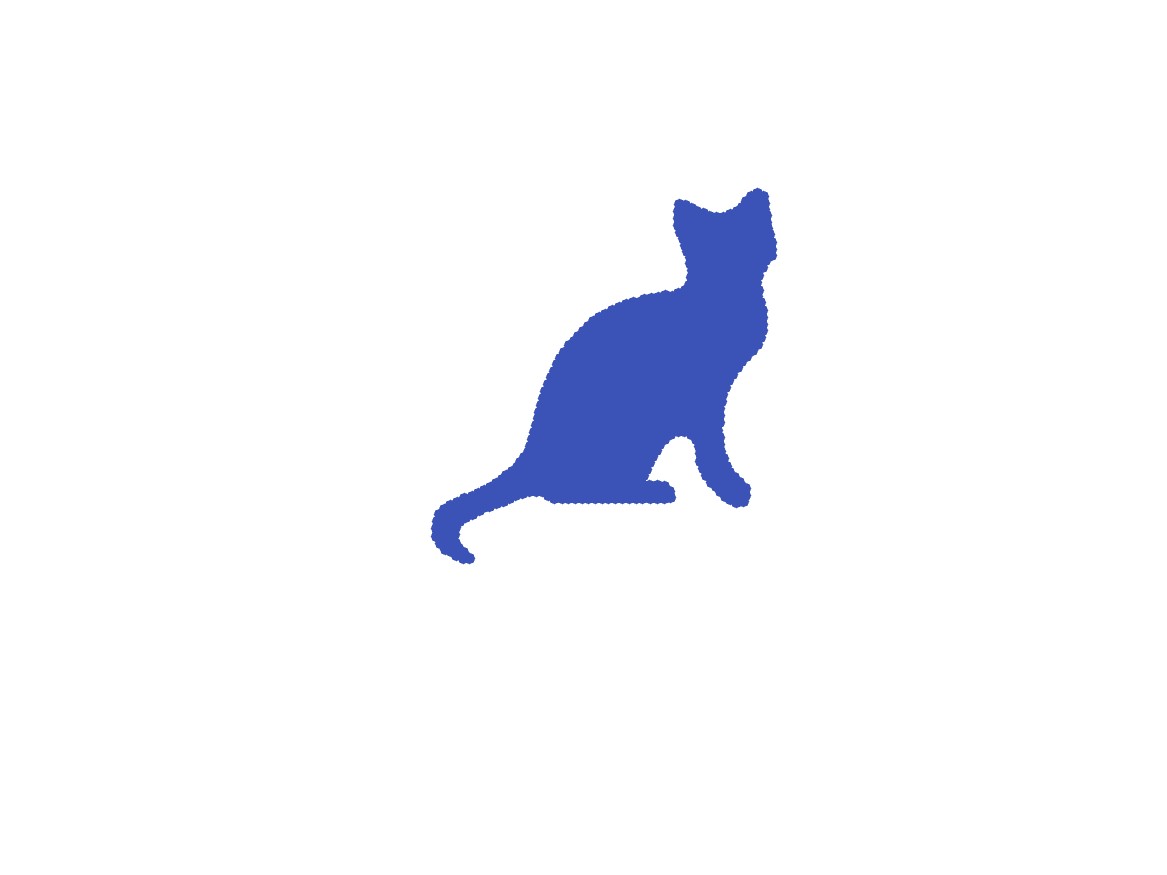}} & 
        {\includegraphics[trim={6cm 4cm 6cm 4cm},clip,width=0.1\textwidth]{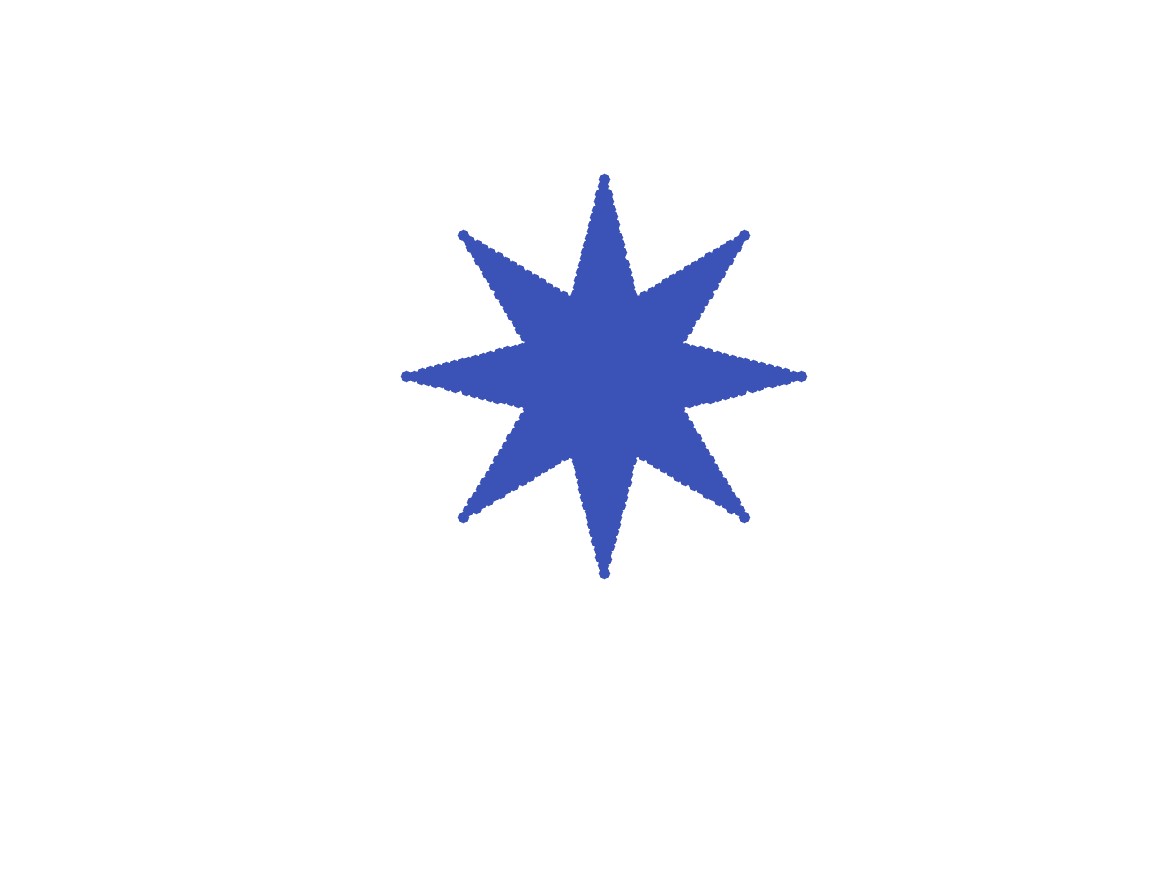}} & 
        {\includegraphics[trim={6cm 4cm 6cm 4cm},clip,width=0.1\textwidth]{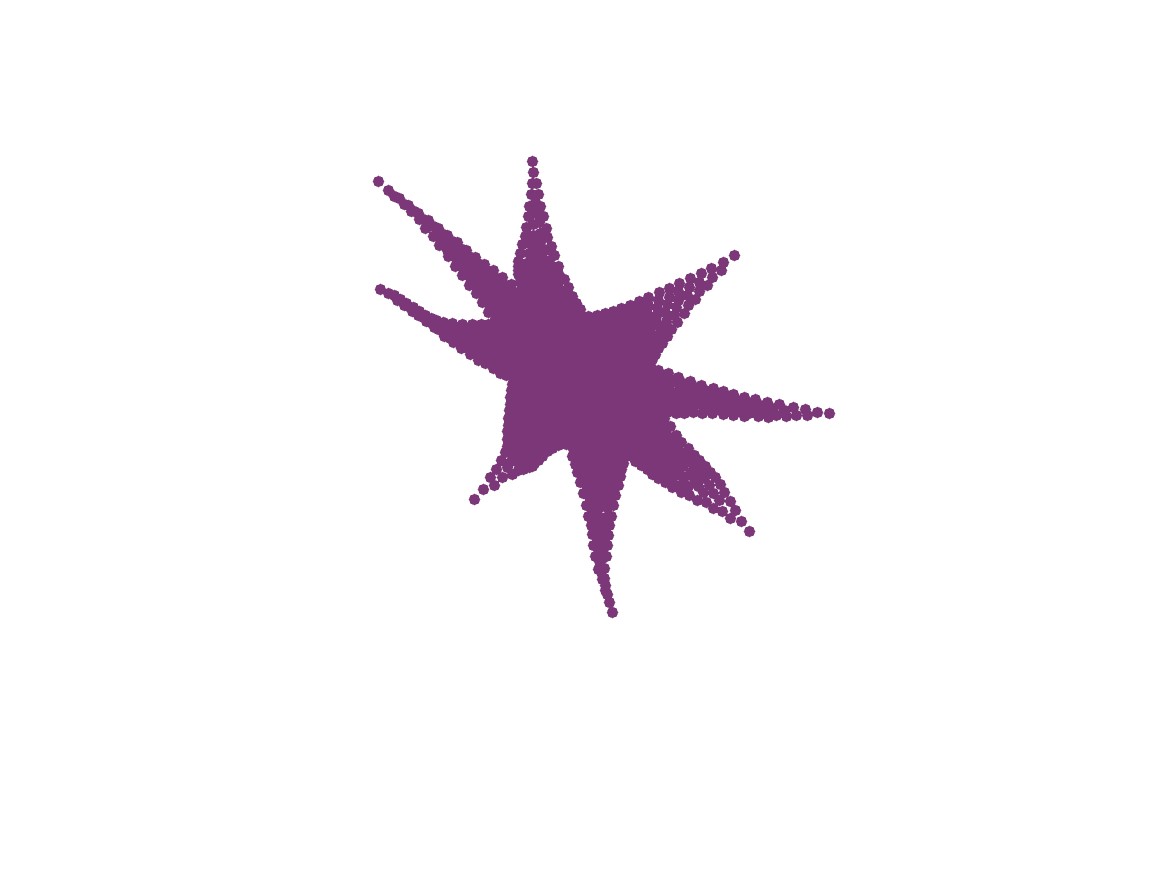}} & 
        {\includegraphics[trim={6cm 4cm 6cm 4cm},clip,width=0.1\textwidth]{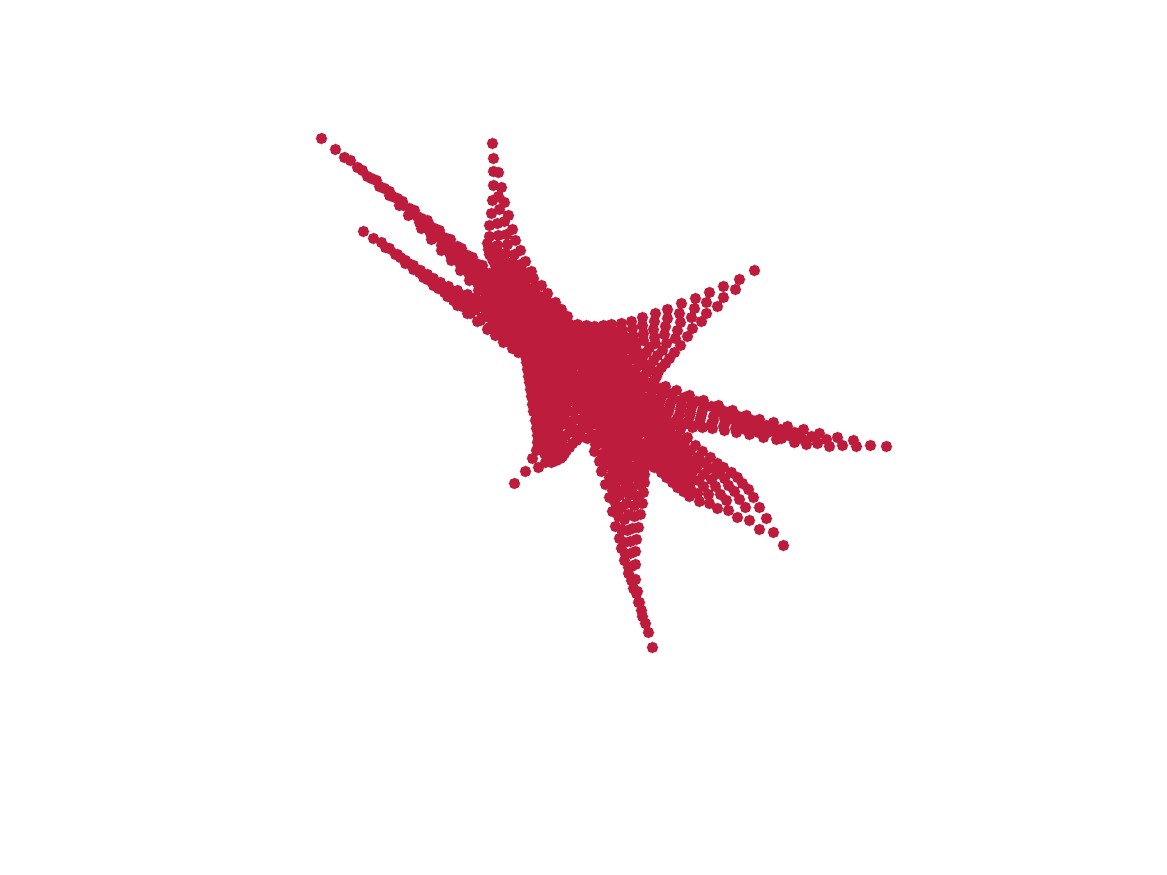}} & {\includegraphics[trim={6cm 4cm 6cm 4cm},clip,width=0.1\textwidth]{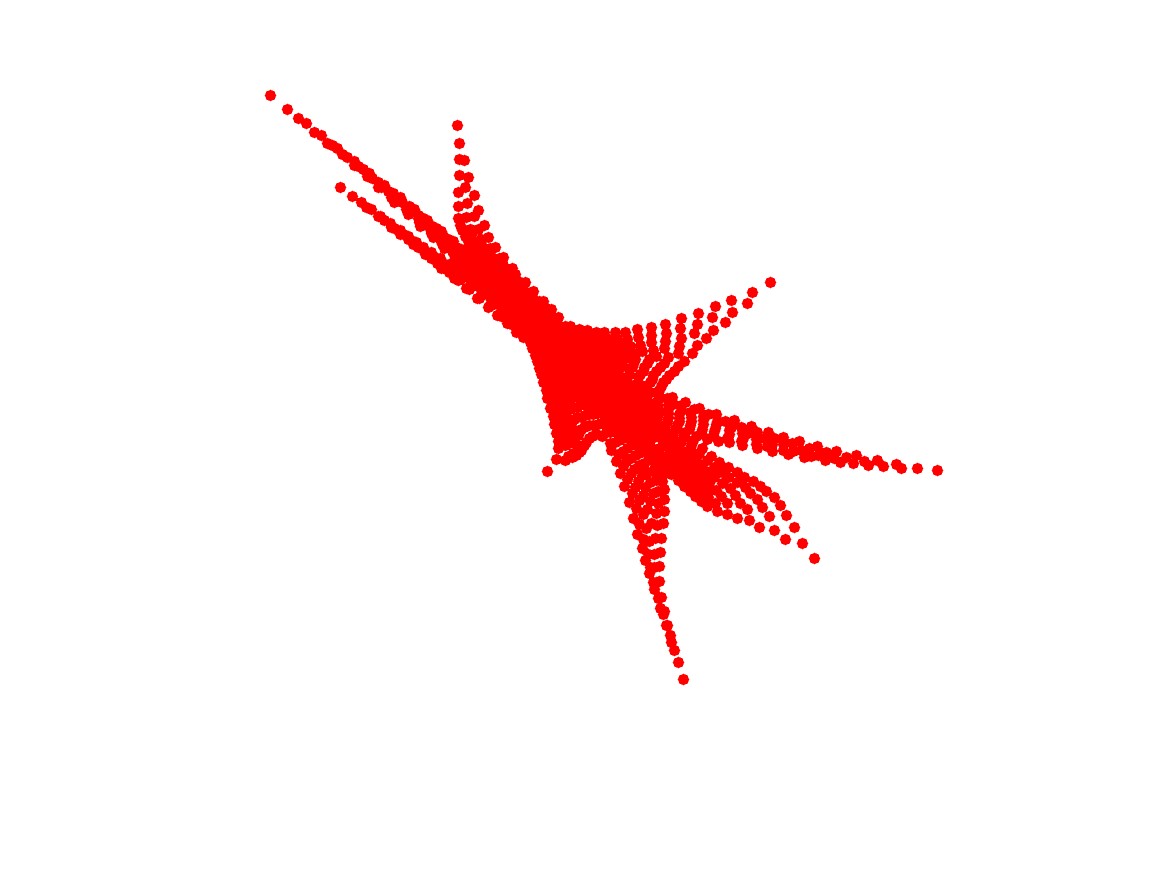}} \\
        {\includegraphics[trim={6cm 4cm 6cm 4cm},clip,width=0.1\textwidth]{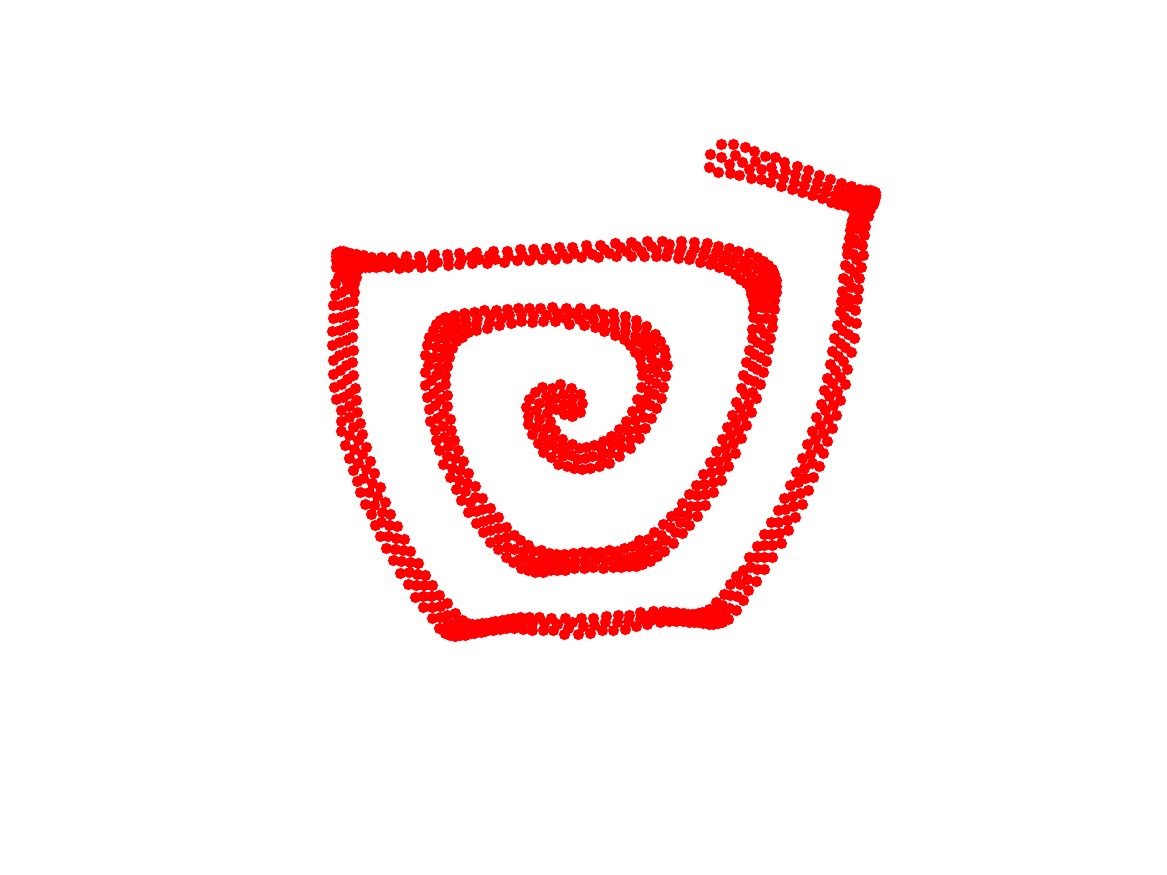}} & {\includegraphics[trim={6cm 4cm 6cm 4cm},clip,width=0.1\textwidth]{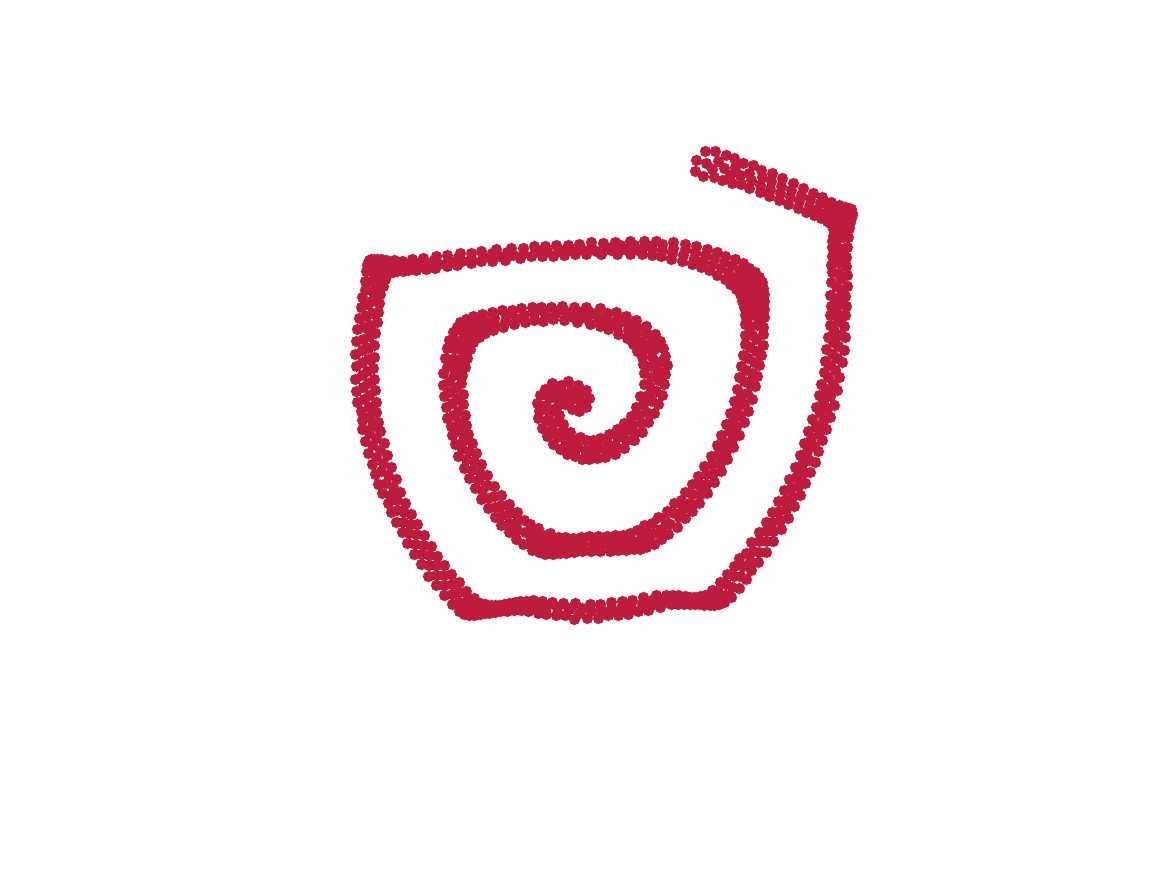}} & 
        {\includegraphics[trim={6cm 4cm 6cm 4cm},clip,width=0.1\textwidth]{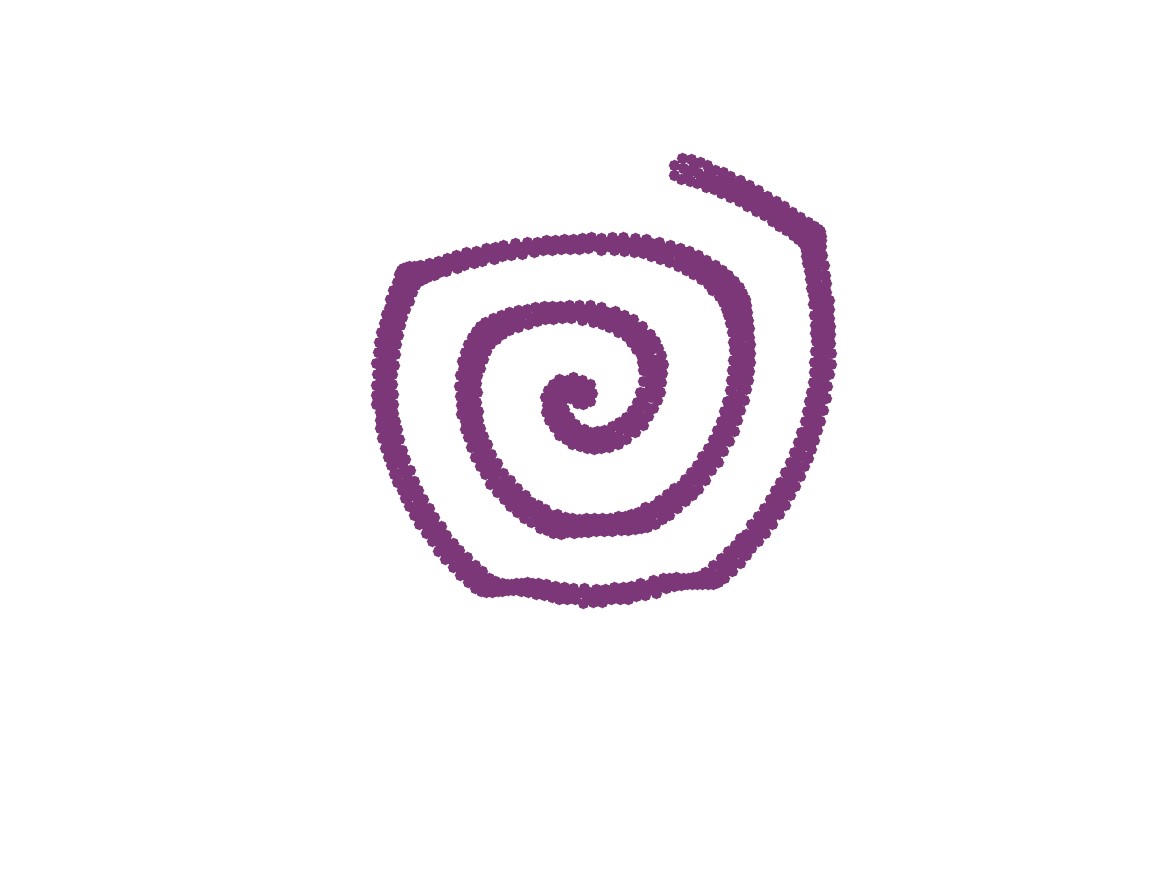}} & 
        {\includegraphics[trim={6cm 4cm 6cm 4cm},clip,width=0.1\textwidth]{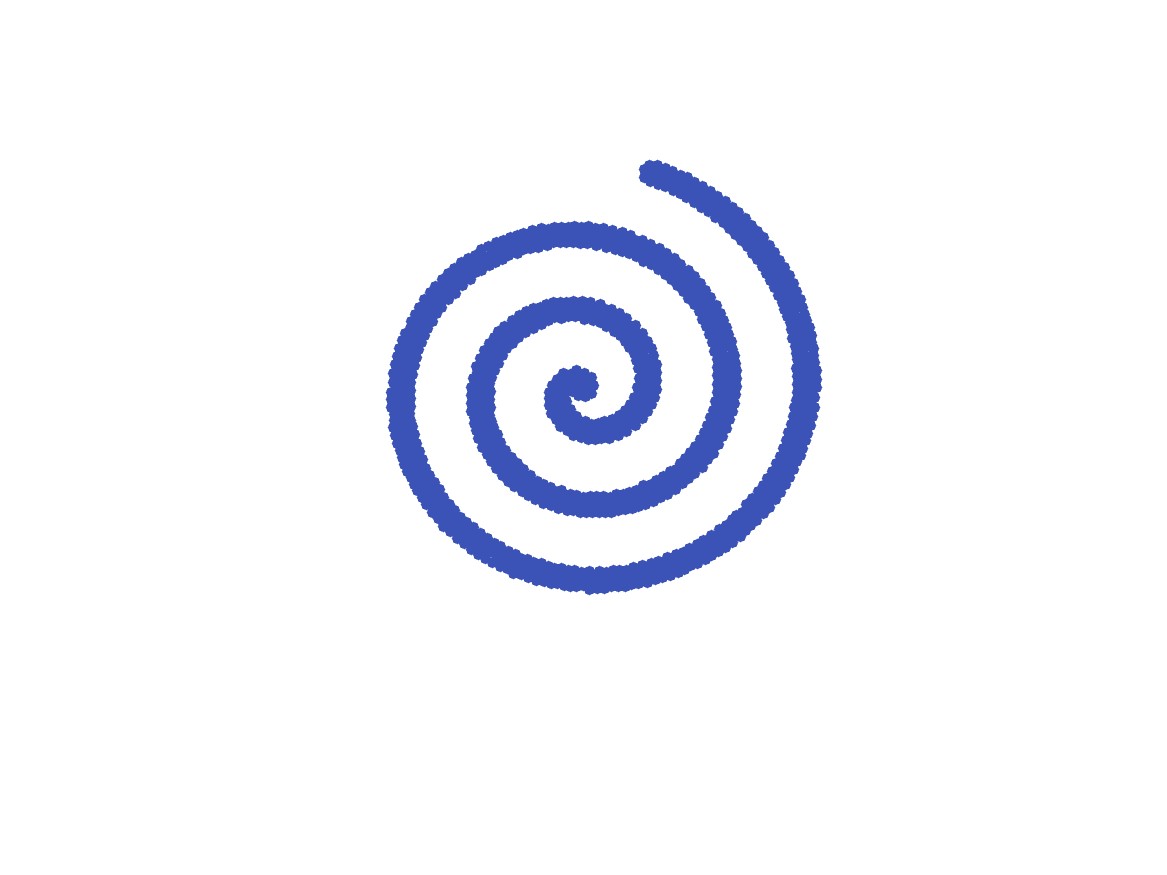}} & 
        {\includegraphics[trim={6cm 4cm 6cm 4cm},clip,width=0.1\textwidth]{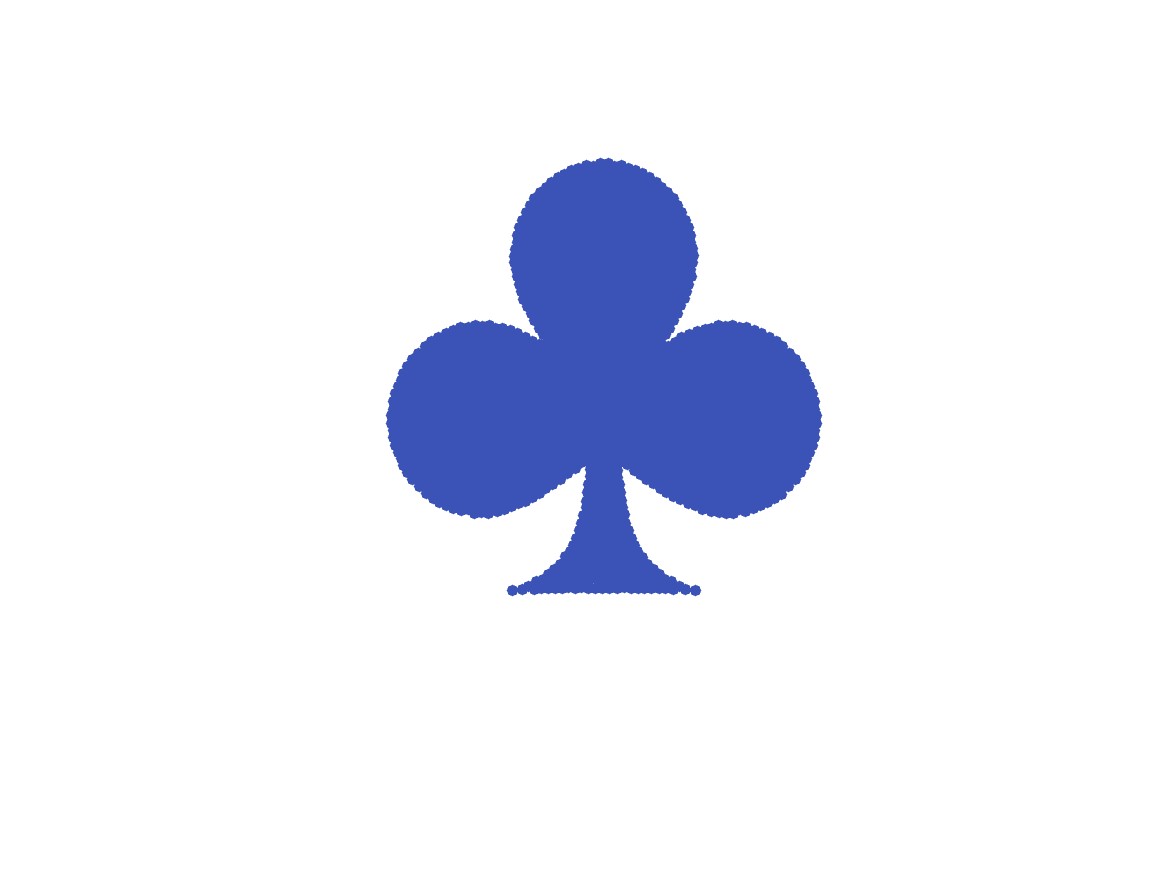}} & 
        {\includegraphics[trim={6cm 4cm 6cm 4cm},clip,width=0.1\textwidth]{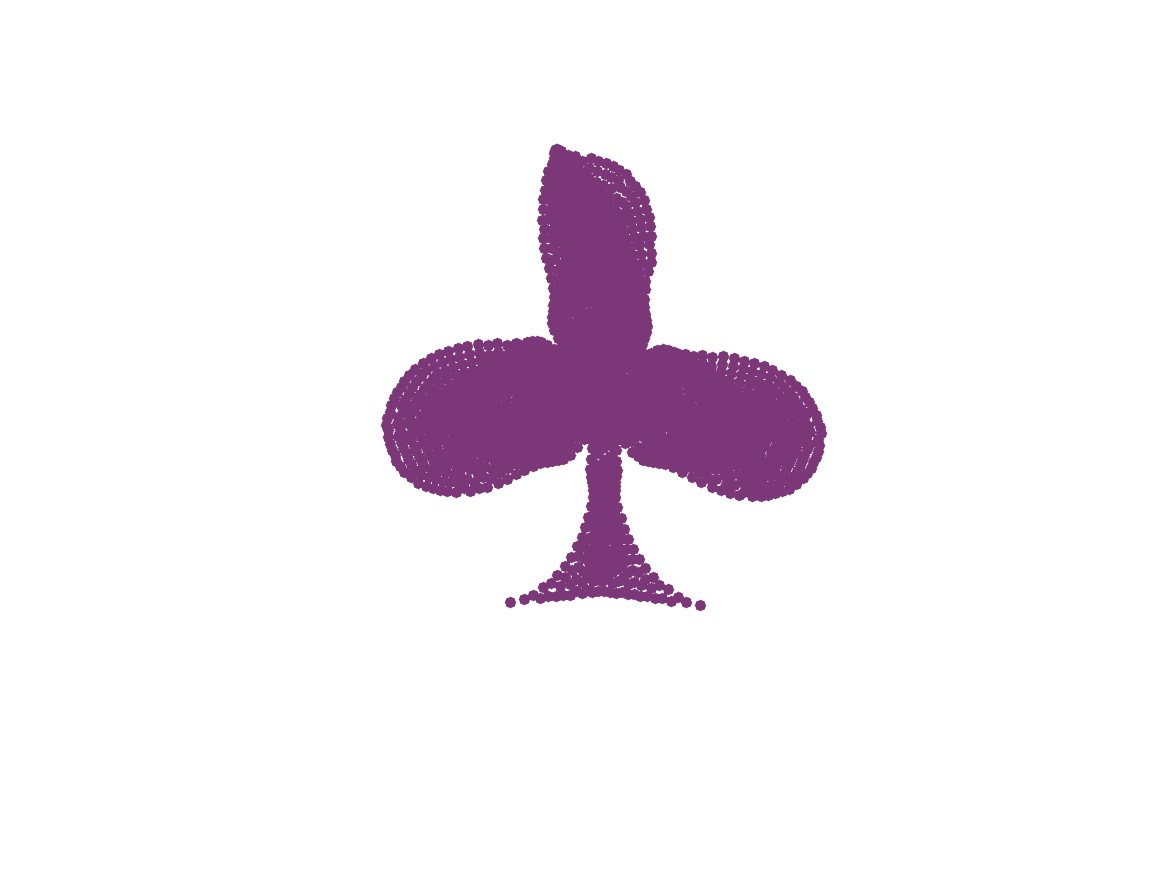}} & 
        {\includegraphics[trim={6cm 4cm 6cm 4cm},clip,width=0.1\textwidth]{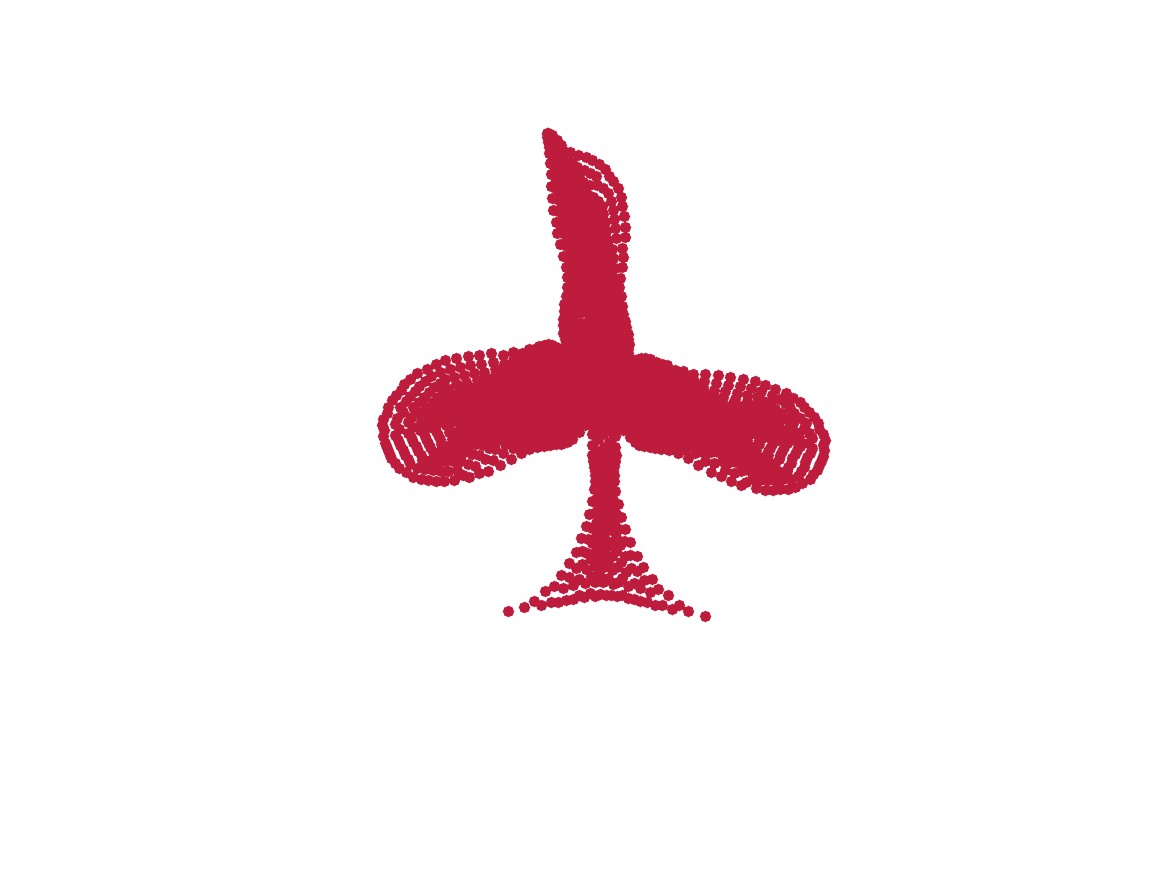}} & {\includegraphics[trim={6cm 4cm 6cm 4cm},clip,width=0.1\textwidth]{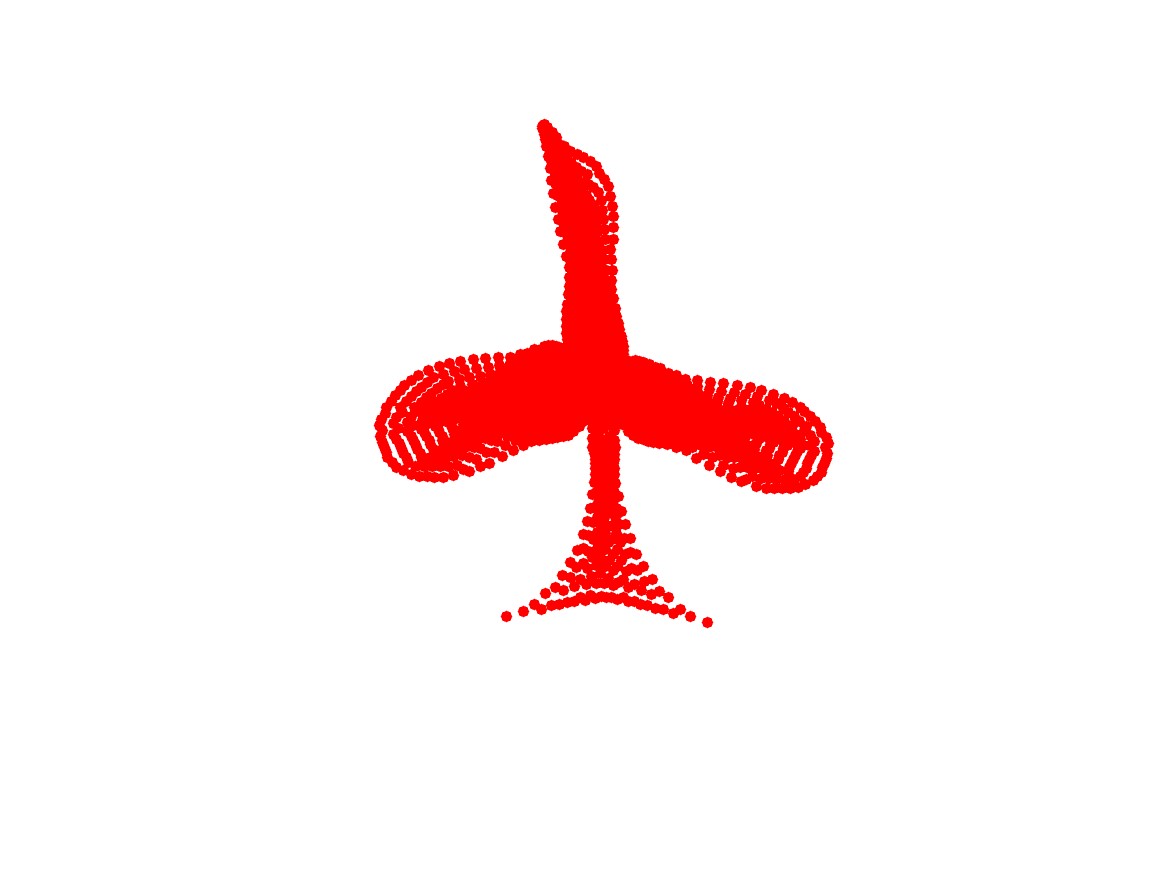}} \\
        & & {\includegraphics[trim={6cm 4cm 6cm 4cm},clip,width=0.1\textwidth]{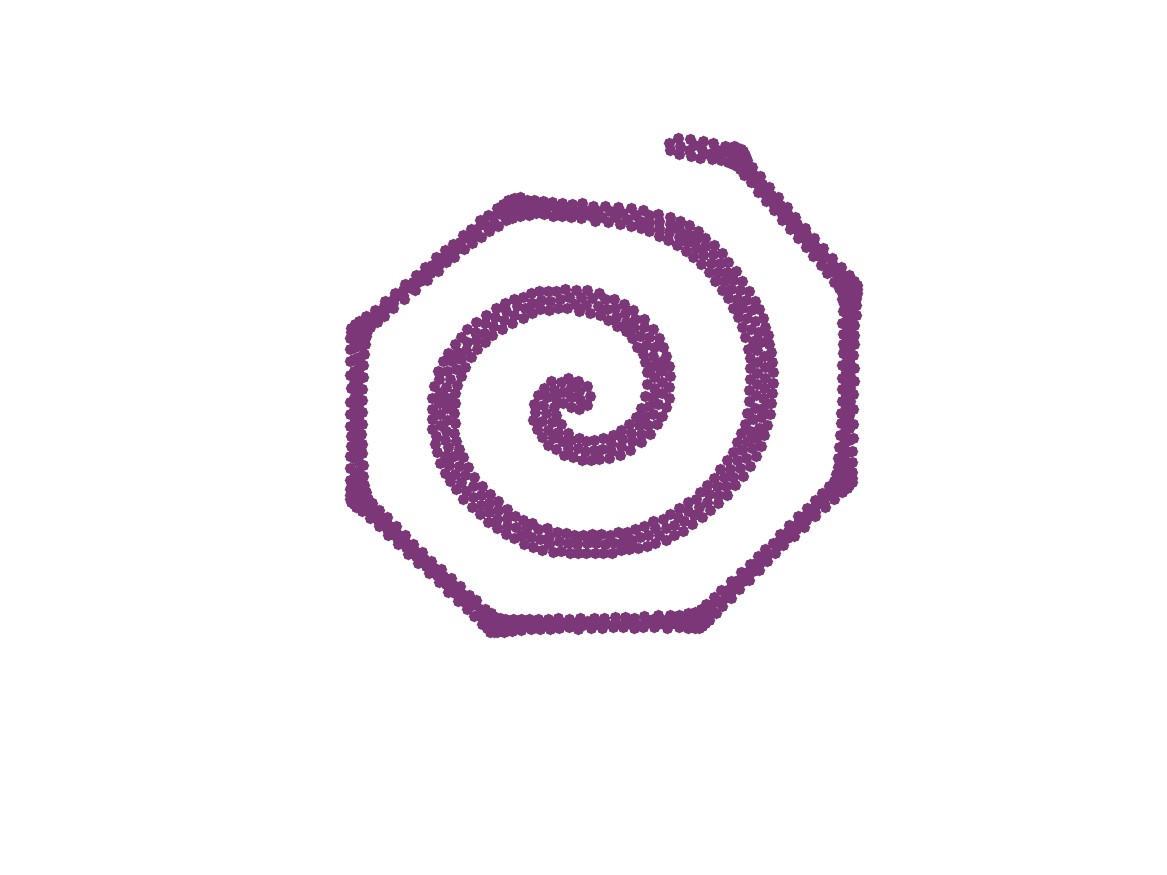}} & 
        {\includegraphics[trim={6cm 4cm 6cm 4cm},clip,width=0.1\textwidth]{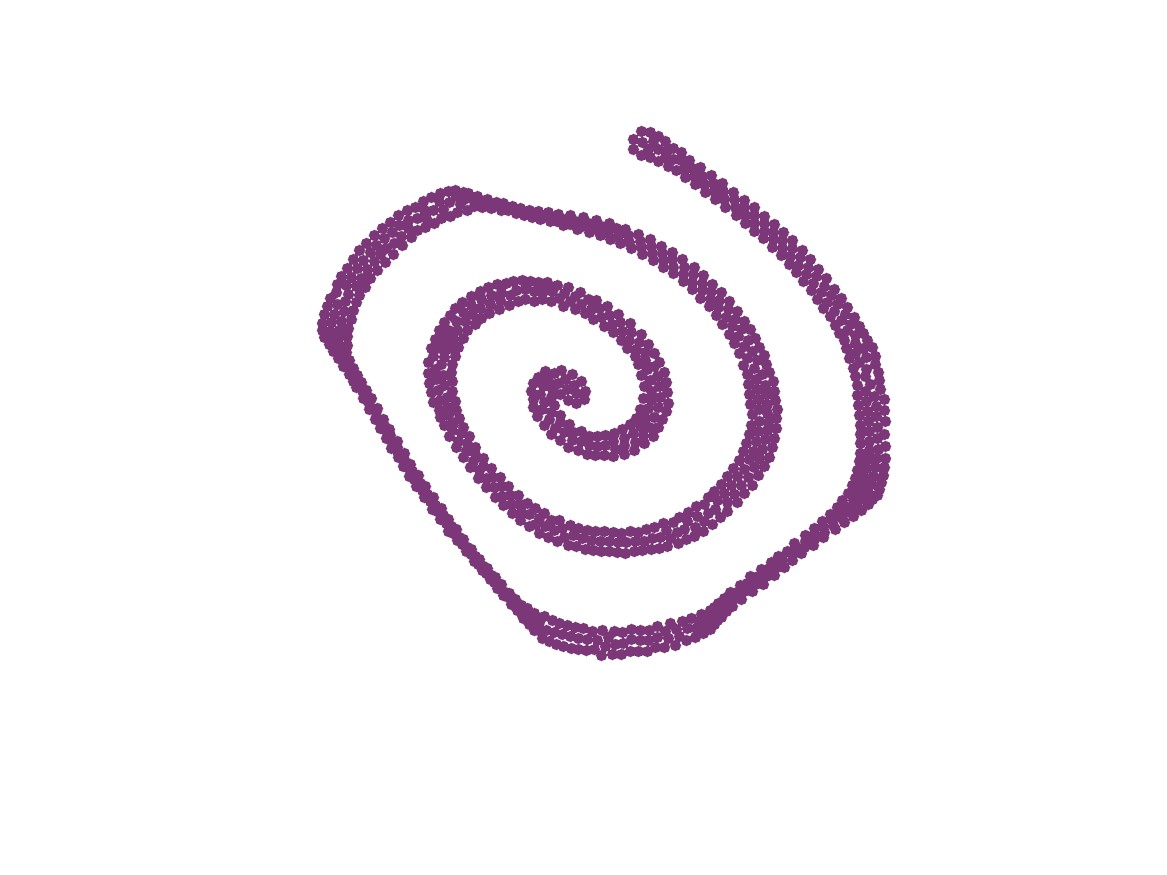}} & 
        {\includegraphics[trim={6cm 4cm 6cm 4cm},clip,width=0.1\textwidth]{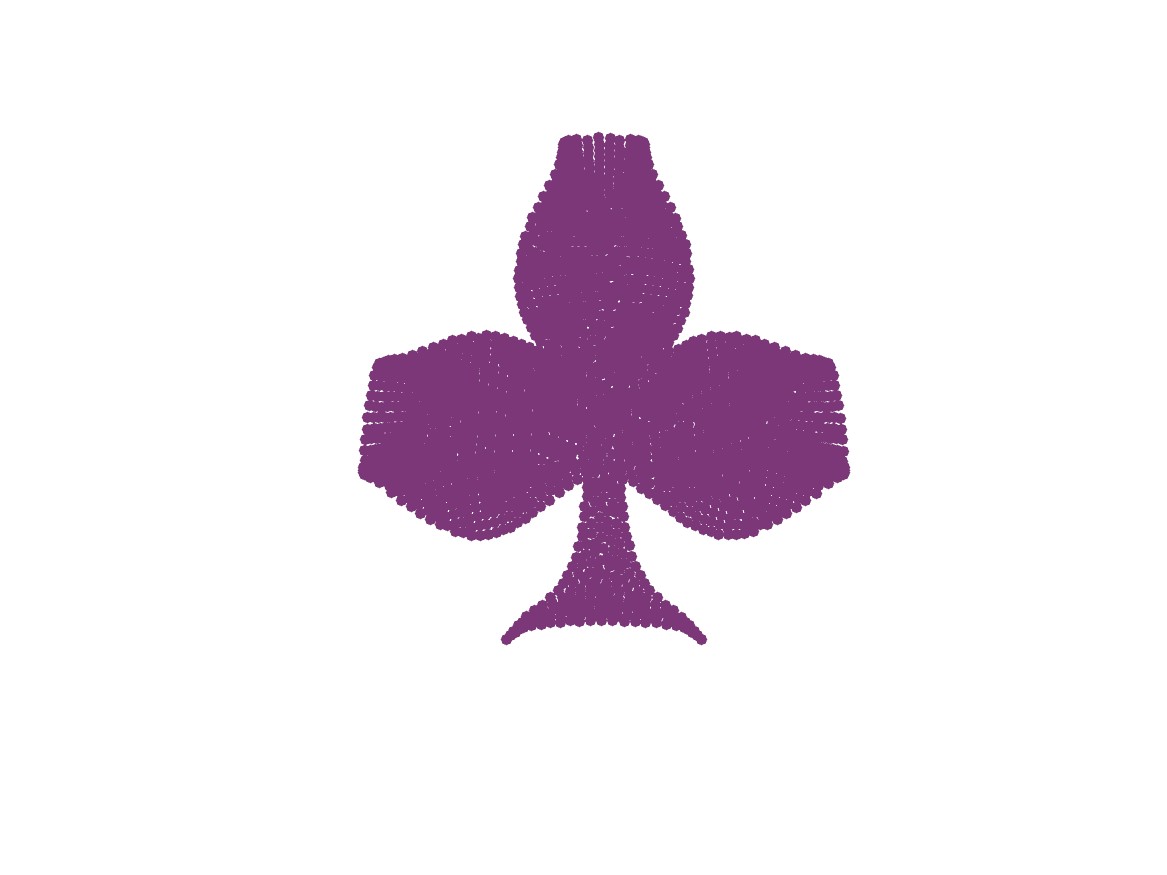}} & 
        {\includegraphics[trim={6cm 4cm 6cm 4cm},clip,width=0.1\textwidth]{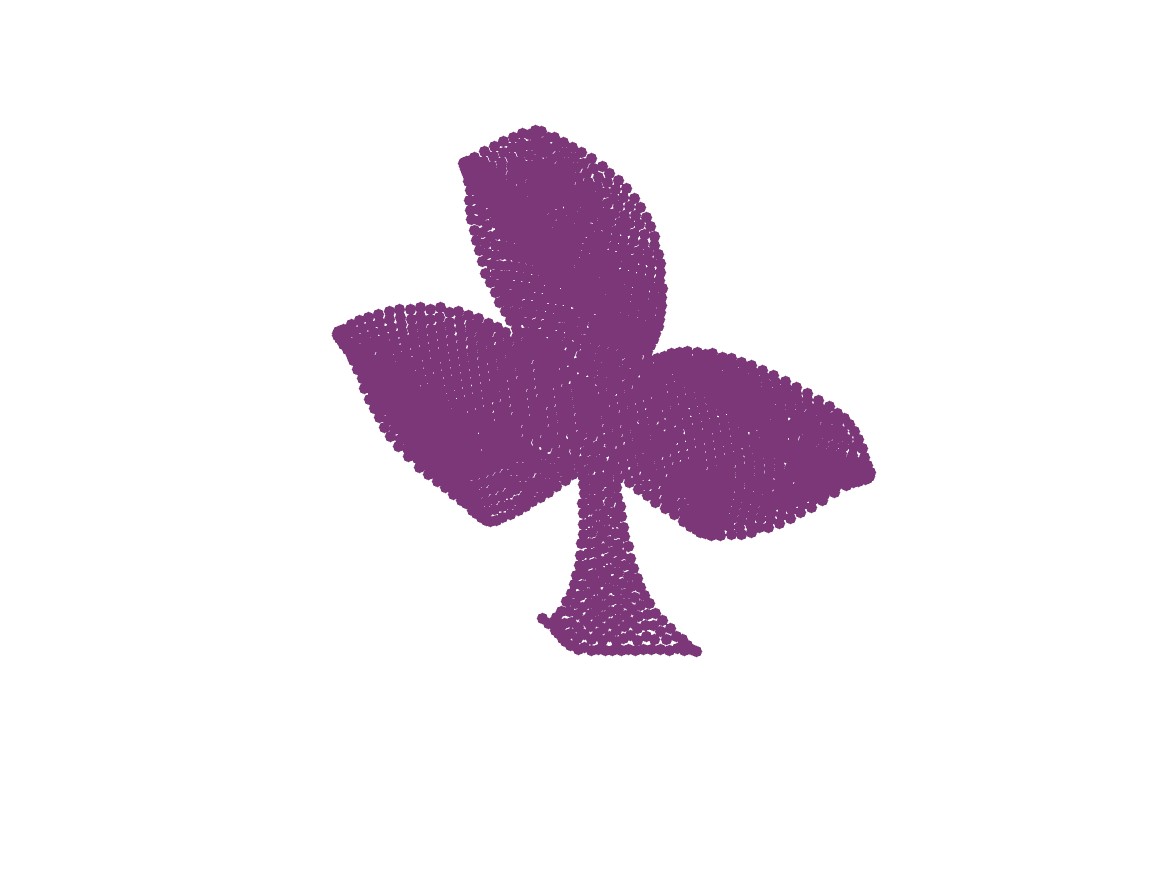}} & & \\
        & {\includegraphics[trim={6cm 4cm 6cm 4cm},clip,width=0.1\textwidth]{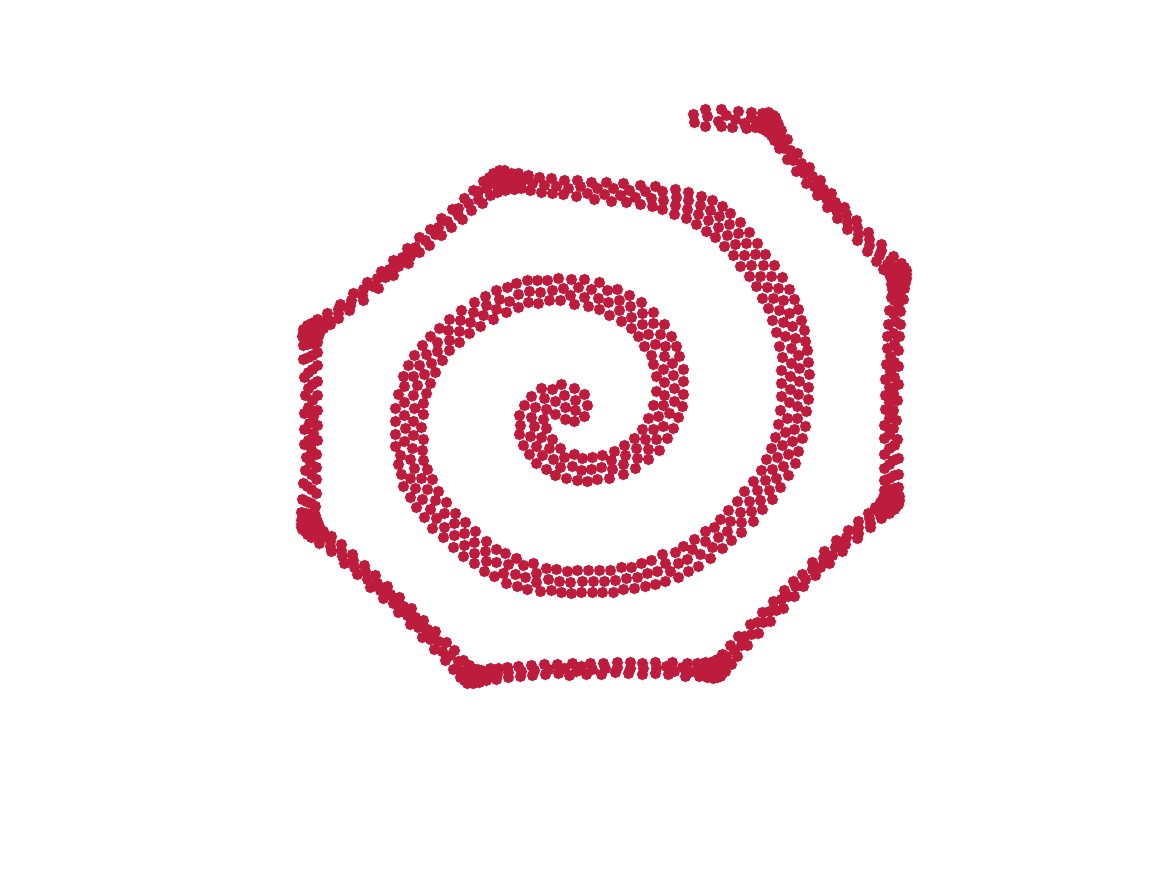}} & & 
        {\includegraphics[trim={6cm 4cm 6cm 4cm},clip,width=0.1\textwidth]{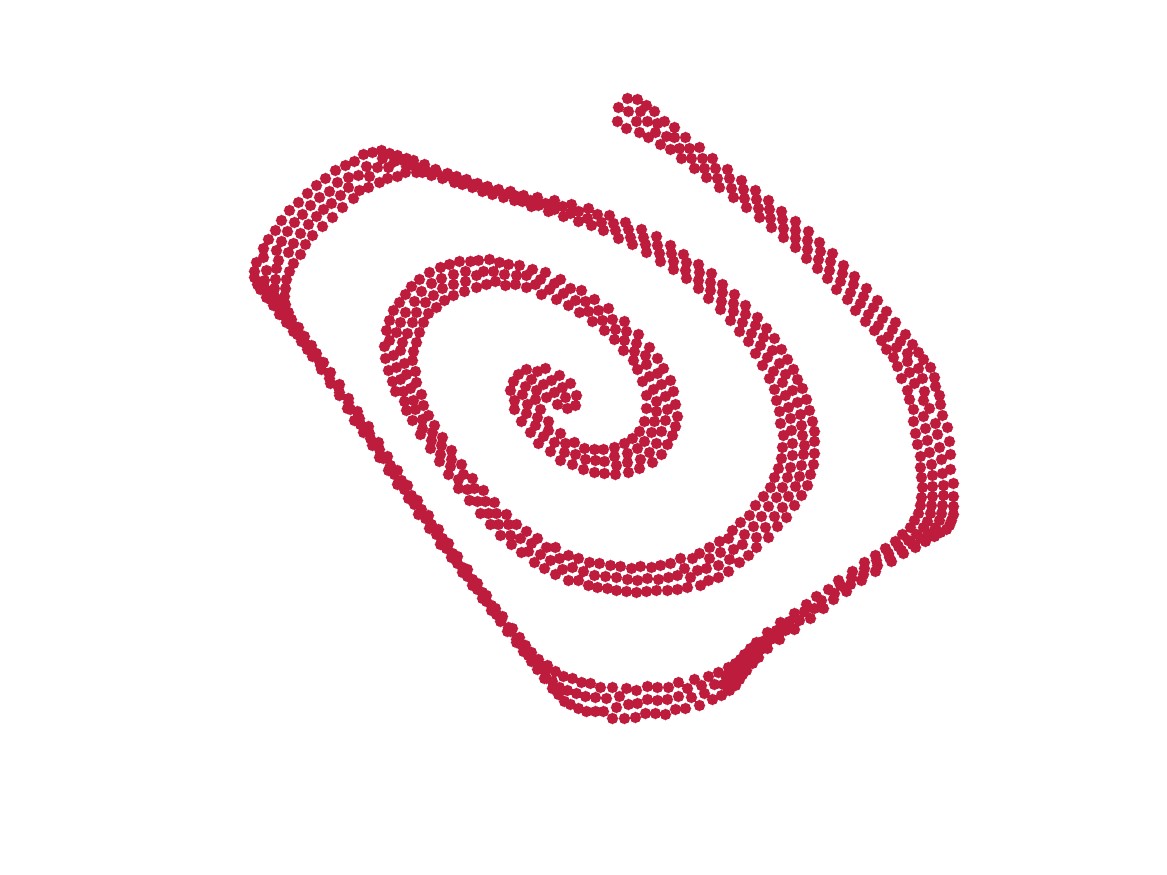}} & 
        {\includegraphics[trim={6cm 4cm 6cm 4cm},clip,width=0.1\textwidth]{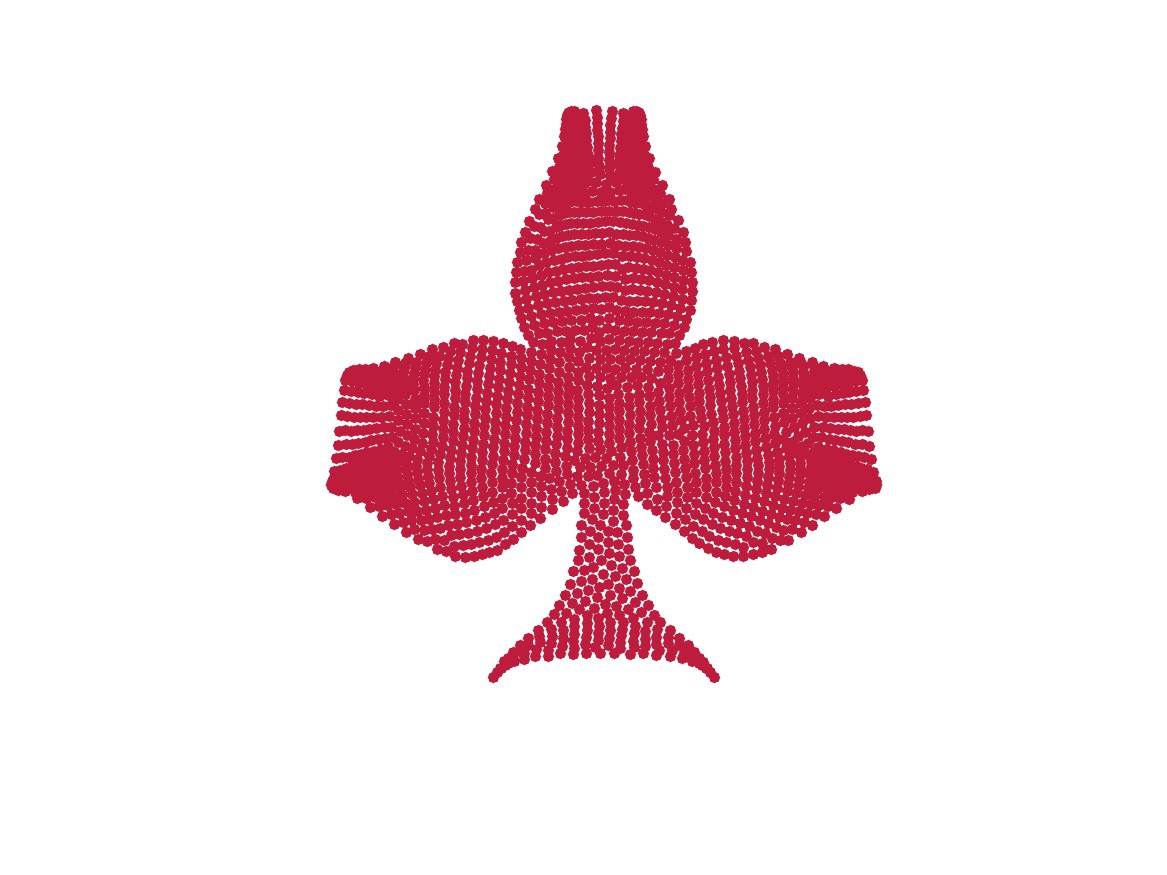}} & & 
        {\includegraphics[trim={6cm 4cm 6cm 4cm},clip,width=0.1\textwidth]{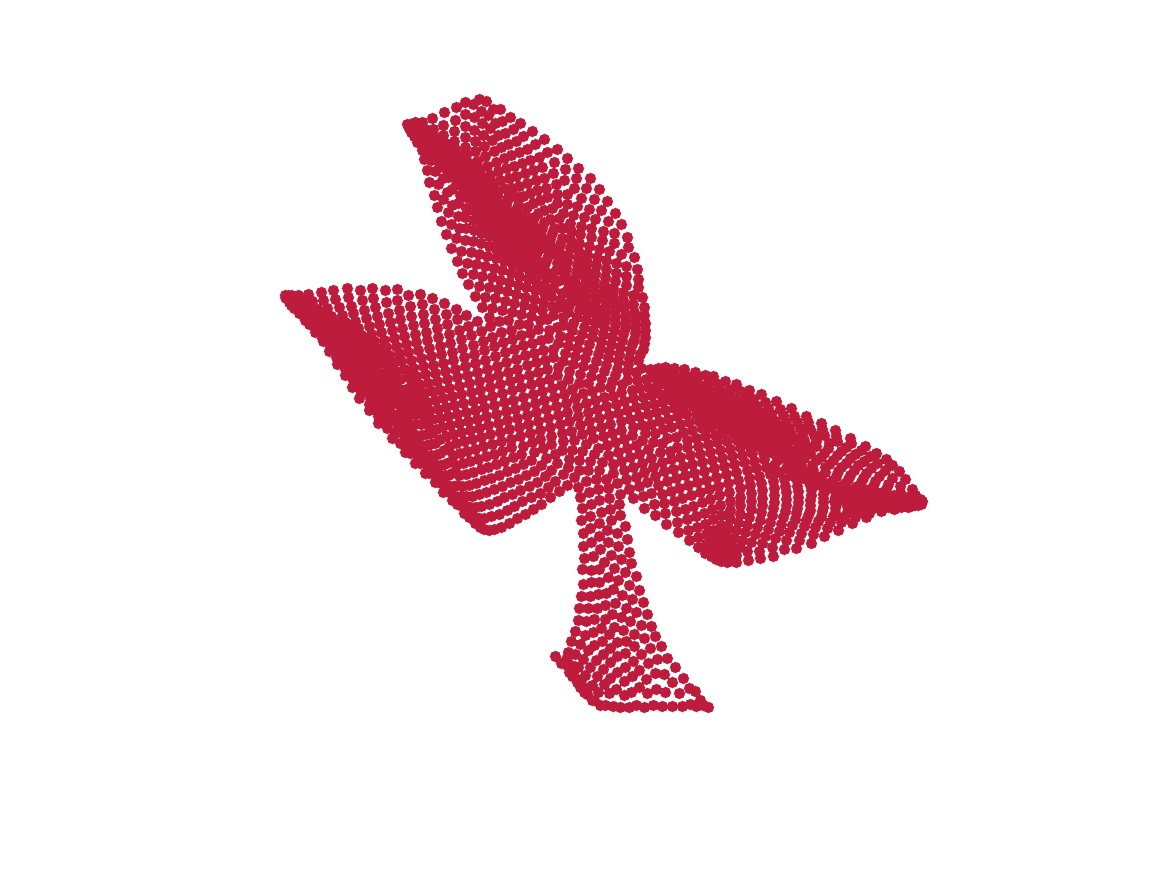}} & \\
        {\includegraphics[trim={6cm 4cm 6cm 4cm},clip,width=0.1\textwidth]{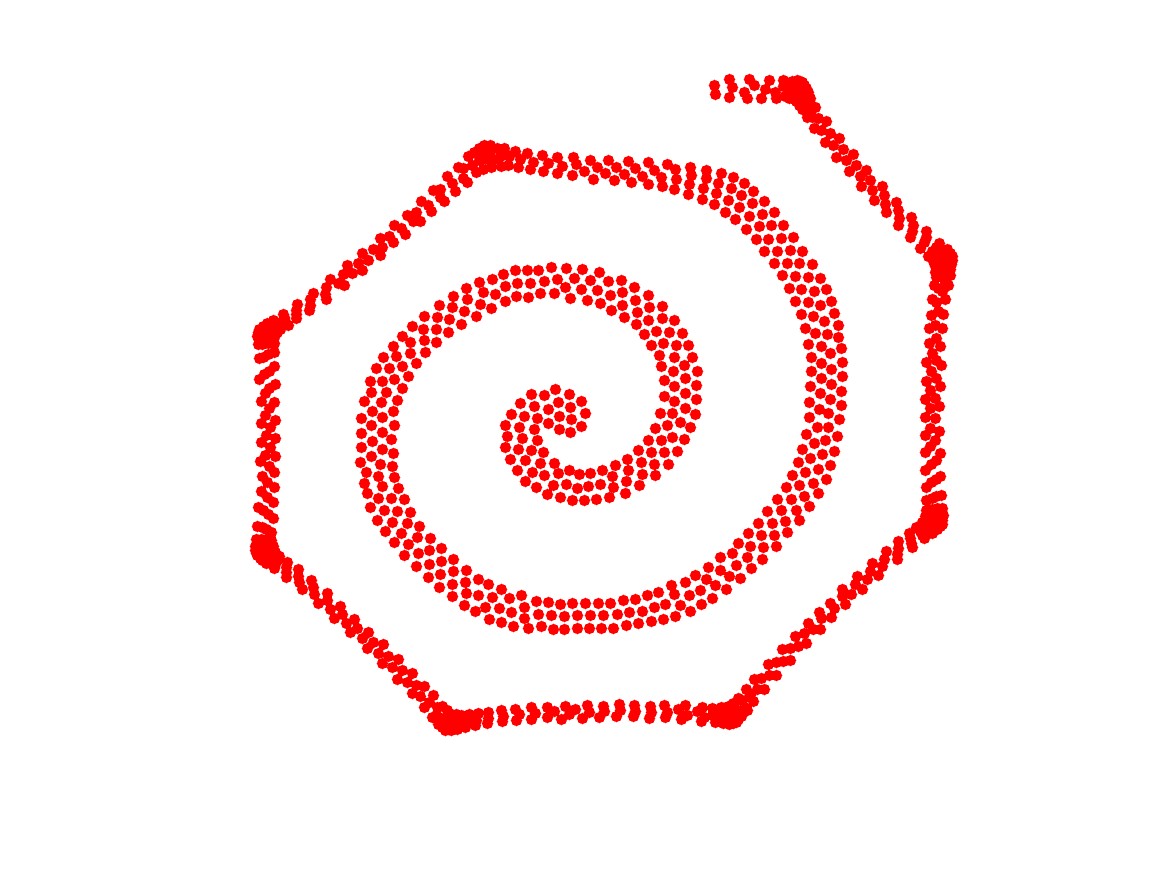}} & & & {\includegraphics[trim={6cm 4cm 6cm 4cm},clip,width=0.1\textwidth]{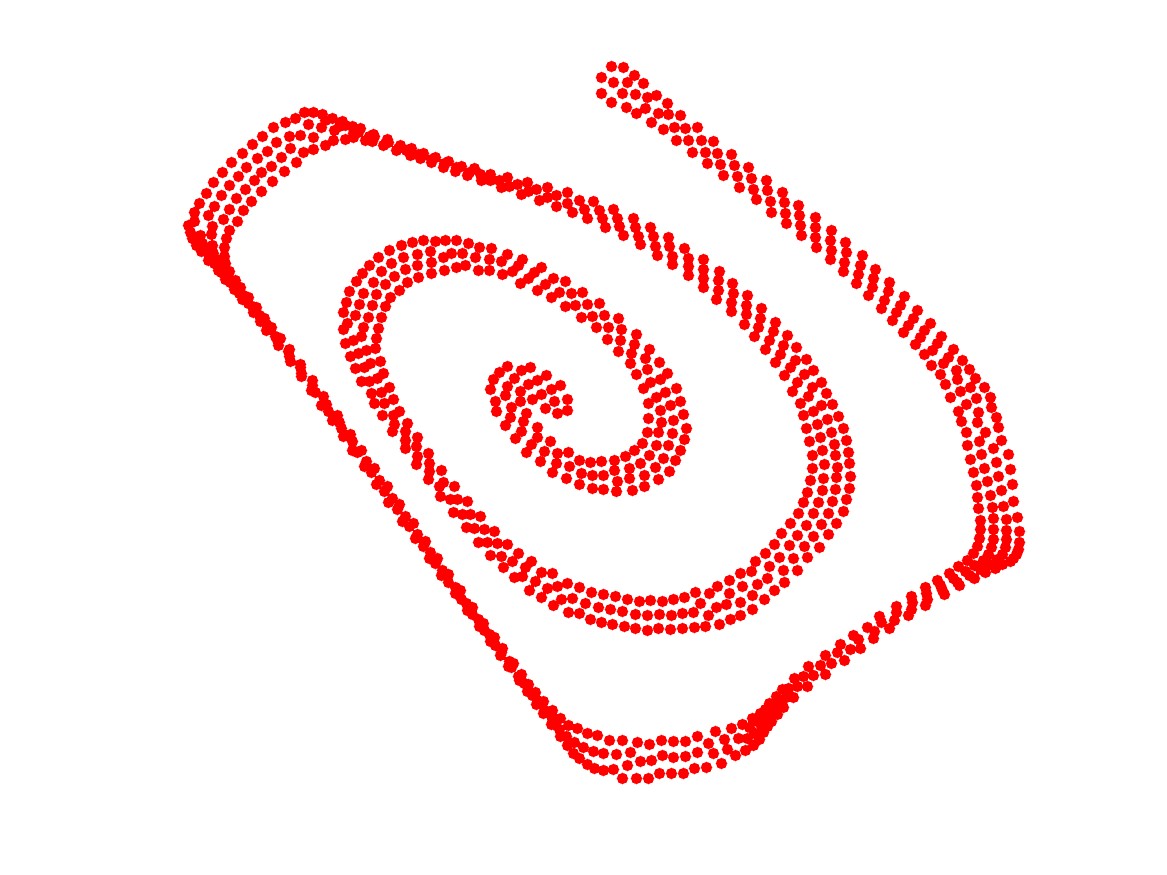}} & {\includegraphics[trim={6cm 4cm 6cm 4cm},clip,width=0.1\textwidth]{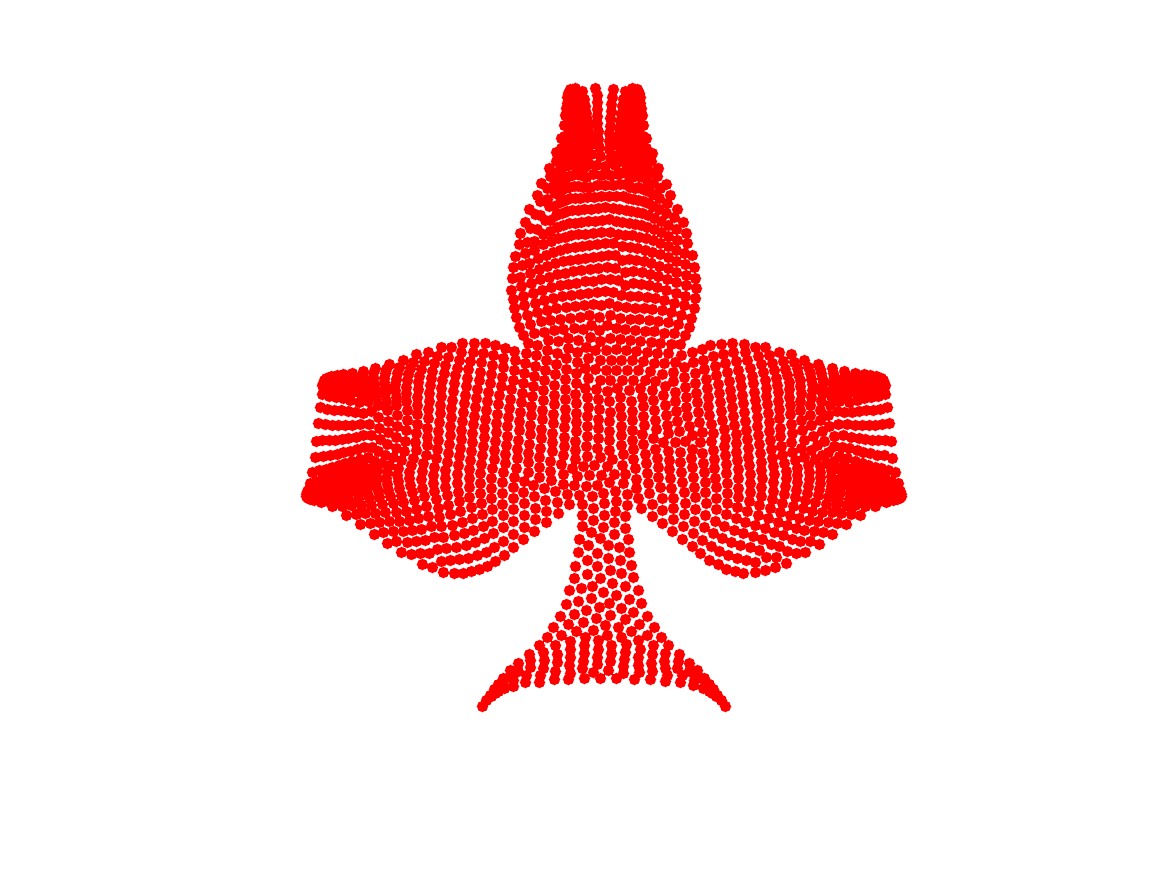}} & & & {\includegraphics[trim={6cm 4cm 6cm 4cm},clip,width=0.1\textwidth]{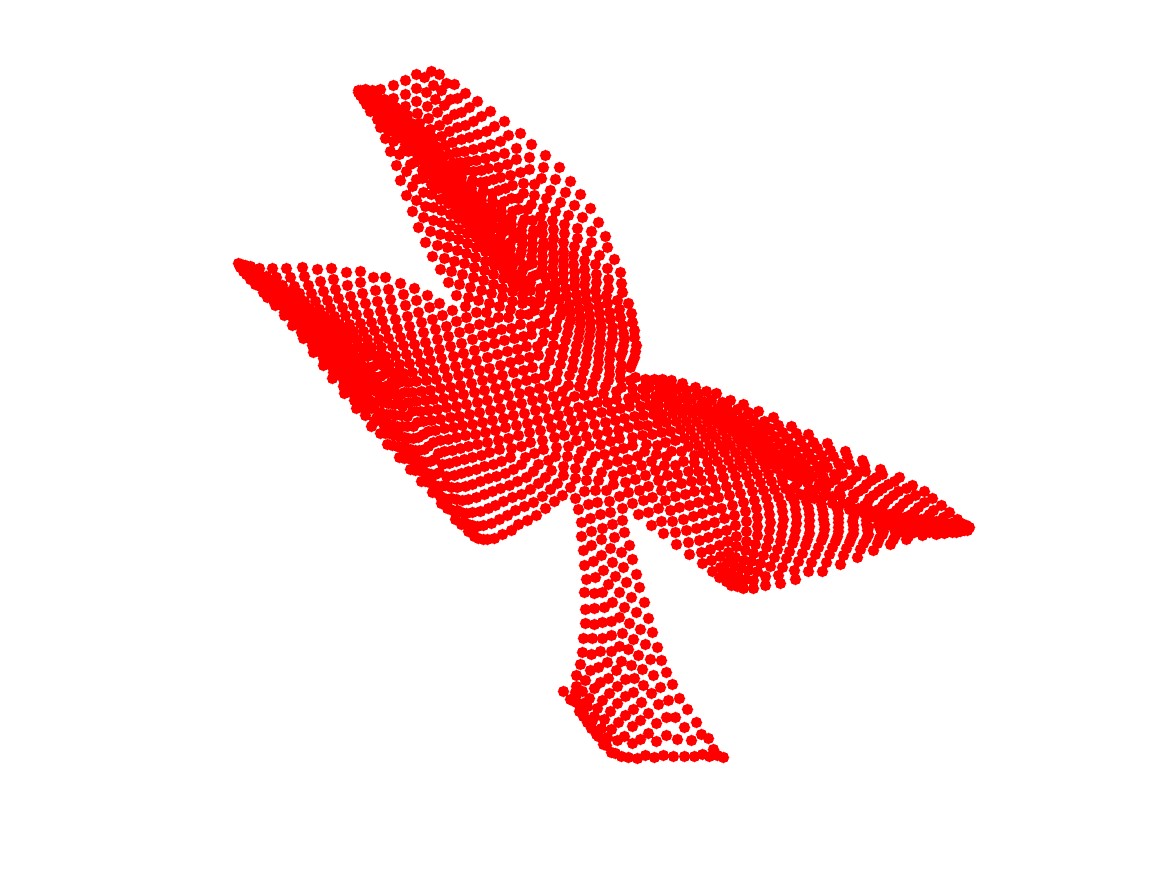}} 
    \end{tabularx}
   \caption{Example of extrapolation. The four images in the center are pairwise extrapolated in both positive and negative direction, from $t=1$ (blue) to $t=3$ (red).}
    \label{fig:extrapolation}
\end{figure}

In Figure \ref{fig:comparison} we compare the metric extrapolation with the Lagrangian extrapolation, where particles simply move on straight lines extending the geodesic trajectories from $\nu_0$ to $\nu_1$. Although for small times the two are qualitatively similar, for bigger times the structure of the initial data is completely lost in the Lagrangian extrapolation differently from the metric one.

\begin{figure}[t]
    \centering
    \setlength{\tabcolsep}{0pt}
    %\hspace{-1em}
    \begin{tabularx}{0.8\textwidth}{>{\centering\arraybackslash}X>{\centering\arraybackslash}X>{\centering\arraybackslash}X>{\centering\arraybackslash}X>{\centering\arraybackslash}X>{\centering\arraybackslash}X>{\centering\arraybackslash}X}
        \includegraphics[trim={6.5cm 3.1cm 5.3cm 2.15cm},clip,width=0.11\textwidth]{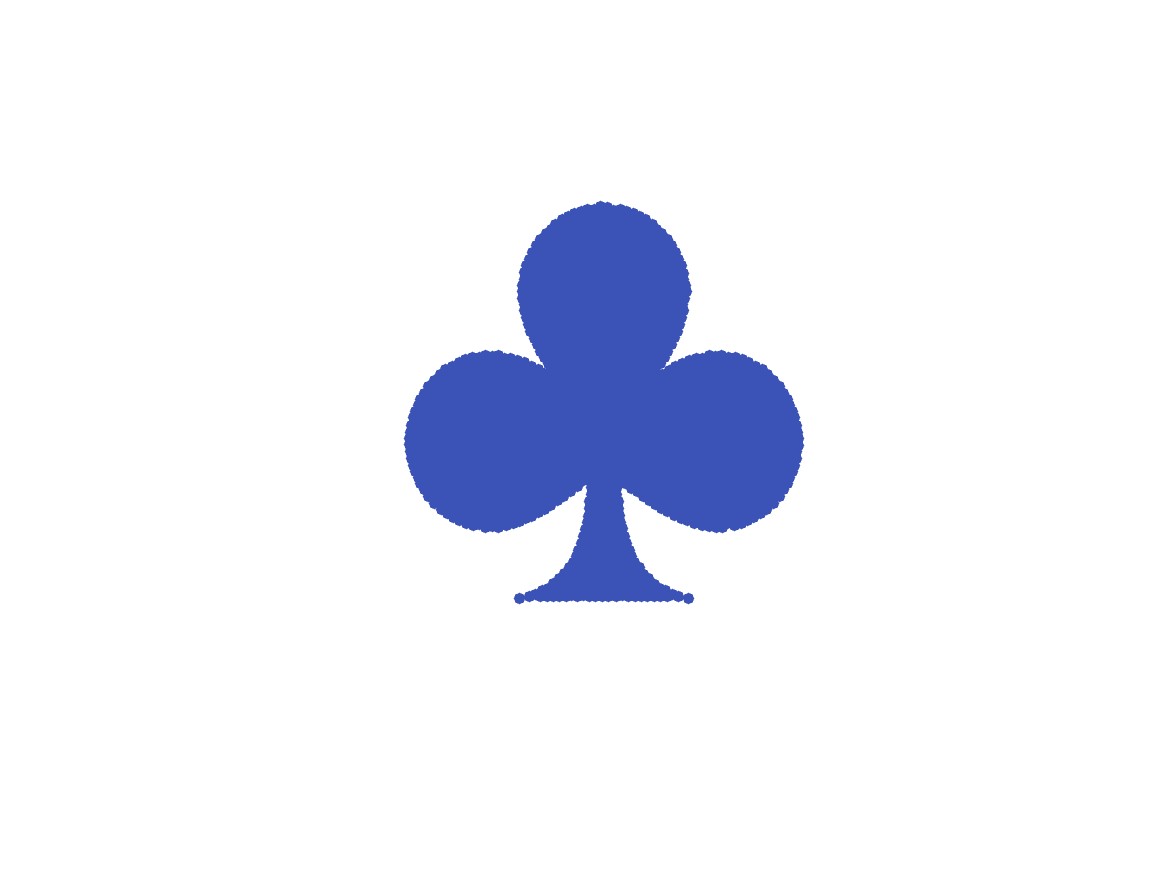} & \includegraphics[trim={6.5cm 3.1cm 5.3cm 2.15cm},clip,width=0.11\textwidth]{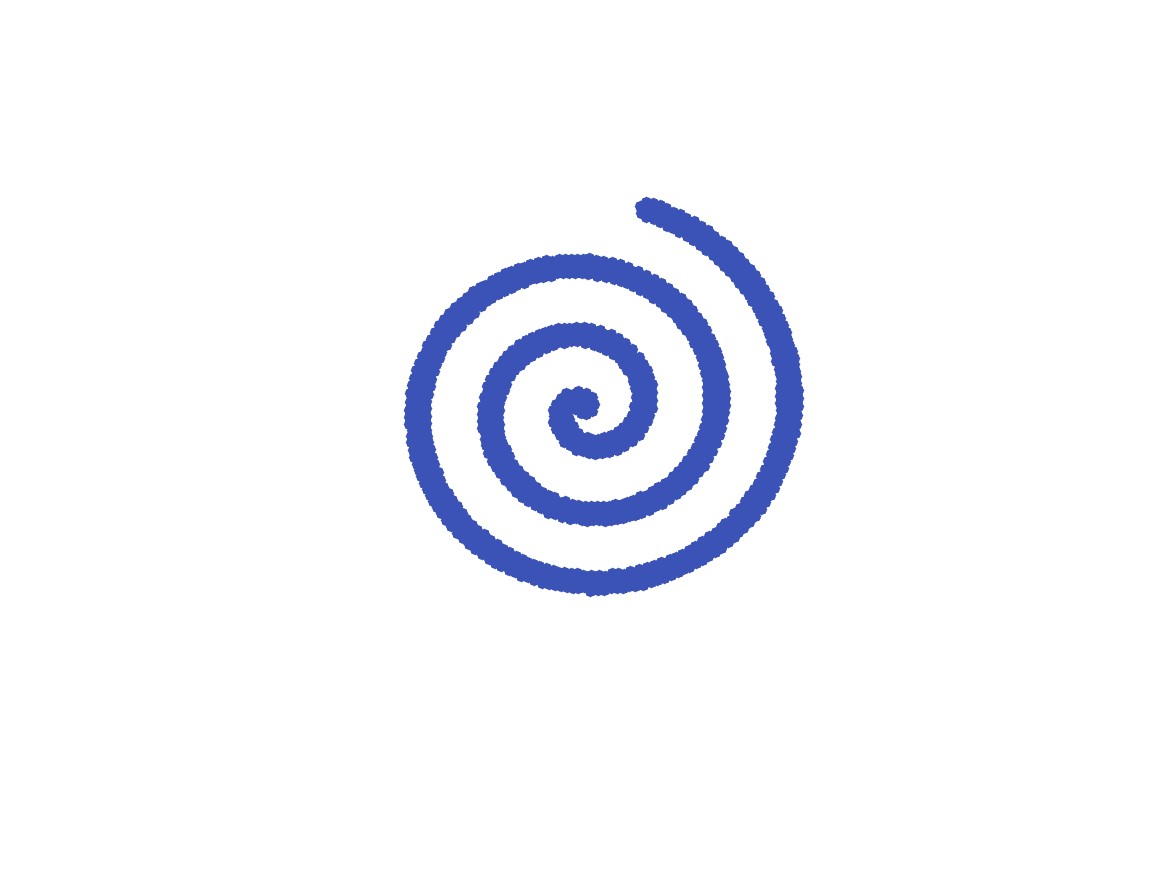} &
        \includegraphics[trim={6.5cm 3.1cm 5.3cm 2.15cm},clip,width=0.11\textwidth]{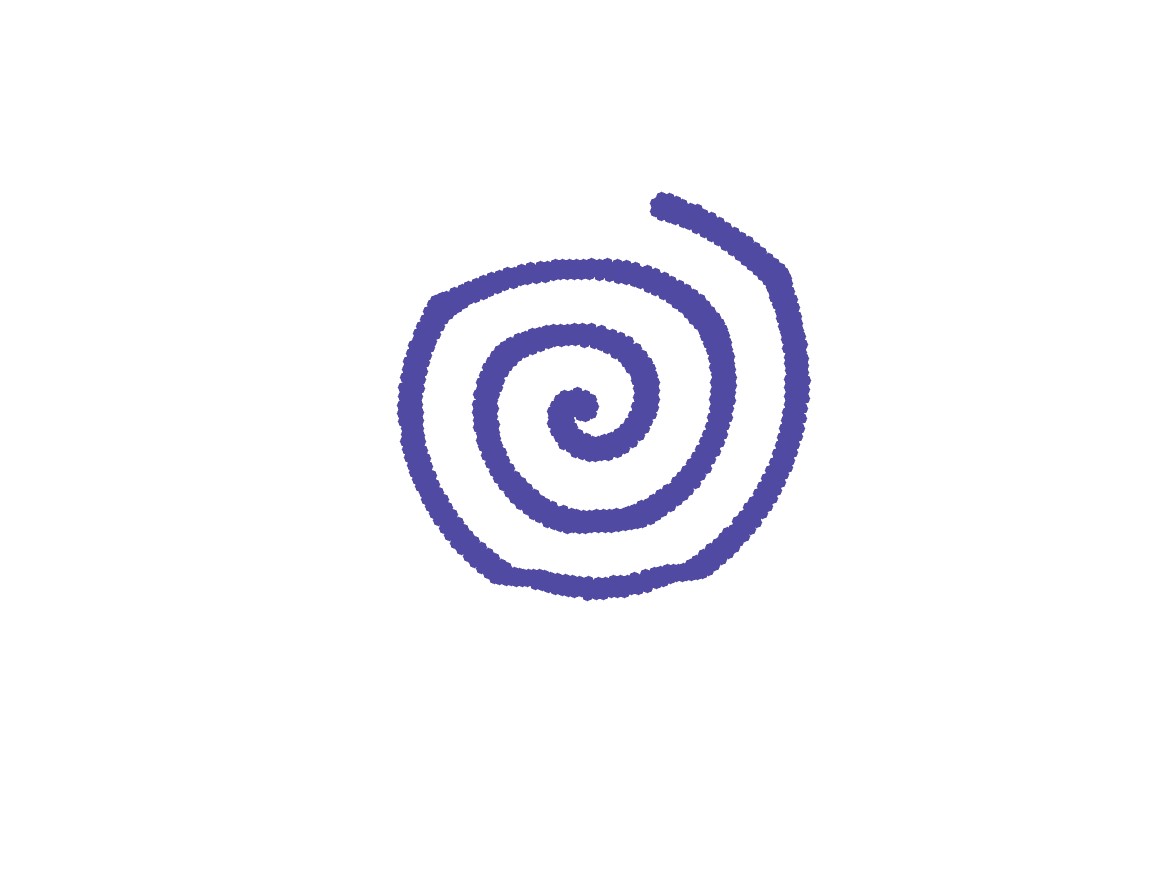} &
        \includegraphics[trim={6.5cm 3.1cm 5.3cm 2.15cm},clip,width=0.11\textwidth]{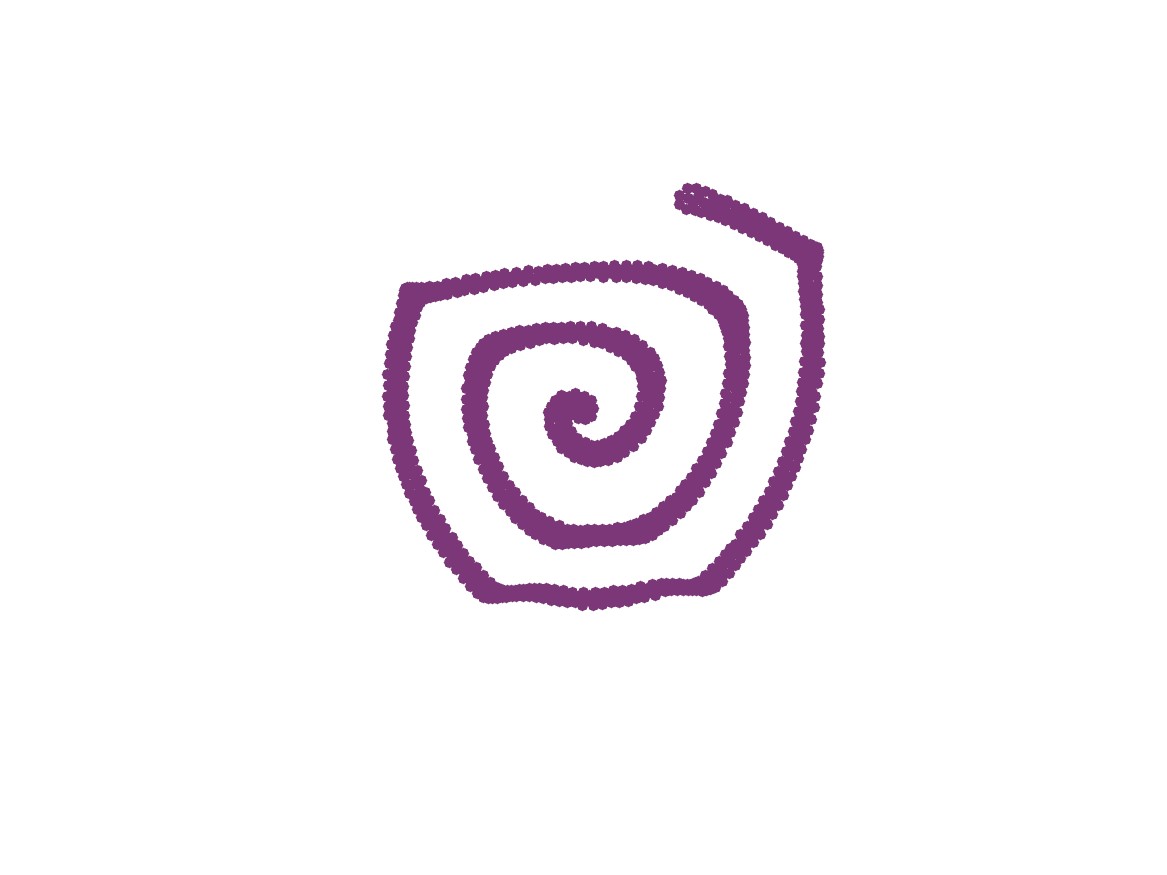} &
        \includegraphics[trim={6.5cm 3.1cm 5.3cm 2.15cm},clip,width=0.11\textwidth]{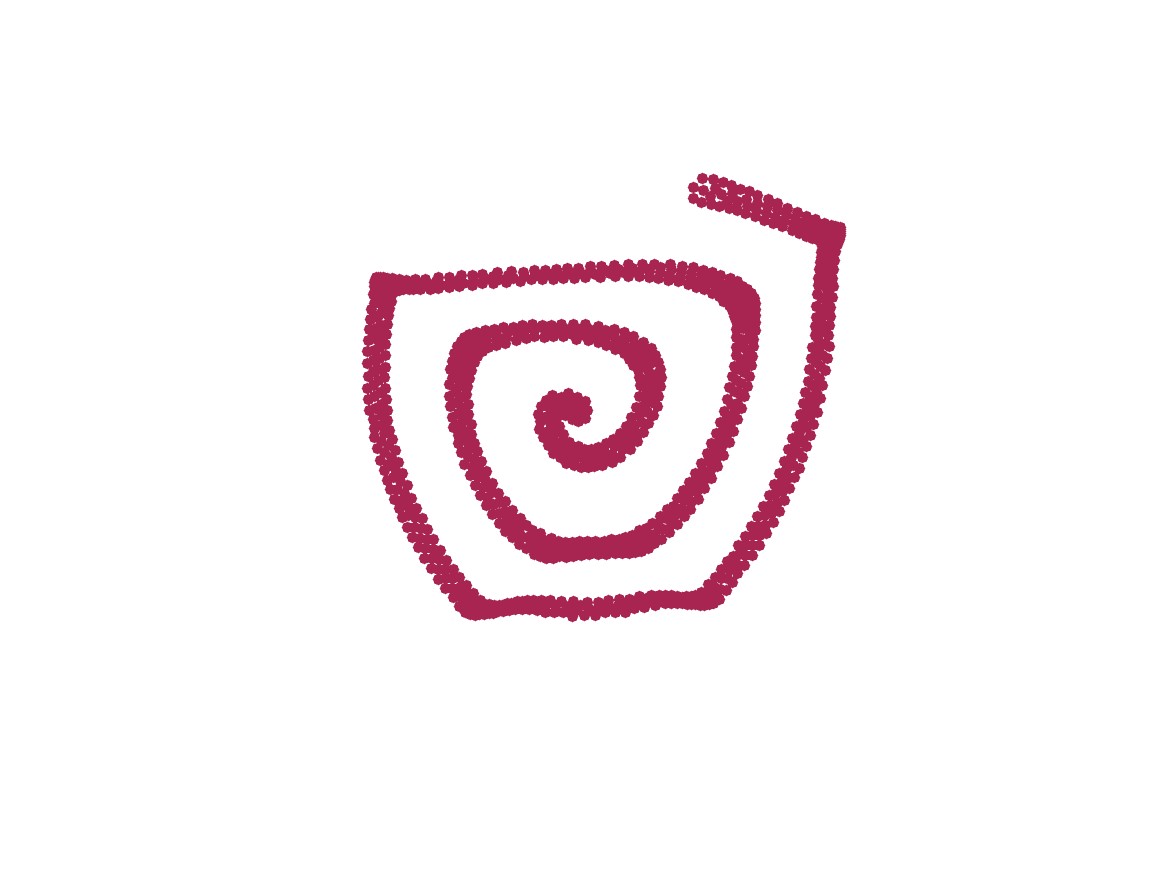} &
        \includegraphics[trim={6.5cm 3.1cm 5.3cm 2.15cm},clip,width=0.11\textwidth]{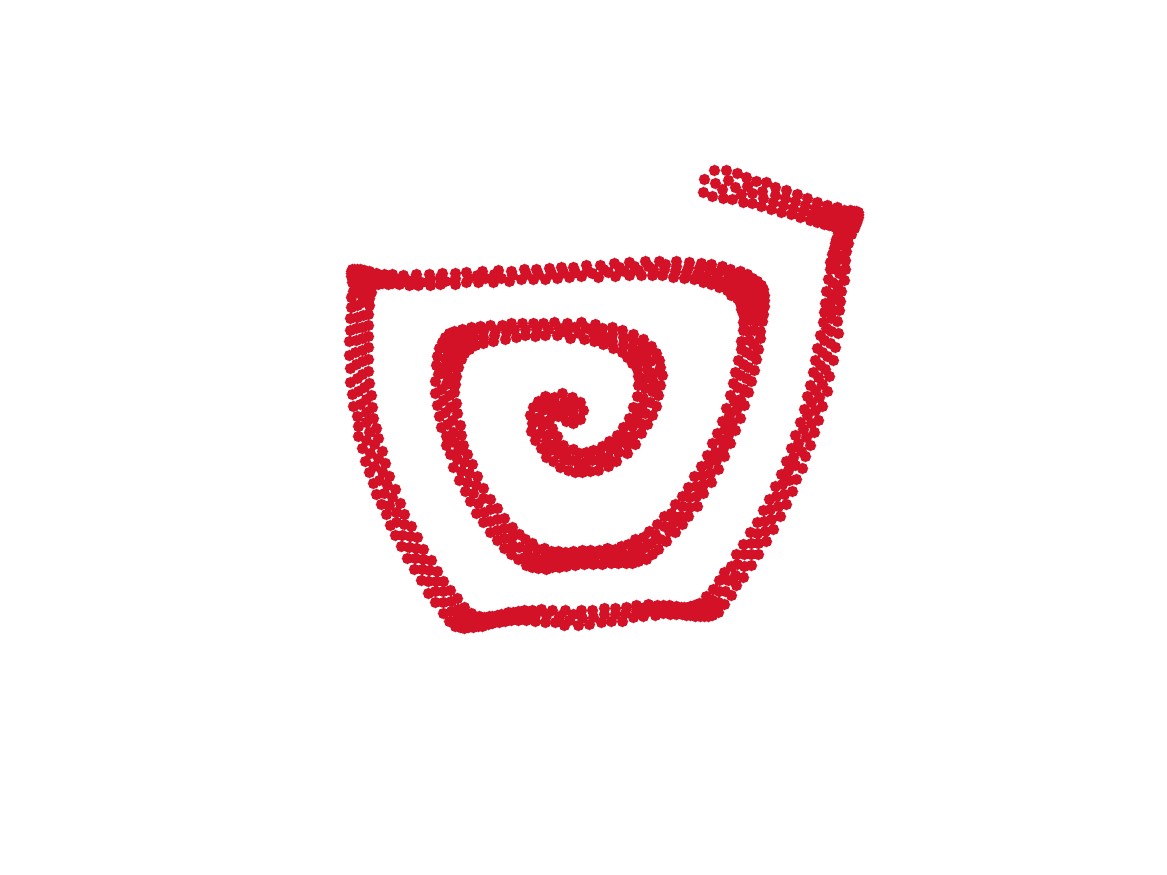} &
        \includegraphics[trim={6.5cm 3.1cm 5.3cm 2.15cm},clip,width=0.11\textwidth]{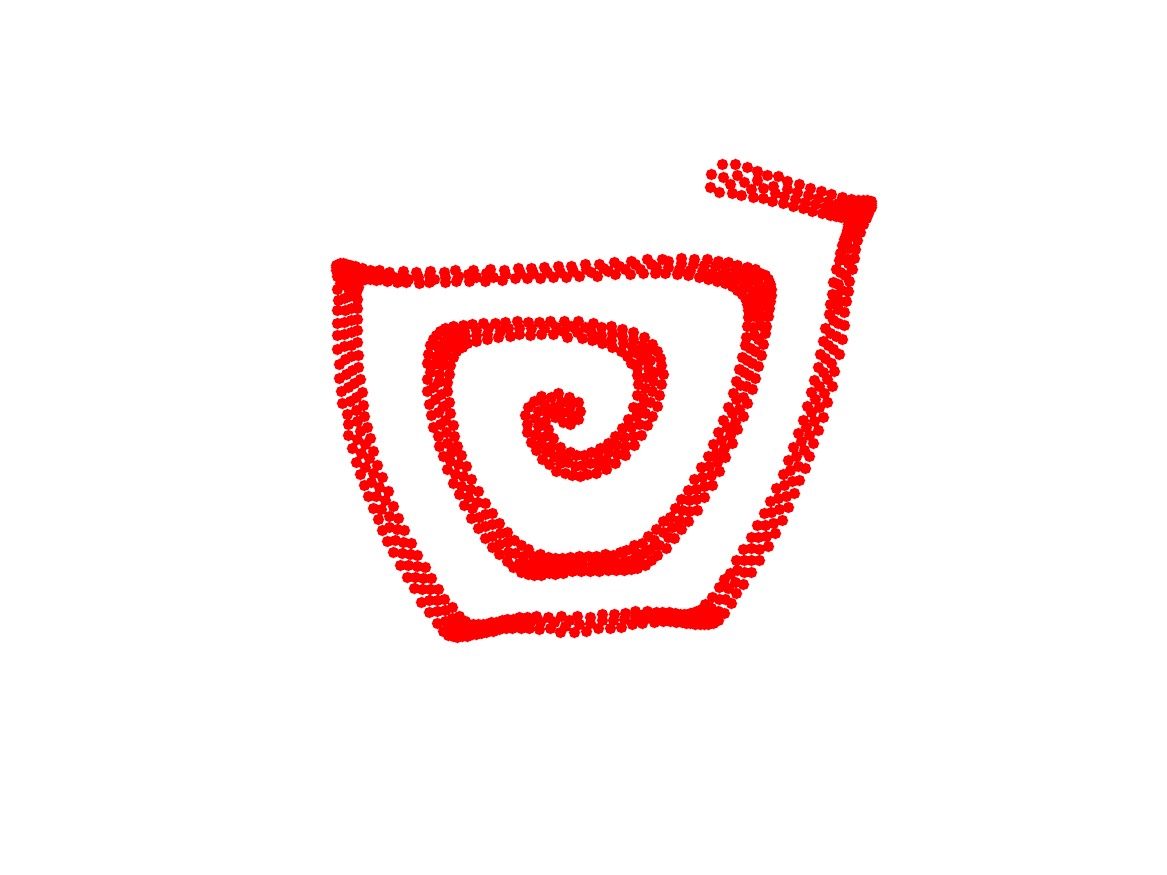} \\
        \includegraphics[trim={6.5cm 3.1cm 5.3cm 2.15cm},clip,width=0.11\textwidth]{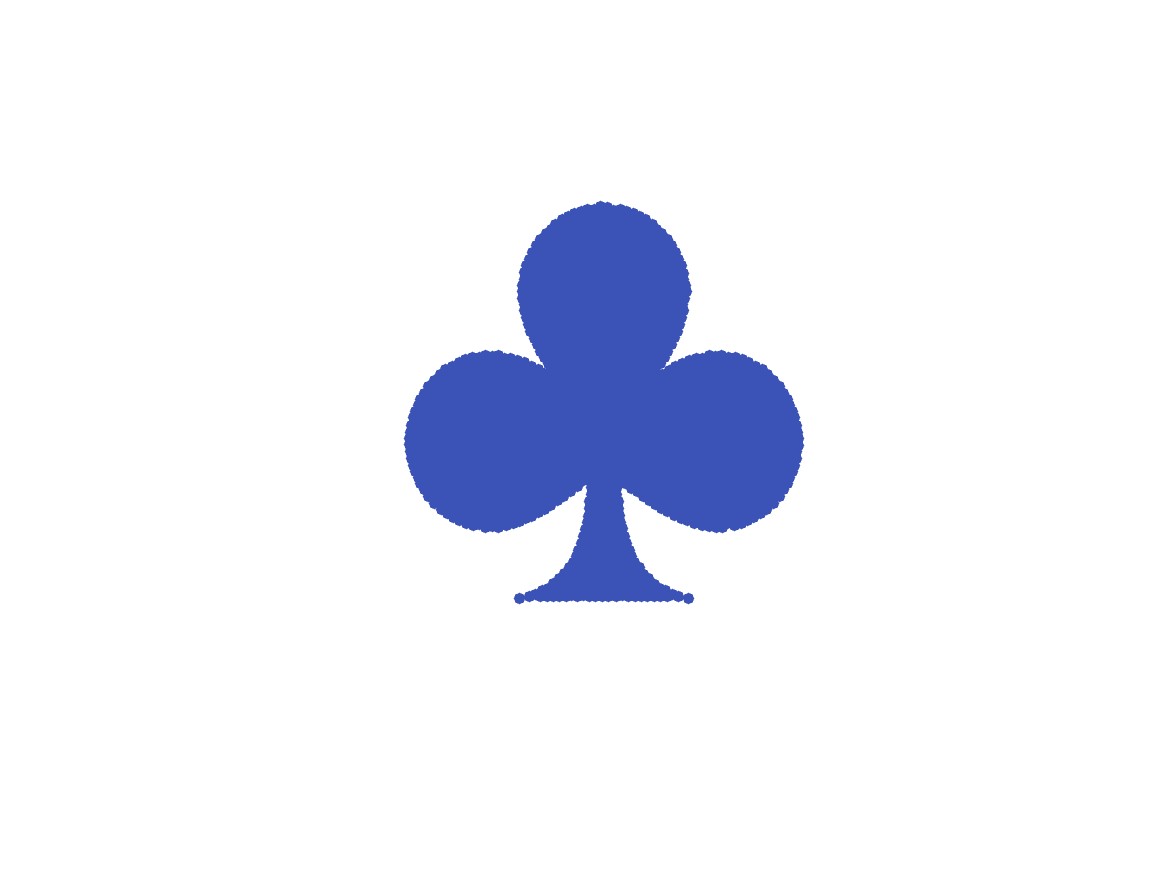} & \includegraphics[trim={6.5cm 3.1cm 5.3cm 2.15cm},clip,width=0.11\textwidth]{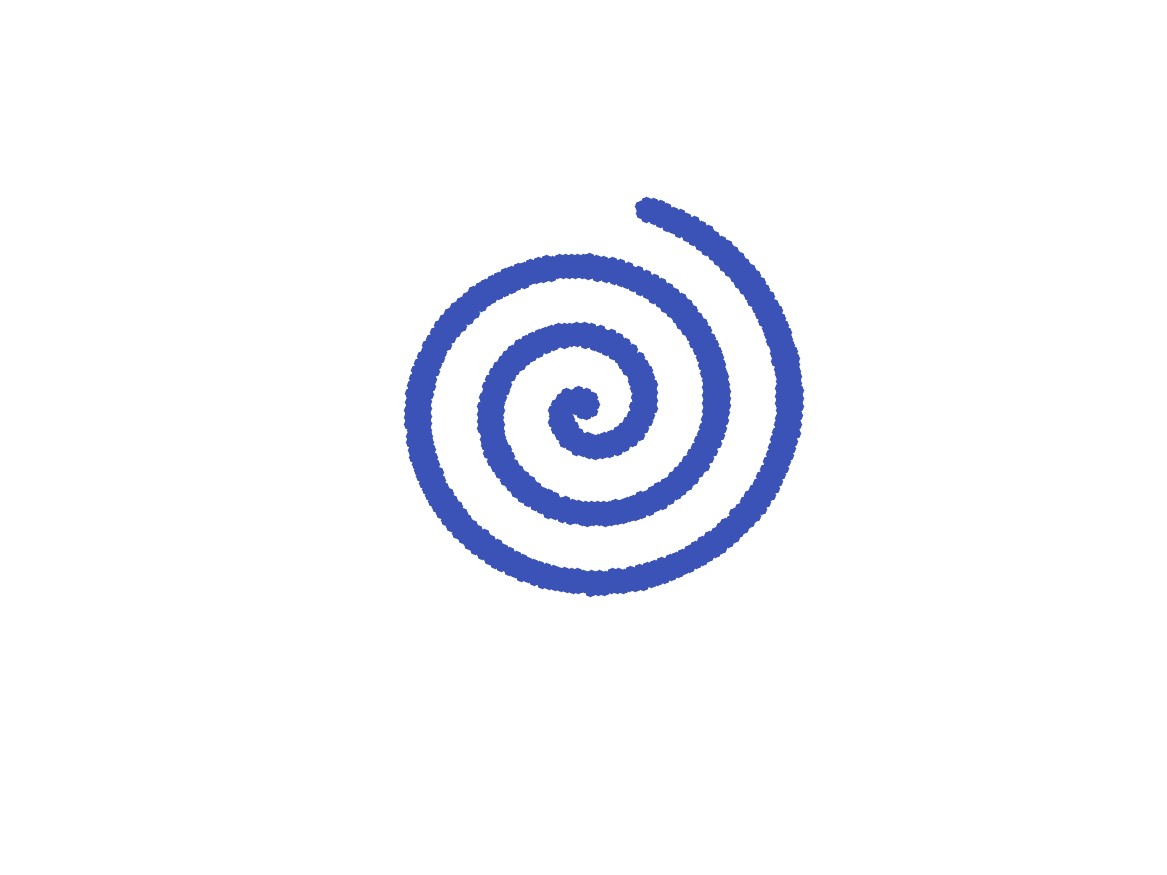} &
        \includegraphics[trim={6.5cm 3.1cm 5.3cm 2.15cm},clip,width=0.11\textwidth]{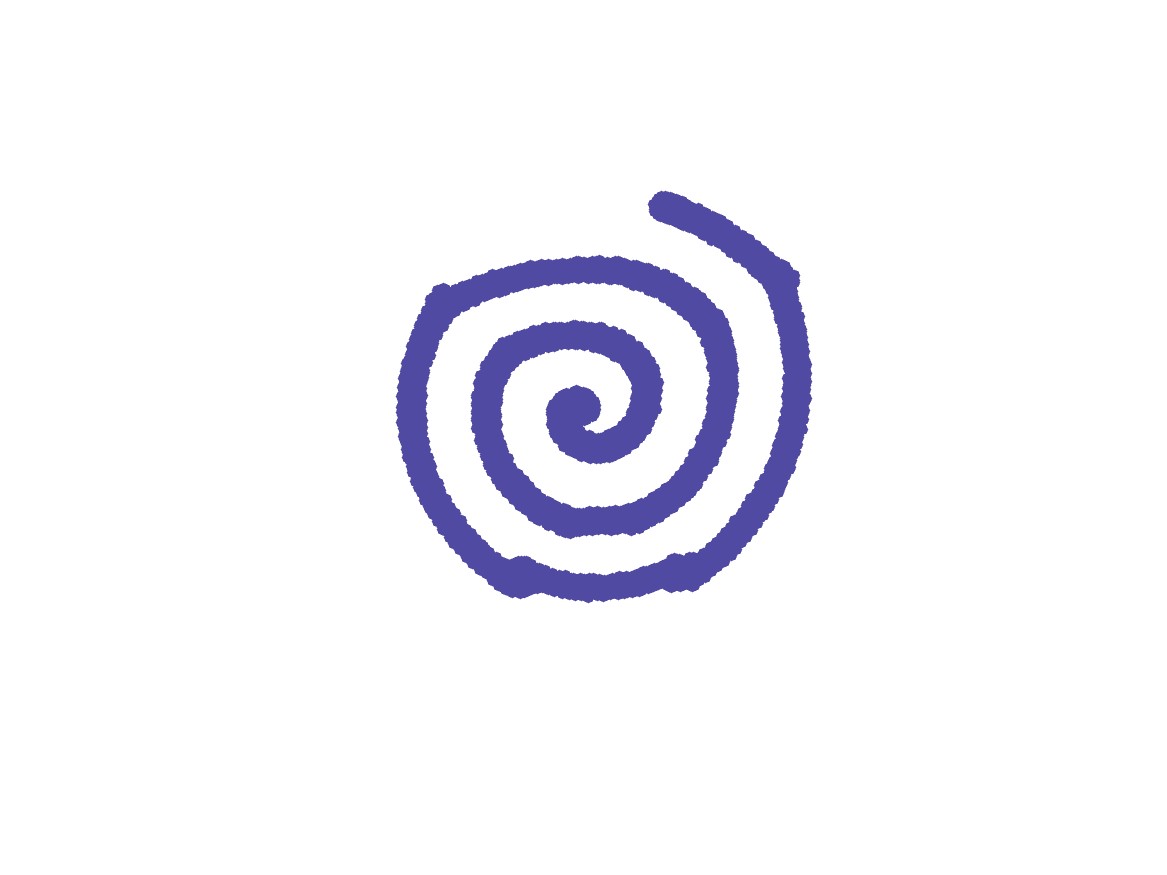} &
        \includegraphics[trim={6.5cm 3.1cm 5.3cm 2.15cm},clip,width=0.11\textwidth]{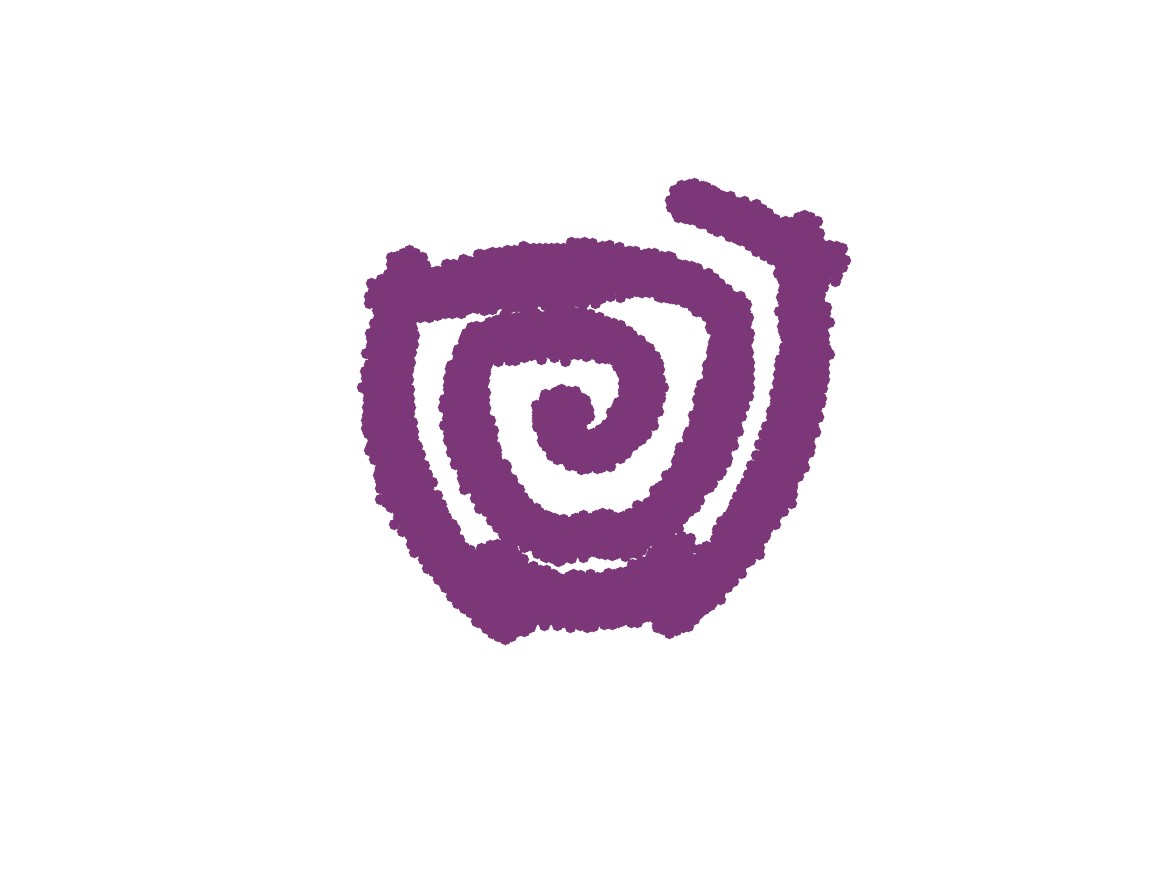} &
        \includegraphics[trim={6.5cm 3.1cm 5.3cm 2.15cm},clip,width=0.11\textwidth]{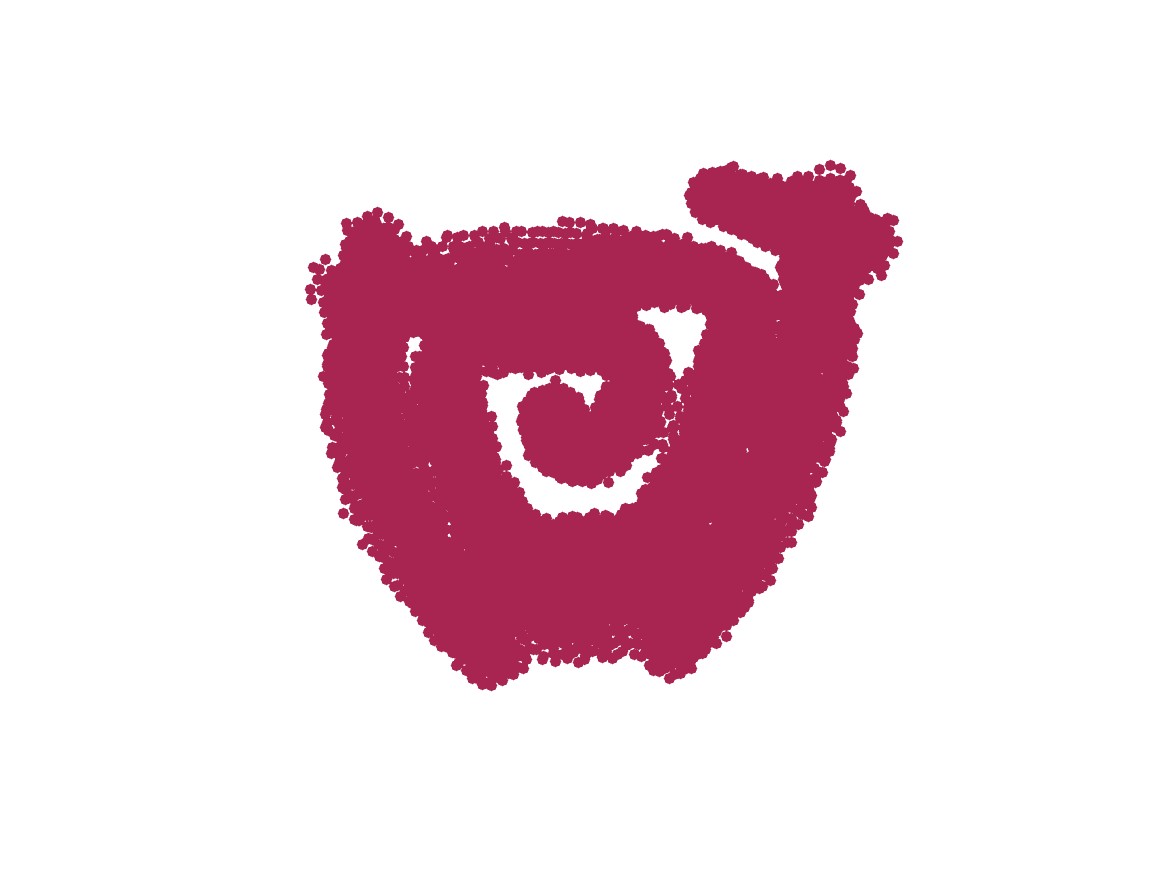} &
        \includegraphics[trim={6.5cm 3.1cm 5.3cm 2.15cm},clip,width=0.11\textwidth]{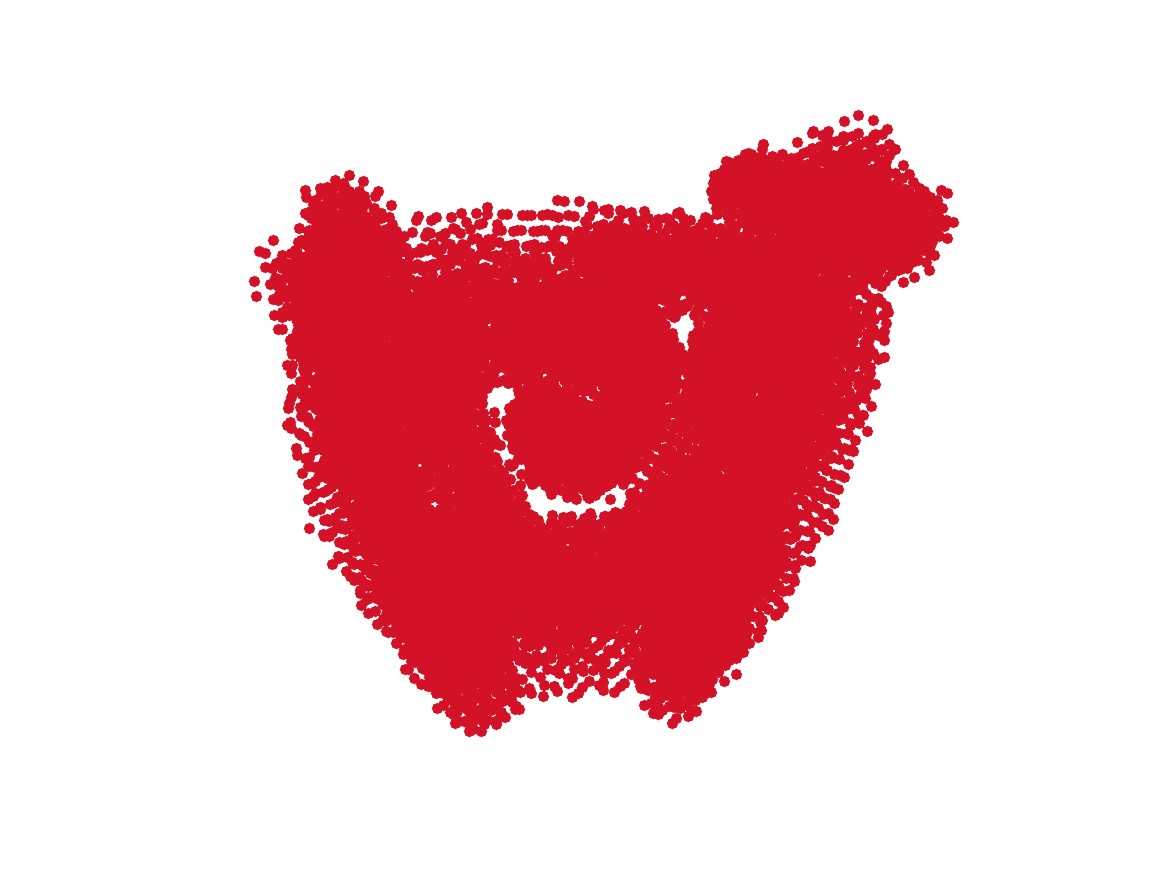} &
        \includegraphics[trim={6.5cm 3.1cm 5.3cm 2.15cm},clip,width=0.11\textwidth]{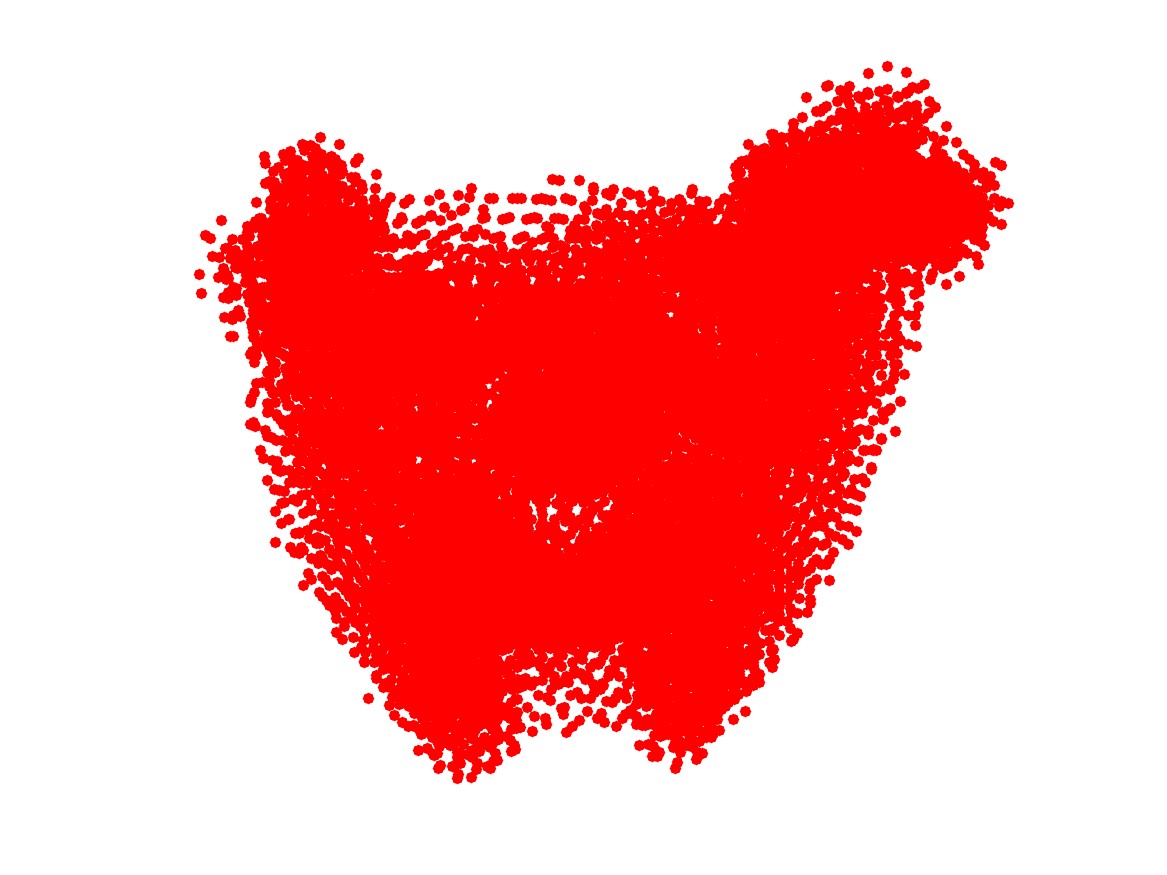} \\
    \end{tabularx}
    \caption{Comparison between the metric extrapolation (top row) and the Lagrangian extrapolation (bottom row), from $\nu_0$ (first column from the left) to $\nu_1$ (second column from the left). Time grows from left to right, from $t=1$ (blue) to $t=4$ (red).}
    \label{fig:comparison}
\end{figure}

\subsubsection{Gradient flow of the opposite Wasserstein distance} Let us consider the curve $\nu_{\mathrm{gf}}:[1,\infty) \rightarrow \mc{P}_2(\mathbb{R}^d)$ obtained as the $W_2$ gradient flow of the time dependent energy:
\begin{equation}\label{eq:energyneg}
\mc{E}(t,\mu) \coloneqq -\frac{W^2_2(\nu_0,\mu)}{2t}
\end{equation}
for $t\geq 1$ and initial conditions $\nu_{\mathrm{gf}}(1) = \nu_1$. This is a well-defined flow as discussed in \cite{ambrosio2008gradient}, and moreover if there exists a length-minimizing geodesic on $[0,t]$ passing through $\nu_0$ at time 0 and $\nu_1$ at time 1, $\nu$ coincides with it on the interval $[1,t]$ \cite[Theorem 11.2.10]{ambrosio2008gradient} (note that differently from \cite{ambrosio2008gradient} we use a time-dependent energy to avoid rescaling time when comparing the gradient flow with geodesics). Our framework provides a natural way to discretize this flow. As a matter of fact, the JKO scheme \cite{jordan1998variational} with time step $h>0$ yields the following minimization problem, for all $n\geq 0$,
\begin{equation}\label{eq:jko_extra}
\nu^{n+1}_{\mathrm{gf}} = \underset{\mu\in\mc{P}_2(\mathbb{R}^d)}{\mathrm{argmin}}  \frac{W_2^2(\nu^{n}_{\mathrm{gf}},\mu)}{2h} -\frac{W_2^2(\nu_0,\mu)}{2t^{n+1}}    
\end{equation}
where we set $\nu^0_{\mathrm{gf}} =\nu_1$ and $t^n = 1 +nh$, which is a metric extrapolation problem. Remarkably, in one dimension  $\nu_{\mathrm{gf}}(t)=\nu_t$, the metric extrapolation from $\nu_0$ to $\nu_1$ at time $t$, and it also yields a sticky solution to the pressureless Euler equations (see \cite{natile2009wasserstein} and Remark \ref{rem:pressureless}). 
In higher dimension, however, in general $\nu_{\mathrm{gf}}(t)\neq \nu_t$, as it could be expected since only in one dimension the $L^2$ transport cost is equivalent to an $L^2$ distance on quantile functions. Note also that the metric extrapolation coincides with one single JKO step for the energy \eqref{eq:energyneg} up to the final time. 
A simple scenario where one can observe the difference between the two flows is provided in Figure \ref{fig:flowdiff}, where $\nu_0$ and $\nu_1$ are a sum of Dirac measures with equal masses:
\[
\nu_0 = \frac{1}{3} \sum_{i=1}^3 \delta_{x_i} \,,\quad \nu_1 =\frac{1}{2} \sum_{j=1}^2 \delta_{y_j}
\]
where $x_1=(-2,1)$, $x_2=(2,1)$, $x_3=(0,0)$, and $y_1=(-1, 1)$, $y_2=(1,1)$ on the left, and $y_2=(1,1.2)$ on the right. The two flows are represented on the time interval $[1,5.5]$ and with time step $h = 4.5\cdot 10^{-3}$.
Note that when the mass is distributed symmetrically with respect to the $y$ axis the two flows coincide (the same happens for the explicit examples of Section \ref{sec:particular}).  However for a slightly perturbed $\nu_1$ we obtain two different curves for the metric extrapolation and the gradient flow. Note that these results were obtained using an interior point scheme to solve the quadratic program \eqref{eq:bar_atomic} with high precision, but one obtains the same picture using the entropic scheme of this section for sufficiently small $\varepsilon$. 
\begin{figure}[t]
\begin{overpic}[width=0.37\textwidth]{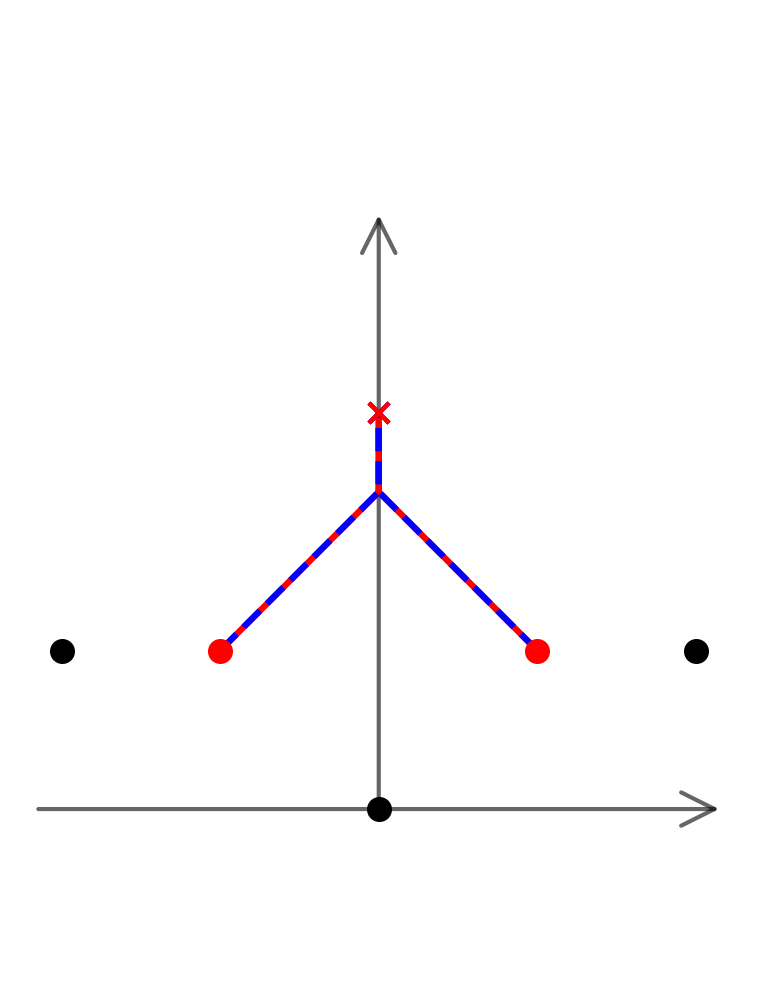}
 \put (60,70) {\color{red}$\displaystyle\nu_1$}
 \put (95,45) {$\displaystyle\nu_0$}
\end{overpic}
\hspace{1em}
\begin{overpic}[width=0.37\textwidth]{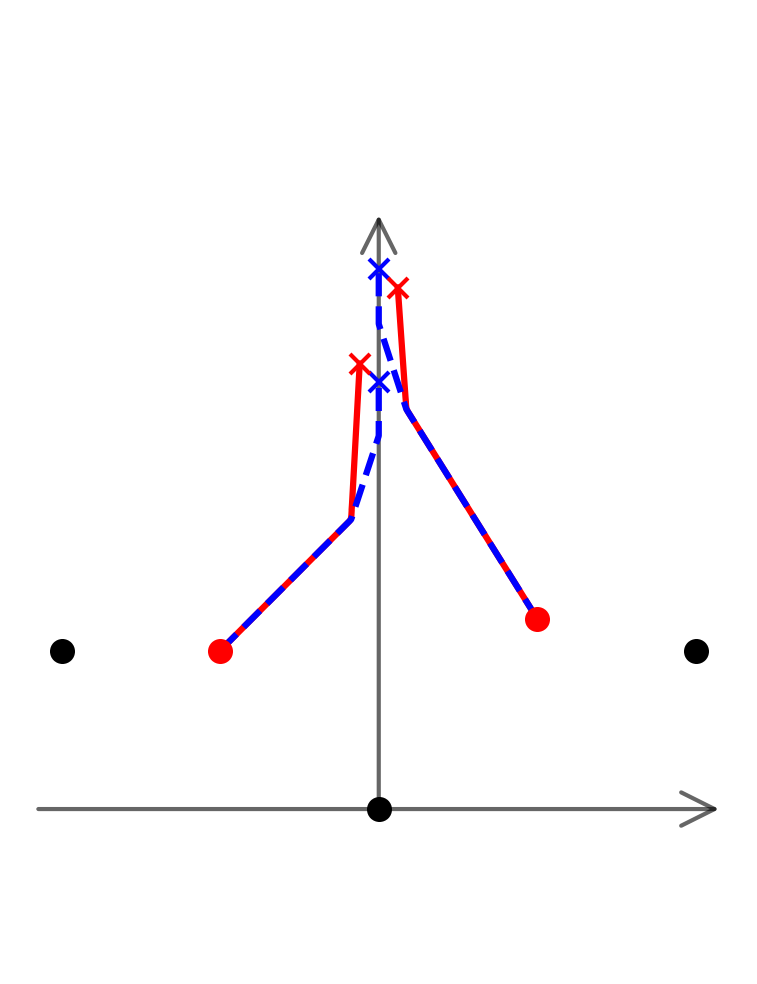}
 \put (60,70) {\color{red}$\displaystyle\nu_1$}
 \put (95,45) {$\displaystyle\nu_0$}
\end{overpic}
\caption{A comparison between the gradient flow $\nu_\mathrm{gf}$ (dashed blue line) and metric extrapolation $\nu_t$ (solid red line)  on the same time interval, for the same $\nu_0$ and two different initial conditions for $\nu_1$.}\label{fig:flowdiff}
\end{figure}

In Figure \ref{fig:flow_distance} we show another example of gradient flow of the opposite Wasserstein distance. As in the examples of Figure \ref{fig:extrapolation}, the measures $\nu_0$ and $\nu_{\mathrm{gf}}(1)=\nu_1$ are obtained as optimal quantization in the $W_2$ sense of two uniformly distributed densities. The flow is approximated on the time interval $[1,4]$ using the JKO scheme \eqref{eq:jko_extra} and the entropic scheme presented in this section, for $h=0.1$ and $\varepsilon=10^{-3}$. The evolution is compared with the metric extrapolation $\nu_t$.

\begin{figure}[t]
    \centering
    \setlength{\tabcolsep}{0pt}
    \hspace{-1em}
    \begin{tabularx}{0.9\textwidth}{>{\centering\arraybackslash}X>{\centering\arraybackslash}X>{\centering\arraybackslash}X>{\centering\arraybackslash}X>{\centering\arraybackslash}X>{\centering\arraybackslash}X>{\centering\arraybackslash}X>{\centering\arraybackslash}X}
        \includegraphics[trim={3cm 1cm 3cm 1cm},clip,width=0.1\textwidth]{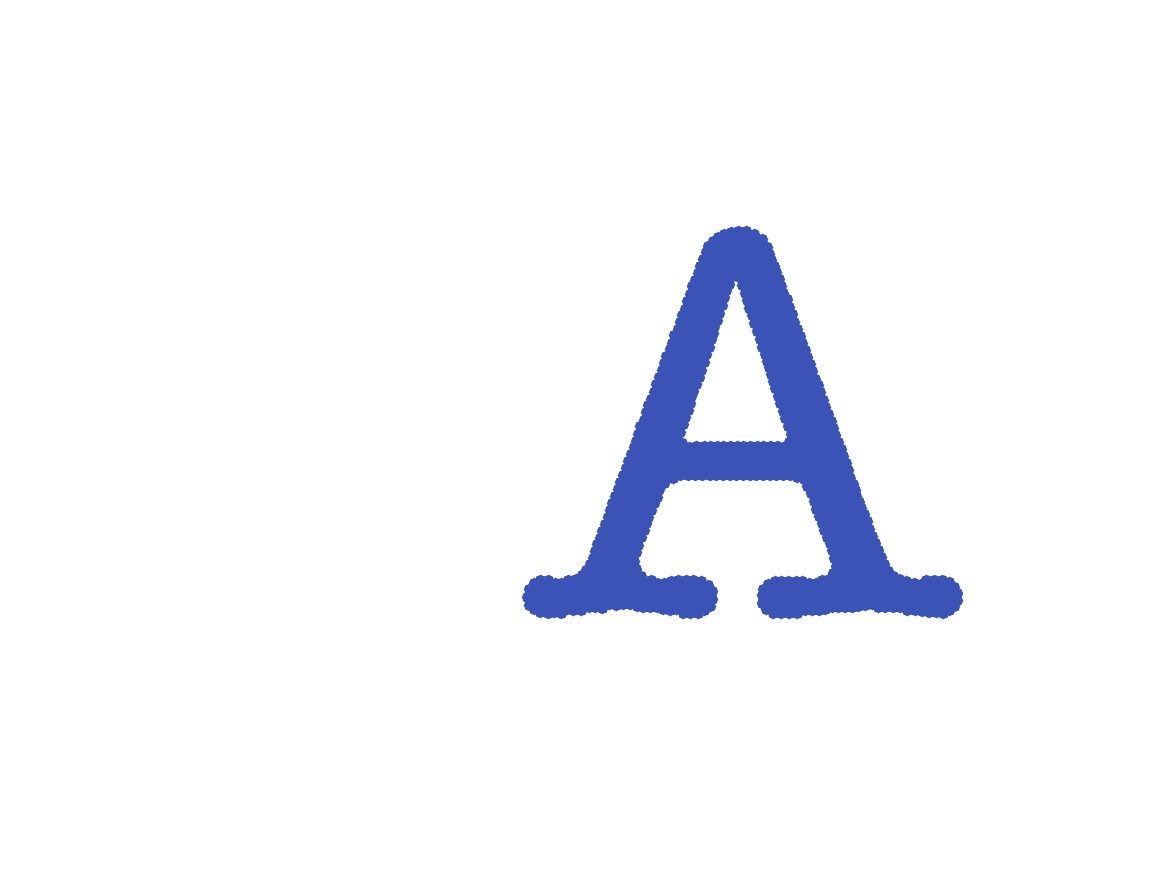} & \includegraphics[trim={3cm 1cm 3cm 1cm},clip,width=0.1\textwidth]{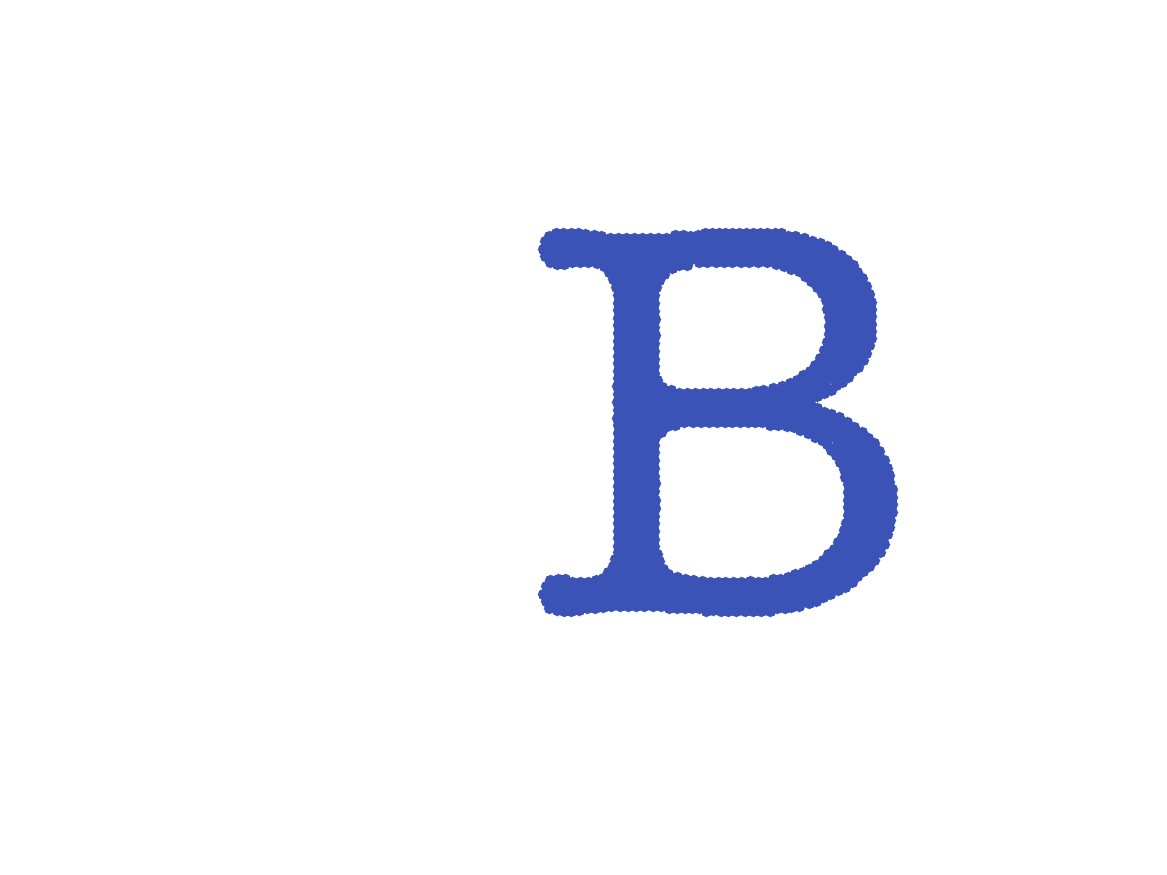} &
        \includegraphics[trim={3cm 1cm 3cm 1cm},clip,width=0.1\textwidth]{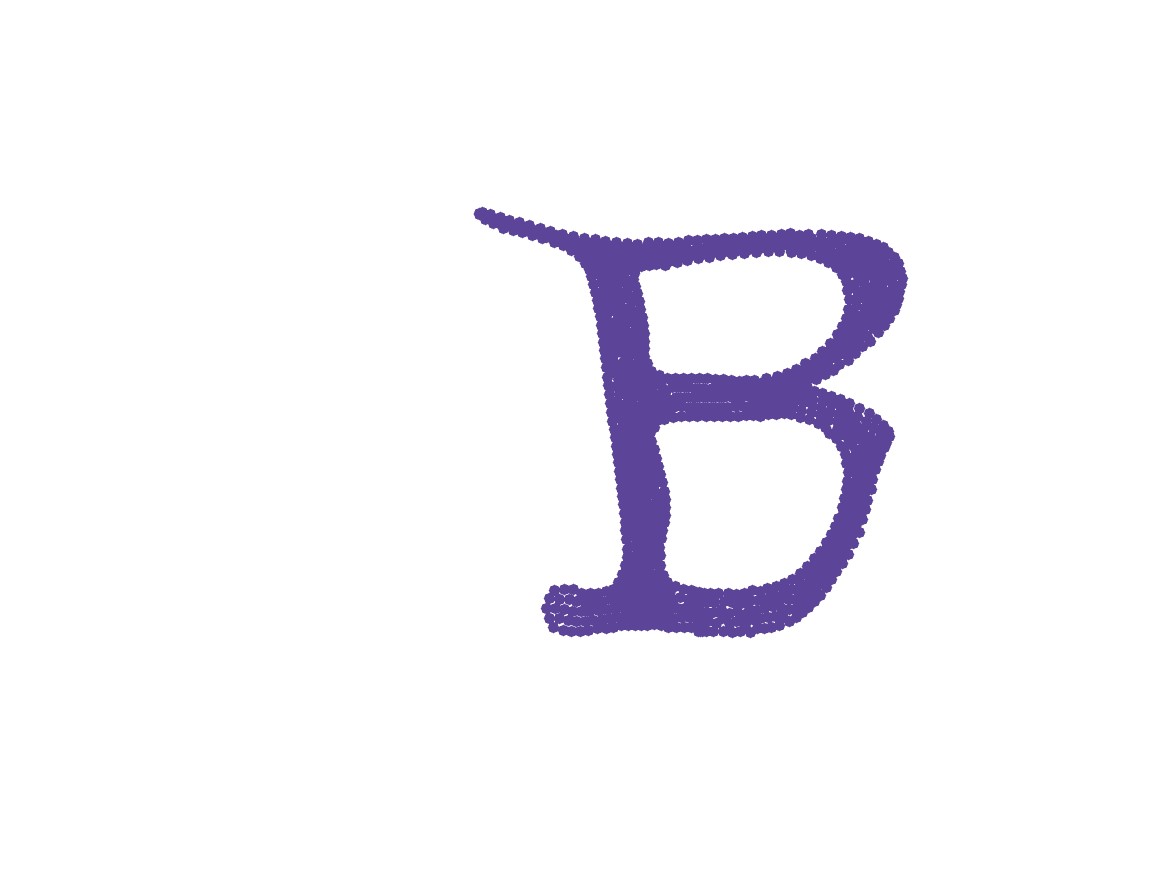} &
        \includegraphics[trim={3cm 1cm 3cm 1cm},clip,width=0.1\textwidth]{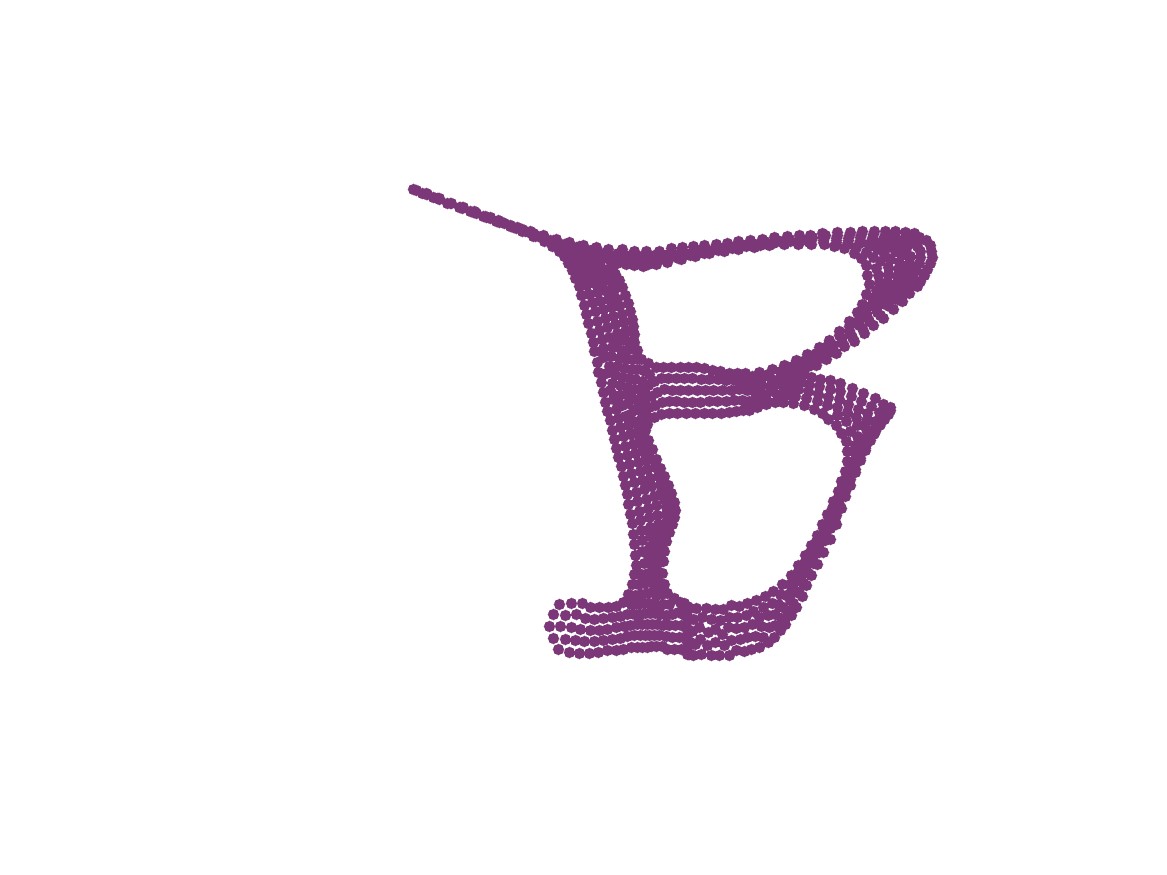} &
        \includegraphics[trim={3cm 1cm 3cm 1cm},clip,width=0.1\textwidth]{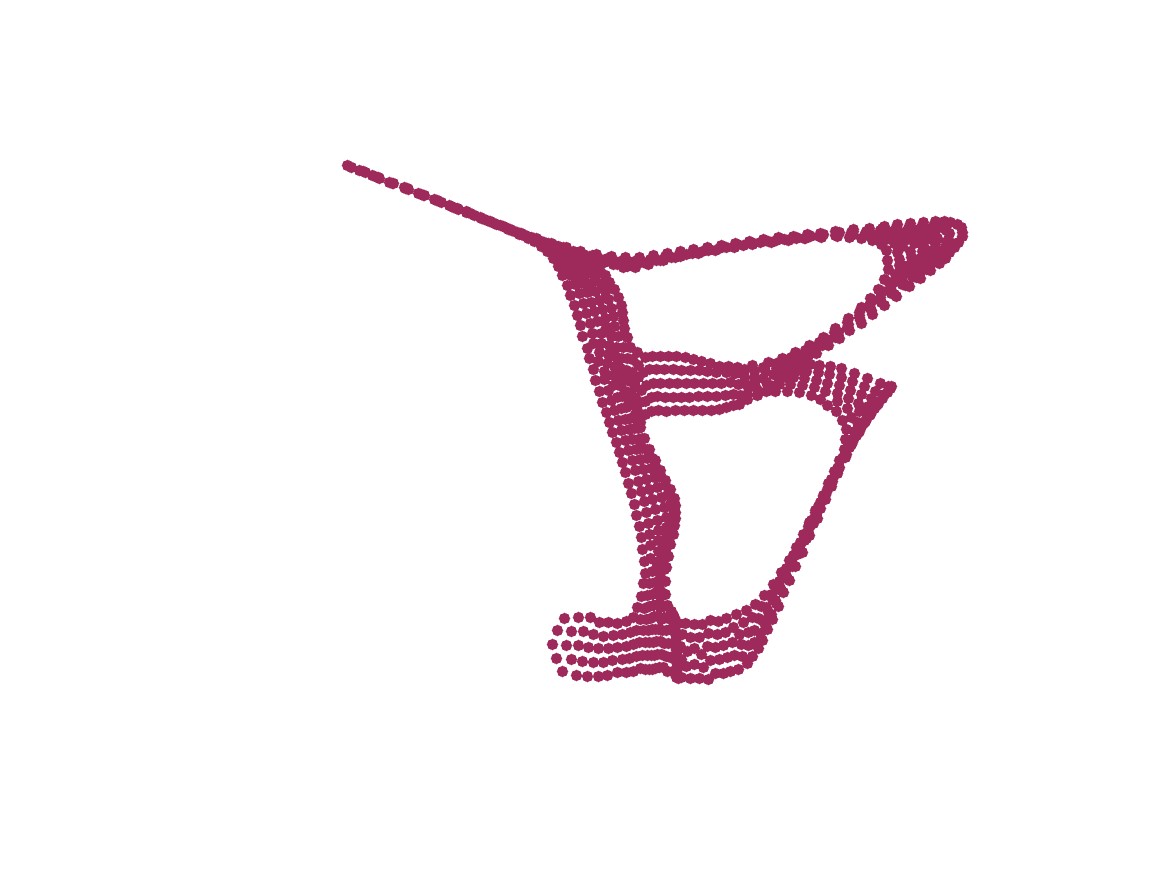} &
        \includegraphics[trim={3cm 1cm 3cm 1cm},clip,width=0.1\textwidth]{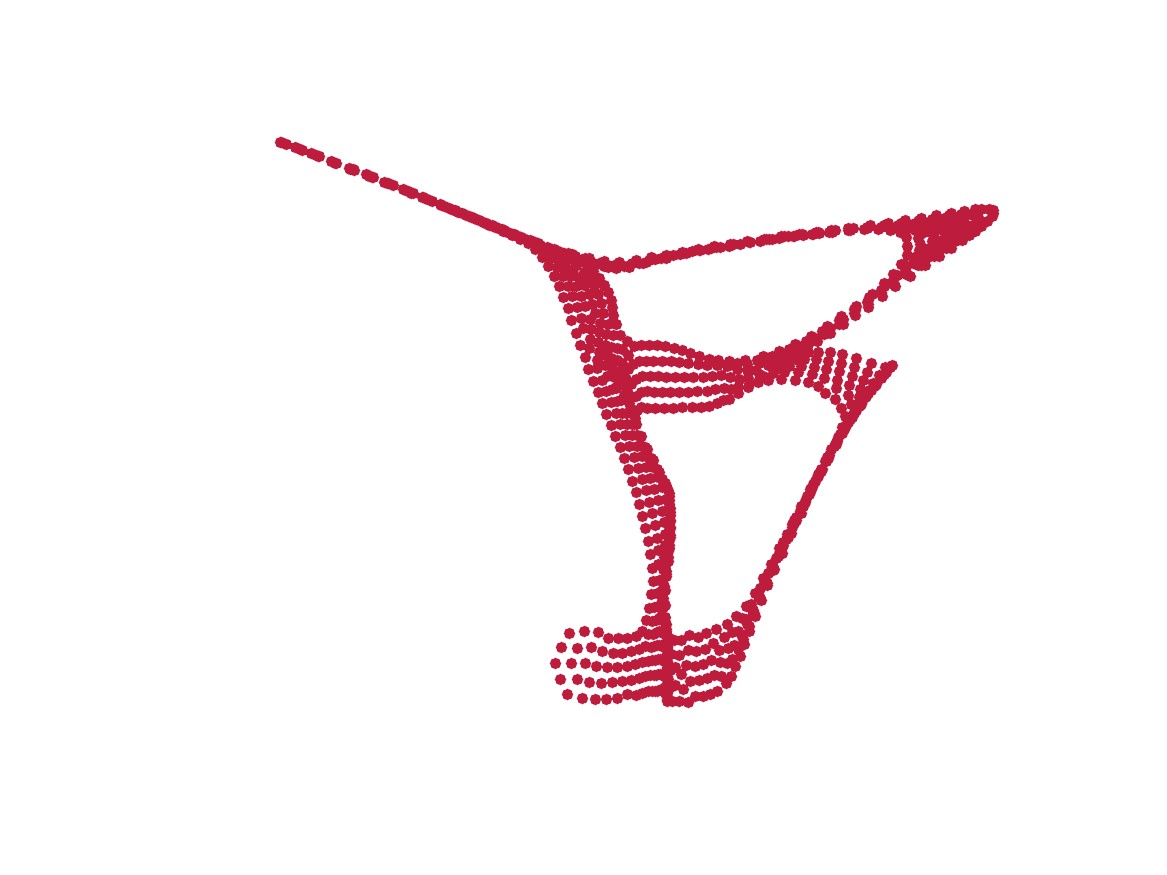} &
        \includegraphics[trim={3cm 1cm 3cm 1cm},clip,width=0.1\textwidth]{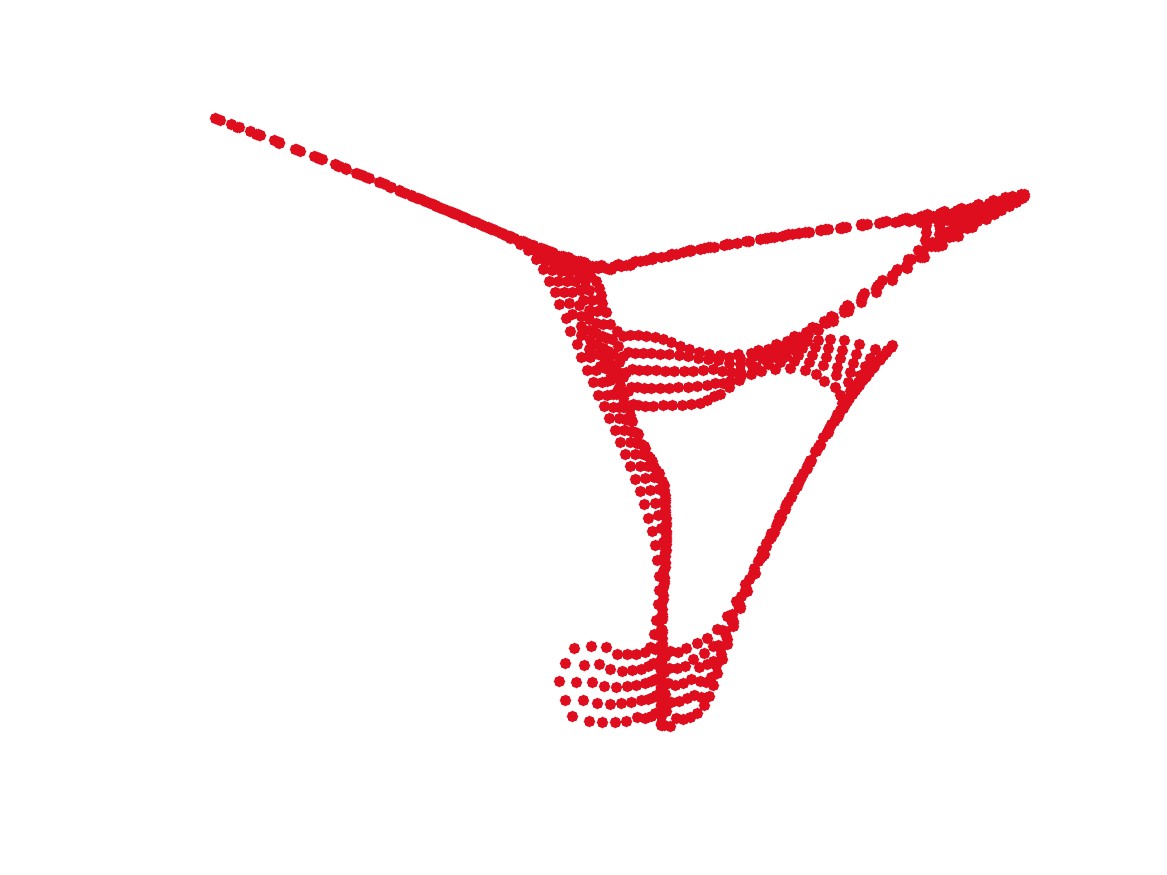} &
        \includegraphics[trim={3cm 1cm 3cm 1cm},clip,width=0.1\textwidth]{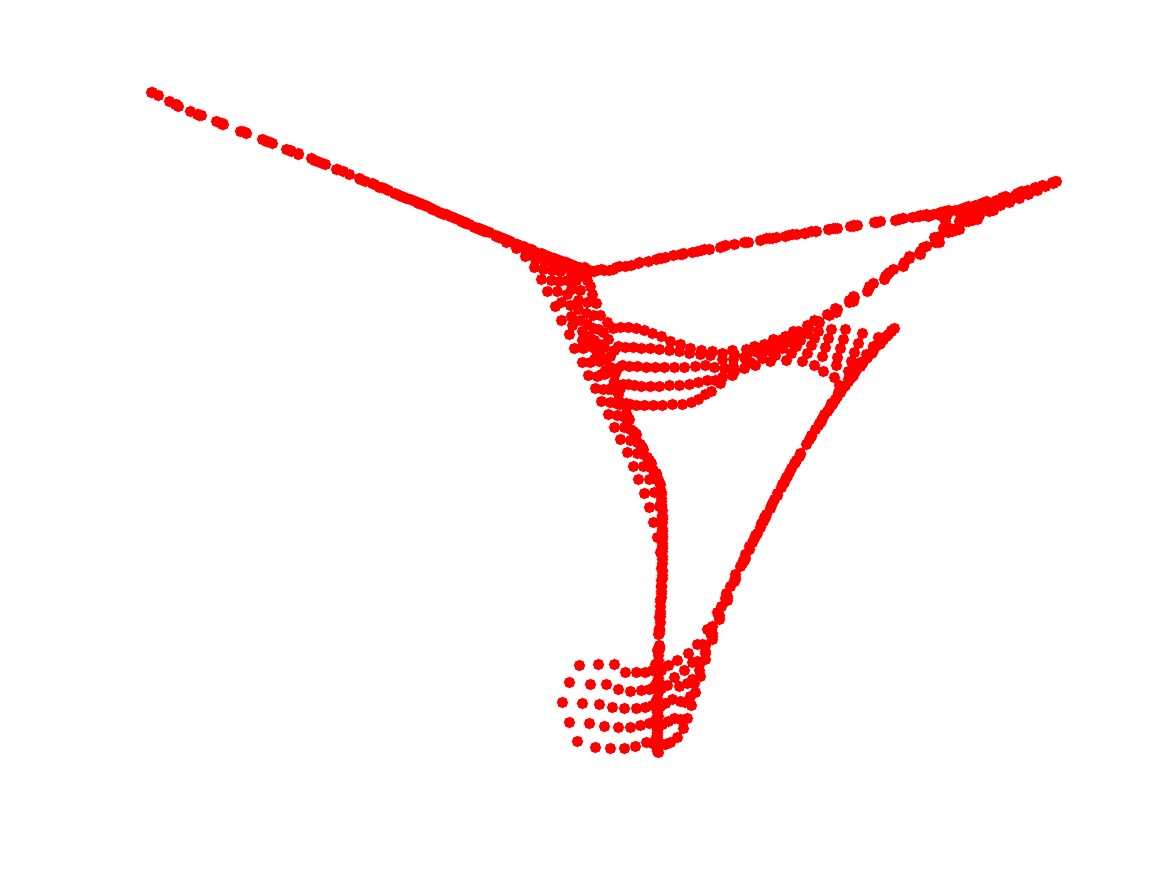} \\
        \includegraphics[trim={3cm 1cm 3cm 1cm},clip,width=0.1\textwidth]{lettersAB_initial_1.jpg} & \includegraphics[trim={3cm 1cm 3cm 1cm},clip,width=0.1\textwidth]{lettersAB_initial_2.jpg} &
        \includegraphics[trim={3cm 1cm 3cm 1cm},clip,width=0.1\textwidth]{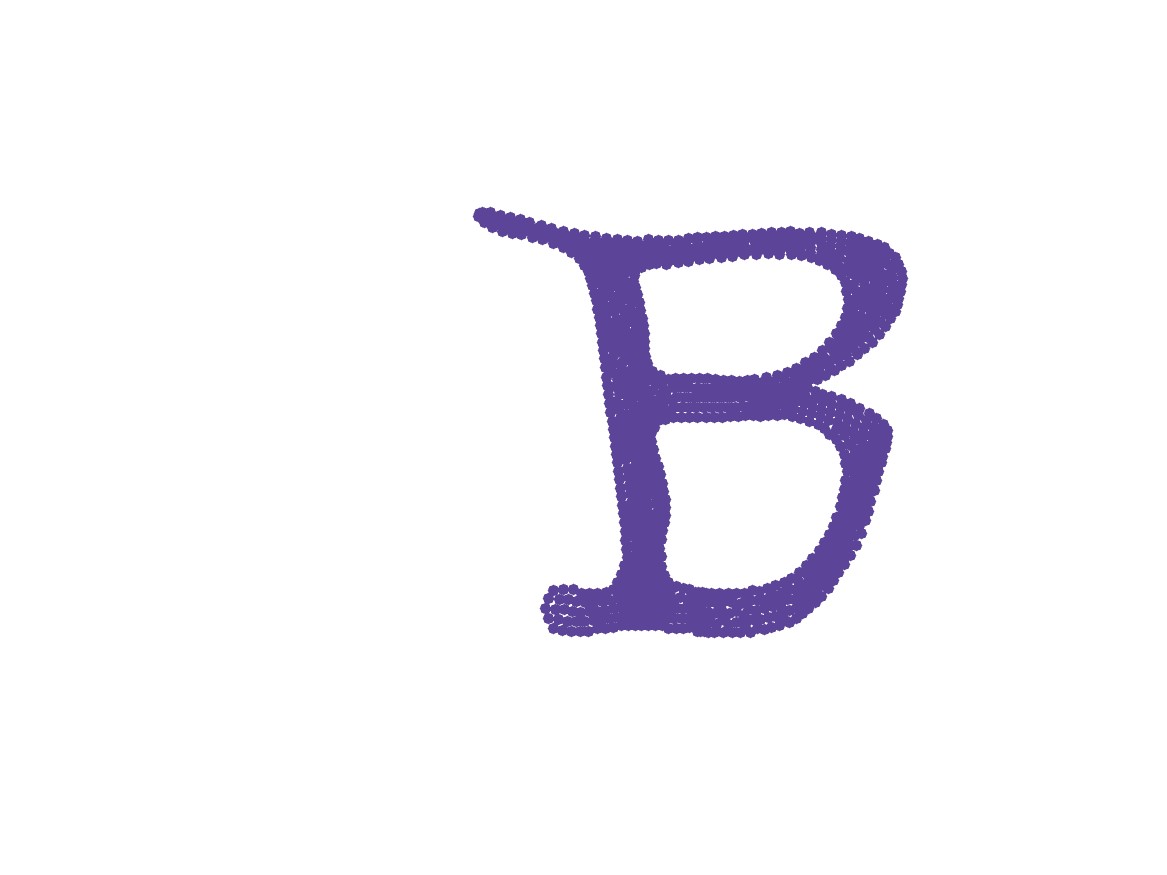} &
        \includegraphics[trim={3cm 1cm 3cm 1cm},clip,width=0.1\textwidth]{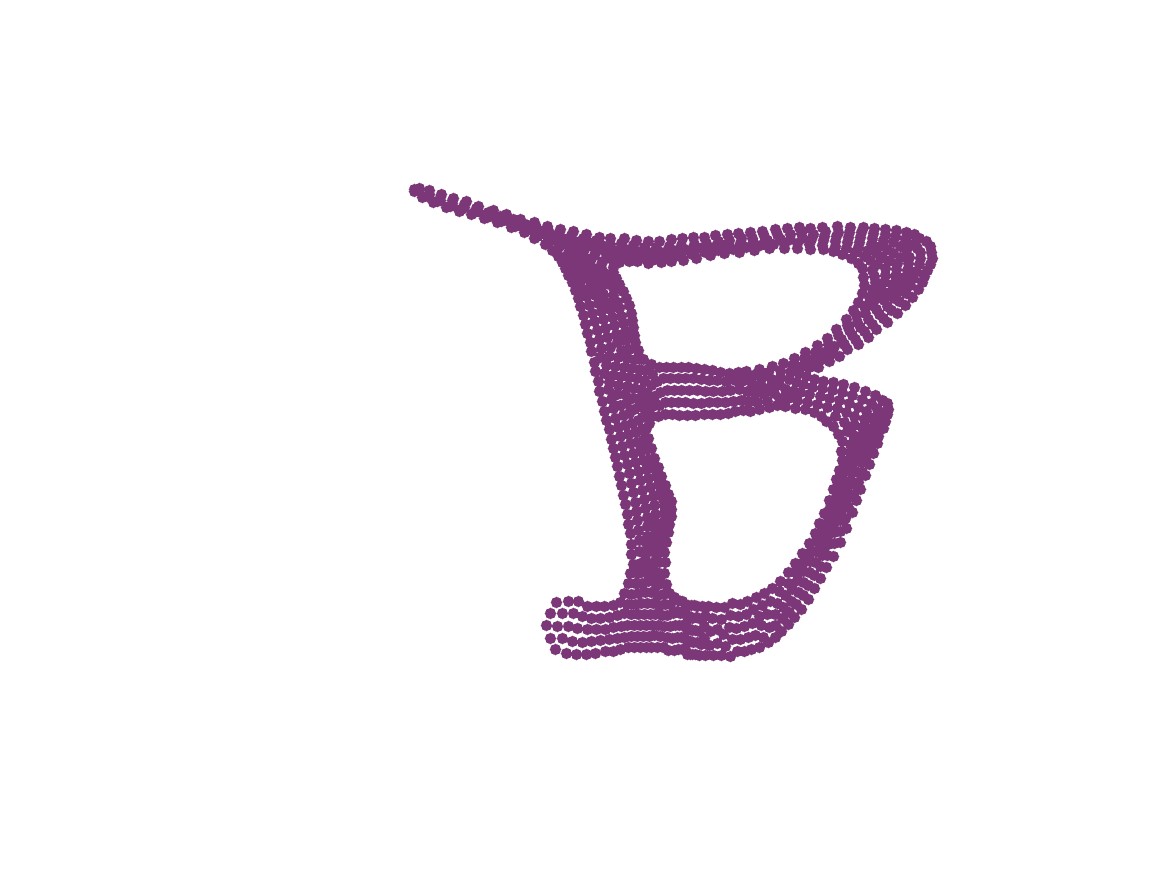} &
        \includegraphics[trim={3cm 1cm 3cm 1cm},clip,width=0.1\textwidth]{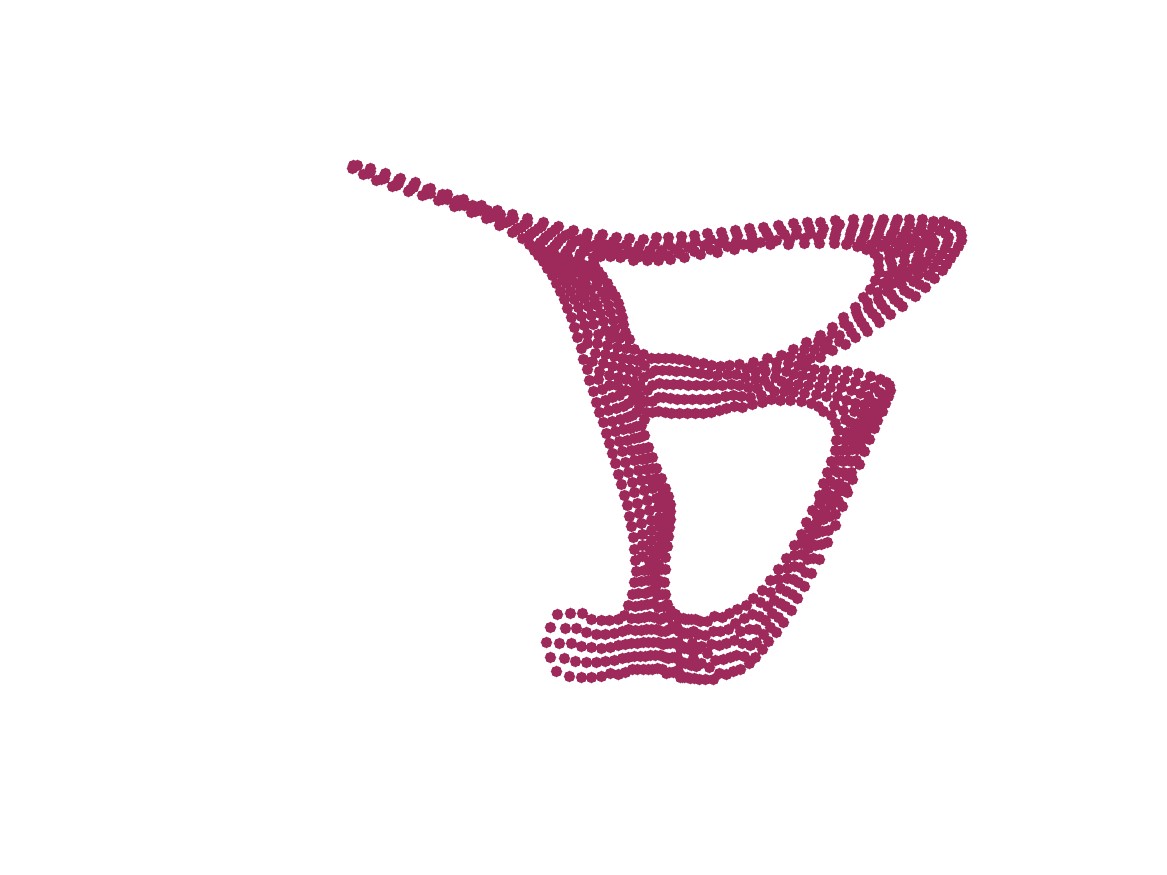} &
        \includegraphics[trim={3cm 1cm 3cm 1cm},clip,width=0.1\textwidth]{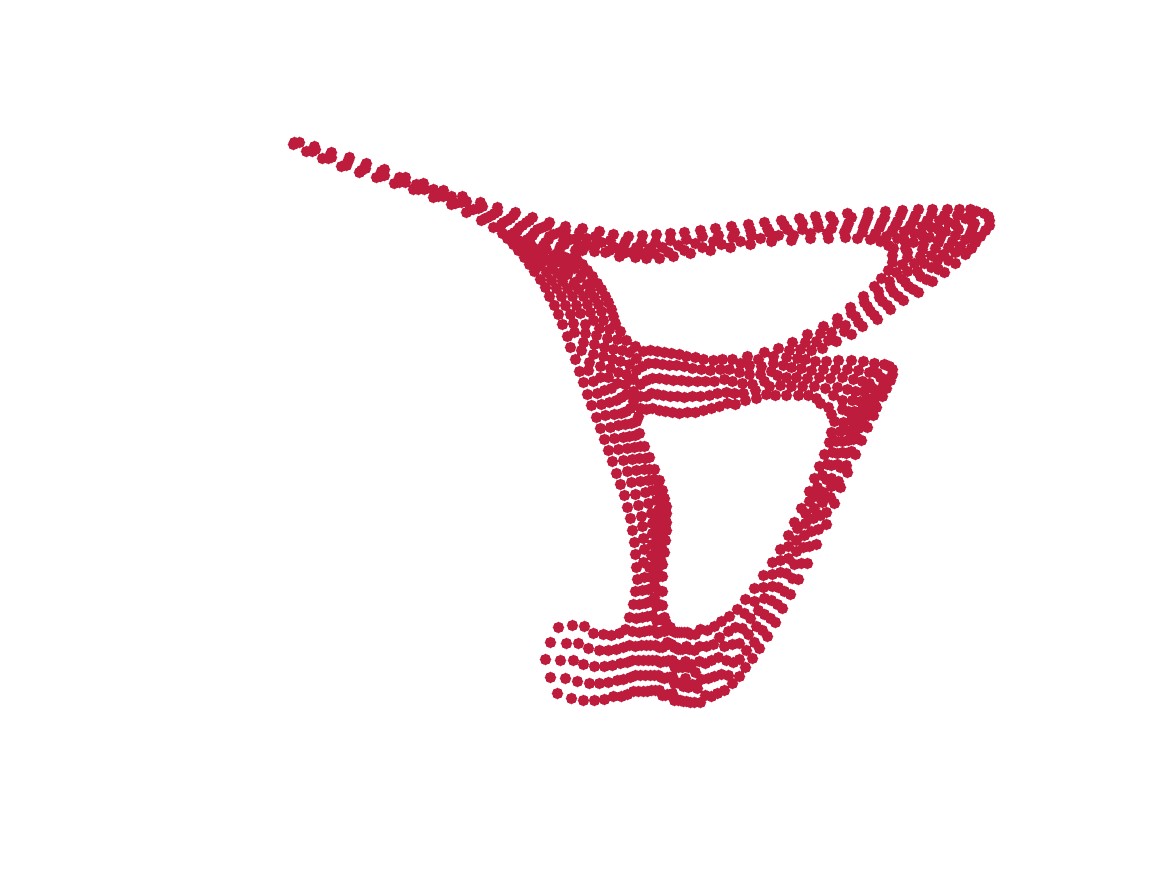} &
        \includegraphics[trim={3cm 1cm 3cm 1cm},clip,width=0.1\textwidth]{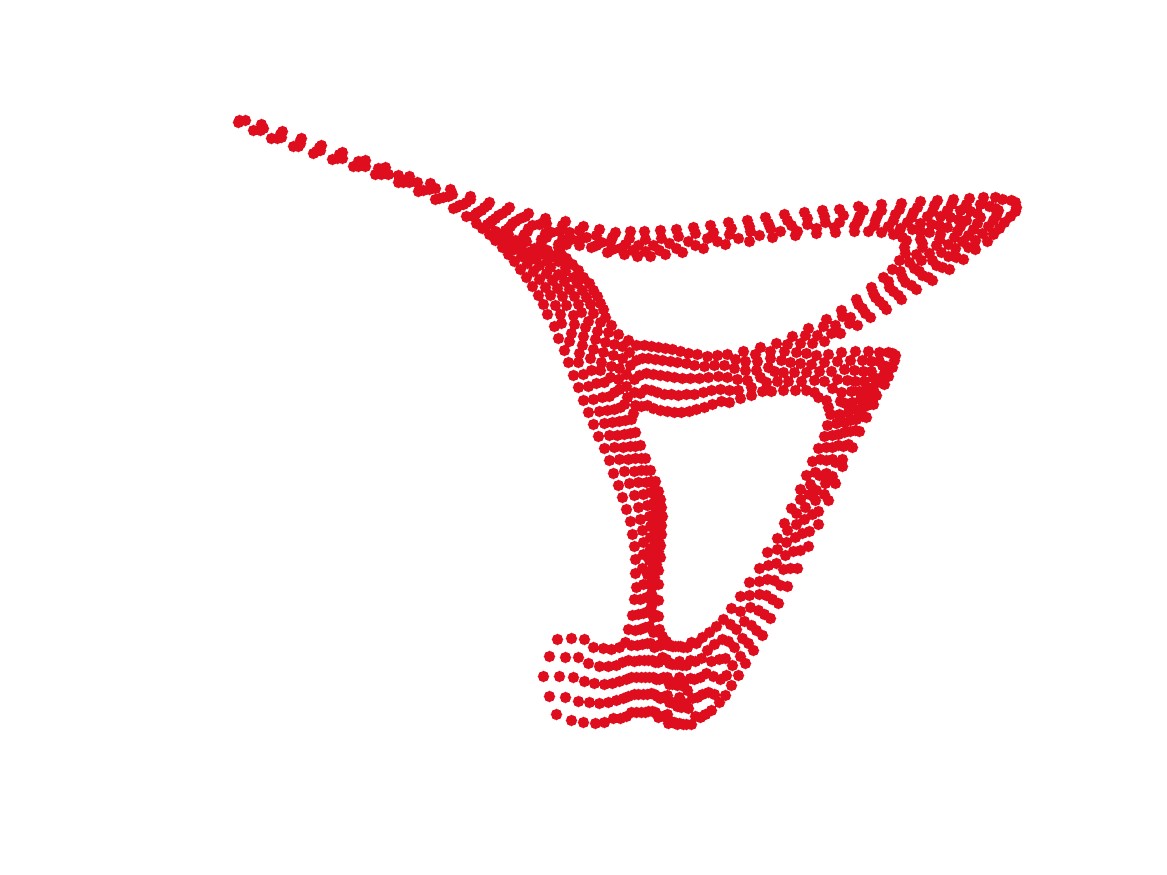} &
        \includegraphics[trim={3cm 1cm 3cm 1cm},clip,width=0.1\textwidth]{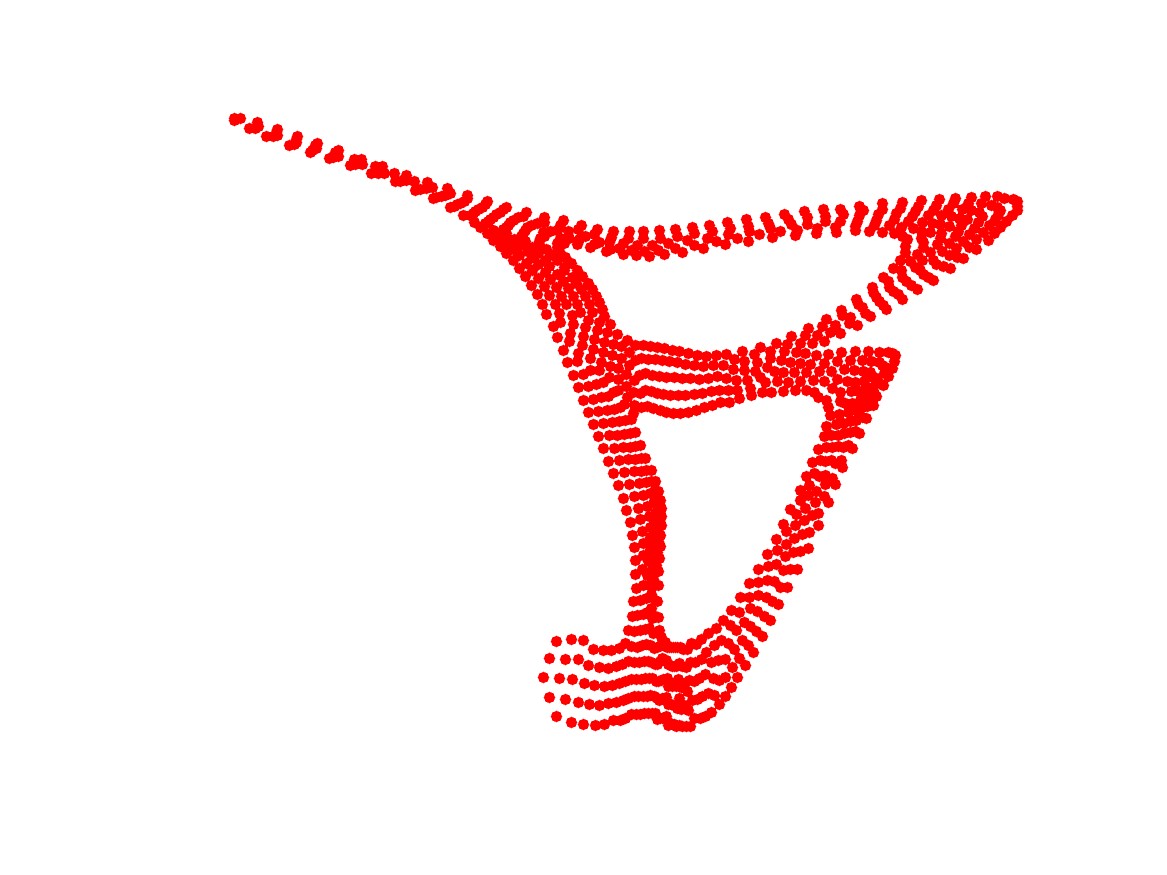} \\
    \end{tabularx}
    \caption{Example of gradient flow of the opposite Wasserstein distance (top row) from $\nu_0$ (first column from the left) with initial condition $\nu_1$ (second column from the left) and comparison with the metric extrapolation (bottom row). Time grows from left to right, from $t=1$ (blue) to $t=4$ (red).}
    \label{fig:flow_distance}
\end{figure}

\section*{Acknowledgements}

This work  was partly supported by the Labex CEMPI (ANR-11-LABX-0007-01).
The work of GT is supported by the Bézout Labex (New Monge Problems), funded by ANR, reference ANR-10-LABX-58. The authors would like to thank  Yann Brenier, Guillaume Carlier and Quentin M\'erigot,  for fruitful discussions and suggestions on the topic.

\bibliographystyle{plain}
\bibliography{refs}

\end{document}

%% file: Preambolo.tex
\usepackage[english]{babel}
\usepackage{amsmath, amsthm, amsfonts,stmaryrd,amssymb}
\numberwithin{equation}{section}
\usepackage{bbm}
\usepackage{accents}
\usepackage{tikz}
\usetikzlibrary{calc,intersections,through,backgrounds}

\usepackage{tikz-cd}
\usepackage{mathrsfs}                      % Nice \mathscr fonts
\usepackage{mathtools}                     % \mathrlap
\usepackage{booktabs}
\usepackage{todonotes}
\usepackage[all]{xy}
\usepackage[utf8]{inputenc}
\usepackage{mathtools}                     % Extendable arrows with text
\usepackage{framed}
\usepackage{graphicx}   
\usepackage[margin=3cm]{geometry}
\usepackage{hyperref} % crossreferences
\usepackage{xcolor}
\hypersetup{
    colorlinks=true,
    linkcolor=blue,  
    urlcolor=cyan,
    pdftitle={},
    pdfauthor={},
}
\usepackage{tabularx}

\usepackage{subfig}
\usepackage{pgfplots}
\pgfplotsset{compat=1.18}
\usepackage{cleveref}
\usepackage{thmtools}
\usepackage{stmaryrd}
\usepackage{enumerate}
\usepackage{dsfont}
\usepackage{cancel}
\usepackage[abs]{overpic}

\theoremstyle{plain}
\newtheorem{theorem}{Theorem}[section]

\newtheorem{remark}[theorem]{Remark}
\newtheorem{lemma}[theorem]{Lemma}

\newtheorem{proposition}[theorem]{Proposition}

\theoremstyle{definition}

\def\ed{\mathrm{d}}

\def\R{{\mathbb R}}

\def\exp{\operatorname{exp}}

\let\mc=\mathcal

\let\mr=\mathrm

\def\x{\mathbf{x}}
\def\y{\mathbf{y}}

\def\Pc{\mathcal{P}}

\def\Rd{\mathbb{R}^d}
\def\Id{\text{Id}}

\usepackage[ruled,vlined]{algorithm2e}
 
\DeclareMathOperator*{\argmin}{argmin}

\def\ed{\mathrm{d}}

\usepackage{array,multirow,makecell}
\setcellgapes{1pt}
\makegapedcells
\newcolumntype{R}[1]{>{\raggedleft\arraybackslash }b{#1}}
\newcolumntype{L}[1]{>{\raggedright\arraybackslash }b{#1}}
\newcolumntype{C}[1]{>{\centering\arraybackslash }b{#1}}
\usepackage{pifont}

\def\conv{\text{\normalfont conv}}

%% file: main_revised.bbl
\begin{thebibliography}{10}

\bibitem{agueh2011barycenters}
Martial Agueh and Guillaume Carlier.
\newblock Barycenters in the {Wasserstein} space.
\newblock {\em SIAM Journal on Mathematical Analysis}, 43(2):904--924, 2011.

\bibitem{ambrosio2008gradient}
Luigi Ambrosio, Nicola Gigli, and Giuseppe Savar{\'e}.
\newblock {\em Gradient flows: in metric spaces and in the space of probability
  measures}.
\newblock Springer Science \& Business Media, 2008.

\bibitem{bauschke2017convex}
Combettes P.~L. Bauschke, H.~H.
\newblock {\em {Convex Analysis and Monotone Operators Theory in {H}ilbert
  Spaces}}.
\newblock CMS Books in Mathematics. Springer-Verlag, second edition, 2017.

\bibitem{benamou2019second}
Jean-David Benamou, Thomas~O Gallou{\"e}t, and Fran{\c{c}}ois-Xavier Vialard.
\newblock Second-order models for optimal transport and cubic splines on the
  {Wasserstein} space.
\newblock {\em Foundations of Computational Mathematics}, 19:1113--1143, 2019.

\bibitem{bertucci2024approximation}
Charles Bertucci and Pierre~Louis Lions.
\newblock An approximation of the squared {W}asserstein distance and an
  application to {H}amilton-{J}acobi equations.
\newblock {\em arXiv preprint arXiv:2409.11793}, 2024.

\bibitem{brenier2018initial}
Yann Brenier.
\newblock The initial value problem for the euler equations of incompressible
  fluids viewed as a concave maximization problem.
\newblock {\em Communications in Mathematical Physics}, 364:579--605, 2018.

\bibitem{brenier2020examples}
Yann Brenier.
\newblock Examples of hidden convexity in nonlinear {PDEs}.
\newblock 2020.

\bibitem{brenier1998one}
Yann Brenier and Emmanuel Grenier.
\newblock Sticky particles and scalar conservation laws.
\newblock {\em SIAM Journal on Numerical Analysis}, 35(6):2317--2328, 1998.

\bibitem{bubeck2015convex}
S{\'e}bastien Bubeck et~al.
\newblock Convex optimization: Algorithms and complexity.
\newblock {\em Foundations and Trends{\textregistered} in Machine Learning},
  8(3-4):231--357, 2015.

\bibitem{carlier2008Toland}
Guillaume Carlier.
\newblock Remarks on {Toland}'s duality, convexity constraint and optimal
  transport.
\newblock {\em Pacific Journal of Optimization}, 4(3):423--432, 2008.

\bibitem{carlier2023sista}
Guillaume Carlier, Arnaud Dupuy, Alfred Galichon, and Yifei Sun.
\newblock Sista: learning optimal transport costs under sparsity constraints.
\newblock {\em Communications on Pure and Applied Mathematics},
  76(9):1659--1677, 2023.

\bibitem{cazelles2021novel}
Elsa Cazelles, Felipe Tobar, and Joaquin Fontbona.
\newblock A novel notion of barycenter for probability distributions based on
  optimal weak mass transport.
\newblock {\em Advances in Neural Information Processing Systems},
  34:13575--13586, 2021.

\bibitem{dalery2023nonlinear}
Maxime Dalery, Genevi{\`e}ve Dusson, Virginie Ehrlacher, and Alexei Lozinski.
\newblock Nonlinear reduced basis using mixture {Wasserstein} barycenters:
  application to an eigenvalue problem inspired from quantum chemistry.
\newblock {\em arXiv preprint arXiv:2307.15423}, 2023.

\bibitem{fan2021conditional}
Jianing Fan and Hans-Georg M{\"u}ller.
\newblock Conditional {Wasserstein} barycenters and interpolation/extrapolation
  of distributions.
\newblock {\em arXiv preprint arXiv:2107.09218}, 2021.

\bibitem{figalli2011multidimensional}
Alessio Figalli, Young-Heon Kim, and Robert~J McCann.
\newblock When is multidimensional screening a convex program?
\newblock {\em Journal of Economic Theory}, 146(2):454--478, 2011.

\bibitem{gallouet2024geodesic}
Thomas Gallou{\"e}t, Andrea Natale, and Gabriele Todeschi.
\newblock From geodesic extrapolation to a variational {BDF2} scheme for
  {W}asserstein gradient flows.
\newblock {\em Mathematics of Computation}, 2024.

\bibitem{gangbo1994elementary}
Wilfrid Gangbo.
\newblock An elementary proof of the polar factorization of vector-valued
  functions.
\newblock {\em Archive for rational mechanics and analysis}, 128:381--399,
  1994.

\bibitem{gozlan2020mixture}
Nathael Gozlan and Nicolas Juillet.
\newblock On a mixture of {Brenier} and {Strassen} theorems.
\newblock {\em Proceedings of the London Mathematical Society},
  120(3):434--463, 2020.

\bibitem{gozlan2017kantorovich}
Nathael Gozlan, Cyril Roberto, Paul-Marie Samson, and Prasad Tetali.
\newblock Kantorovich duality for general transport costs and applications.
\newblock {\em Journal of Functional Analysis}, 273(11):3327--3405, 2017.

\bibitem{jordan1998variational}
Richard Jordan, David Kinderlehrer, and Felix Otto.
\newblock The variational formulation of the {Fokker--Planck} equation.
\newblock {\em SIAM journal on mathematical analysis}, 29(1):1--17, 1998.

\bibitem{kim2024statistical}
Jakwang Kim, Young-Heon Kim, Yuanlong Ruan, and Andrew Warren.
\newblock Statistical inference of convex order by wasserstein projection.
\newblock {\em arXiv preprint arXiv:2406.02840}, 2024.

\bibitem{kim2010continuity}
Young-Heon Kim and Robert~J McCann.
\newblock Continuity, curvature, and the general covariance of optimal
  transportation.
\newblock {\em Journal of the European Mathematical Society}, 12(4):1009--1040,
  2010.

\bibitem{Legendre2017VIM}
Guillaume Legendre and Gabriel Turinici.
\newblock Second-order in time schemes for gradient flows in {Wasserstein} and
  geodesic metric spaces.
\newblock {\em Comptes Rendus Mathematique}, 355:345--353, 03 2017.

\bibitem{leger2024nonnegative}
Flavien L{\'e}ger, Gabriele Todeschi, and Fran{\c{c}}ois-Xavier Vialard.
\newblock Nonnegative cross-curvature in infinite dimensions: synthetic
  definition and spaces of measures.
\newblock {\em arXiv preprint arXiv:2409.18112}, 2024.

\bibitem{Matthes2019bdf2}
Daniel Matthes and Simon Plazotta.
\newblock A variational formulation of the {BDF2} method for metric gradient
  flows.
\newblock {\em ESAIM: Mathematical Modelling and Numerical Analysis},
  53(1):145--172, 2019.

\bibitem{mccann1997convexity}
Robert~J McCann.
\newblock A convexity principle for interacting gases.
\newblock {\em Advances in mathematics}, 128(1):153--179, 1997.

\bibitem{natile2009wasserstein}
Luca Natile and Giuseppe Savar{\'e}.
\newblock A {Wasserstein} approach to the one-dimensional sticky particle
  system.
\newblock {\em SIAM journal on mathematical analysis}, 41(4):1340--1365, 2009.

\bibitem{nesterov2018lectures}
Yurii Nesterov et~al.
\newblock {\em Lectures on convex optimization}, volume 137.
\newblock Springer.

\bibitem{paty2022algorithms}
Fran{\c{c}}ois-Pierre Paty, Philippe Chon{\'e}, and Francis Kramarz.
\newblock Algorithms for weak optimal transport with an application to
  economics.
\newblock {\em arXiv preprint arXiv:2205.09825}, 2022.

\bibitem{paty2020regularity}
Fran{\c{c}}ois-Pierre Paty, Alexandre d’Aspremont, and Marco Cuturi.
\newblock Regularity as regularization: Smooth and strongly convex {B}renier
  potentials in optimal transport.
\newblock In {\em International Conference on Artificial Intelligence and
  Statistics}, pages 1222--1232. PMLR, 2020.

\bibitem{peyre2020computational}
Gabriel Peyré and Marco Cuturi.
\newblock Computational optimal transport, 2020.

\bibitem{santambrogio2015optimal}
Filippo Santambrogio.
\newblock Optimal transport for applied mathematicians.
\newblock {\em Birk{\"a}user, NY}, pages 99--102, 2015.

\bibitem{strassen1965existence}
Volker Strassen.
\newblock The existence of probability measures with given marginals.
\newblock {\em The Annals of Mathematical Statistics}, 36(2):423--439, 1965.

\bibitem{toland1978duality}
John~F Toland.
\newblock Duality in nonconvex optimization.
\newblock {\em Journal of Mathematical Analysis and Applications},
  66(2):399--415, 1978.

\bibitem{toland1979duality}
John~F Toland.
\newblock A duality principle for non-convex optimisation and the calculus of
  variations.
\newblock {\em Archive for Rational Mechanics and Analysis}, 71:1432--0673,
  1978.

\bibitem{tornabene2024generalized}
Francesco Tornabene, Marco Veneroni, and Giuseppe Savar{\'e}.
\newblock Generalized {W}asserstein barycenters.
\newblock {\em arXiv preprint arXiv:2411.06838}, 2024.

\bibitem{vacher2023semi}
Adrien Vacher and Fran{\c{c}}ois-Xavier Vialard.
\newblock Semi-dual unbalanced quadratic optimal transport: fast statistical
  rates and convergent algorithm.
\newblock In {\em International Conference on Machine Learning}, pages
  34734--34758. PMLR, 2023.

\bibitem{villani2009optimal}
C{\'e}dric Villani.
\newblock {\em Optimal transport: old and new}, volume 338.
\newblock Springer, 2009.

\bibitem{vorotnikov2022partial}
Dmitry Vorotnikov.
\newblock Partial differential equations with quadratic nonlinearities viewed
  as matrix-valued optimal ballistic transport problems.
\newblock {\em Archive for Rational Mechanics and Analysis}, 243(3):1653--1698,
  2022.

\end{thebibliography}
